\documentclass[11pt,oneside]{amsart}
\usepackage[mathscr]{eucal}
\usepackage{amssymb, amsmath, array, amscd}
\usepackage[dvips]{graphicx}
\newtheorem{theorem}{Theorem}[section]
\newtheorem{corollary}{Corollary}[section]
\newtheorem{proposition}{Proposition}[section]
\newtheorem{lemma}{Lemma}[section]
\newtheorem{definition}{Definition}[section]
\newtheorem{remark}{Remark}[section]
\newtheorem{example}{Example}[section]

\begin{document}

\title[Uniformisation of foliations by curves]
{UNIFORMISATION OF FOLIATIONS BY CURVES}

\author{Marco Brunella}

\address{Marco Brunella, IMB - CNRS UMR 5584, 9 Avenue Savary,
21078 Dijon, France}

\begin{abstract}
These lecture notes provide a full discussion of certain analytic
aspects of the uniformisation theory of foliations by curves on
compact K\"ahler manifolds, with emphasis on their consequences on
positivity properties of the corresponding canonical bundles.
\end{abstract}

\maketitle

\tableofcontents

\section{Foliations by curves and their uniformisation}

Let $X$ be a complex manifold. A {\bf foliation by curves} $\mathcal
F$ on $X$ is defined by a holomorphic line bundle $T_\mathcal F$ on
$X$ and a holomorphic linear morphism
$$\tau_\mathcal F : T_\mathcal F \rightarrow TX$$
which is injective outside an analytic subset $Sing({\mathcal
F})\subset X$ of codimension at least 2, called the {\bf singular
set} of the foliation. Equivalently, we have an open covering $\{
U_j\}$ of $X$ and a collection of holomorphic vector fields
$v_j\in\Theta (U_j)$, with zero set of codimension at least 2, such
that
$$v_j = g_{jk}v_k \qquad {\rm on}\quad U_j\cap U_k ,$$
where $g_{jk}\in{\mathcal O}^*(U_j\cap U_k)$ is a multiplicative
cocycle defining the dual bundle $T_\mathcal F^* = K_\mathcal F$,
called the {\bf canonical bundle} of $\mathcal F$.

These vector fields can be locally integrated, and by the relations
above these local integral curves can be glued together (without
respecting the time parametrization), giving rise to the {\bf
leaves} of the foliation $\mathcal F$.

By the classical Uniformisation Theorem, the universal covering of
each leaf is either the unit disc $\mathbb D$ (hyperbolic leaf) or
the affine line $\mathbb C$ (parabolic leaf) or the projective line
$\mathbb P$ (rational leaf).

In these notes we shall assume that the ambient manifold $X$ is a
compact connected K\"ahler manifold, and we will be concerned with
the following problem: how the universal covering $\widetilde{L_p}$
of the leaf $L_p$ through the point $p$ depends on $p$ ? For
instance, we may first of all ask about the structure of the subset
of $X$ formed by those points through which the leaf is hyperbolic,
resp. parabolic, resp. rational: is the set of hyperbolic leaves
open in $X$? Is the set of parabolic leaves analytic? But, even if
all the leaves are, say, hyperbolic, there are further basic
questions: the uniformising map of every leaf is almost unique
(unique modulo automorphisms of the disc), and after some
normalization (to get uniqueness) we may ask about the way in which
the uniformising map of $L_p$ depends on the point $p$.
Equivalently, we may put on every leaf its Poincar\'e metric, and we
may ask about the way in which this leafwise metric varies in the
directions transverse to the foliation.

Our main result will be that these universal coverings of leaves can
be glued together in a vaguely ``holomorphically convex'' way. That
is, the {\it leafwise} universal covering of the foliated manifold
$(X,{\mathcal F})$ can be defined and it has a sort of
``holomorphically convex'' structure \cite{Br2} \cite{Br3}. This was
inspired by a seminal work of Il'yashenko \cite{Il1} \cite{Il2}, who
proved a similar result when $X$ is a Stein manifold instead of a
compact K\"ahler one. Related ideas can also be found in Suzuki's
paper \cite{Suz}, still in the Stein case. Another source of
inspiration was Shafarevich conjecture on the holomorphic convexity
of universal coverings of projective (or compact K\"ahler) manifolds
\cite{Nap}.

This main result will allow us to apply results by Nishino
\cite{Nis} and Yamaguchi \cite{Ya1} \cite{Ya2} \cite{Ya3} \cite{Kiz}
concerning the transverse variation of the leafwise Poincar\'e
metric and other analytic invariants. As a consequence of this, for
instance, we shall obtain that if the foliation has at least one
hyperbolic leaf, then: (1) there are no rational leaves; (2)
parabolic leaves fill a subset of $X$ which is {\it complete
pluripolar}, i.e. locally given by the poles of a plurisubharmonic
function. In other words, the set of hyperbolic leaves of $\mathcal
F$ is either empty or potential-theoretically full in $X$.

These results are related also to positivity properties of the
canonical bundle $K_{\mathcal F}$, along a tradition opened by
Arakelov \cite{Ara} \cite{BPV} in the case of algebraic fibrations
by curves and developed by Miyaoka \cite{Miy} \cite{ShB} and then
McQuillan and Bogomolov \cite{MQ1} \cite{MQ2} \cite{BMQ} \cite{Br1}
in the case of foliations on projective manifolds. From this point
of view, our final result is the following ruledness criterion for
foliations:

\begin{theorem} \label{main} \cite{Br3} \cite{Br5}
Let $X$ be a compact connected K\"ahler manifold and let $\mathcal
F$ be a foliation by curves on $X$. Suppose that the canonical
bundle $K_\mathcal F$ is {\rm not} pseudoeffective. Then through
every point $p\in X$ there exists a rational curve tangent to
$\mathcal F$.
\end{theorem}

Recall that a line bundle on a compact connected manifold is {\it
pseudoeffective} if it admits a (singular) hermitian metric with
positive curvature in the sense of currents \cite{Dem}. When $X$ is
projective the above theorem follows also from results of \cite{BMQ}
and \cite{BDP}, but with a totally different proof, untranslatable
in our K\"ahler context.

Let us now describe in more detail the content of these notes.

In Section 2 we shall recall the results by Nishino and Yamaguchi on
Stein fibrations that we shall use later, and also some of
Il'yashenko's results. In Section 3 and 4 we construct the leafwise
universal covering of $(X,{\mathcal F})$: we give an appropriate
definition of leaf $L_p$ of $\mathcal F$ through a point $p\in
X\setminus Sing({\mathcal F})$ (this requires some care, because
some leaves are allowed to pass through some singular points), and
we show that the universal coverings $\widetilde{L_p}$ can be glued
together to get a complex manifold. In Section 5 we prove that the
complex manifold so constructed enjoys some ``holomorphic
convexity'' property. This is used in Section 6 and 8, together with
Nishino and Yamaguchi results, to prove (among other things) Theorem
\ref{main} above. The parabolic case requires also an extension
theorem for certain meromorphic maps into compact K\"ahler
manifolds, which is proved in Section 7.

All this work has been developed in our previous papers \cite{Br2}
\cite{Br3} \cite{Br4} and \cite{Br5} (with few imprecisions which
will be corrected here). Further results and application can be
found in \cite{Br6} and \cite{Br7}.

\section{Some results on Stein fibrations}

\subsection{Hyperbolic fibrations}

In a series of papers, Nishino \cite{Nis} and then Yamaguchi
\cite{Ya1} \cite{Ya2} \cite{Ya3} studied the following situation. It
is given a Stein manifold $U$, of dimension $n+1$, equipped with a
holomorphic submersion $P:U\rightarrow {\mathbb D}^n$ with connected
fibers. Each fiber $P^{-1}(z)$ is thus a smooth connected curve, and
as such it has several potential theoretic invariants (Green
functions, Bergman Kernels, harmonic moduli...). One is interested
in knowing how these invariants vary with $z$, and then in using
this knowledge to obtain some information on the structure of $U$.

For our purposes, the last step in this program has been carried out
by Kizuka \cite{Kiz}, in the following form.

\begin{theorem} \label{kizuka}
\cite{Ya1} \cite{Ya3} \cite{Kiz} If $U$ is Stein, then the fiberwise
Poincar\'e metric on $U\buildrel P\over\rightarrow {\mathbb D}^n$
has a plurisubharmonic variation.
\end{theorem}

This means the following. On each fiber $P^{-1}(z)$, $z\in{\mathbb
D}^n$, we put its Poincar\'e metric, i.e. the (unique) complete
hermitian metric of curvature $-1$ if $P^{-1}(z)$ is uniformised by
${\mathbb D}$, or the identically zero ``metric'' if $P^{-1}(z)$ is
uniformised by ${\mathbb C}$ ($U$ being Stein, there are no other
possibilities). If $v$ is a holomorphic nonvanishing vector field,
defined in some open subset $V\subset U$ and tangent to the fibers
of $P$, then we can take the function on $V$
$$F = \log \| v \|_{Poin}$$
where, for every $q\in V$, $\| v(q)\|_{Poin}$ is the norm of $v(q)$
evaluated with the Poincar\'e metric on $P^{-1}(P(q))$. The
statement above means that, whatever $v$ is, the function $F$ is
{\it plurisubharmonic}, or identically $-\infty$ if all the fibers
are parabolic. Note that if we replace $v$ by $v'=g\cdot v$, with
$g$ a holomorphic nonvanishing function on $V$, then $F$ is replaced
by $F'=F+G$, where $G=\log |g|$ is pluriharmonic. A more intrinsic
way to state this property is: the fiberwise Poincar\'e metric (if
not identically zero) defines on the relative canonical bundle of
$U\buildrel P\over\rightarrow {\mathbb D}^n$ a hermitian metric
(possibly singular) whose curvature is a {\it positive current}
\cite{Dem}. Note also that the plurisubharmonicity of $F$ along the
fibers is just a restatement of the negativity of the curvature of
the Poincar\'e metric. The important fact here is the
plurisubharmonicity along the directions transverse to the fibers,
whence the {\it variation} terminology.

Remark that the poles of $F$ correspond exactly to parabolic fibers
of $U$. We therefore obtain the following dichotomy: either all the
fibers are parabolic ($F\equiv 0$), or the parabolic fibers
correspond to a complete pluripolar subset of ${\mathbb D}^n$
($F\not\equiv 0$).

The theorem above is a generalization of, and was motivated by, a
classical result of Hartogs \cite[II.5]{Ran}, asserting (in modern
language) that a domain $U$ in ${\mathbb D}^n\times{\mathbb C}$ of
the form (Hartogs tube)
$$U = \{\ (z,w)\ \vert \quad |w| < e^{-f(z)} \} ,$$
where $f:{\mathbb D}^n \rightarrow [-\infty ,+\infty )$ is an upper
semicontinuous function, is Stein if and only if $f$ is
plurisubharmonic. Indeed, in this special case the Poincar\'e metric
is easily computed, and one checks that the plurisubharmonicity of
$f$ is equivalent to the plurisubharmonic variation of the fiberwise
Poincar\'e metric. This special case suggests also that some
converse statement to Theorem \ref{kizuka} could be true.

We give the proof of Theorem \ref{kizuka} only in a particular case,
which is anyway the only case that we shall actually use.

We start with a fibration $P:U\rightarrow {\mathbb D}^n$ as above,
but without assuming $U$ Stein. We consider an open subset
$U_0\subset U$ such that:
\begin{enumerate}
\item[(i)] for every $z\in{\mathbb D}^n$, the intersection $U_0\cap
P^{-1}(z)$ is a disc, relatively compact in the fiber $P^{-1}(z)$;
\item[(ii)] the boundary $\partial U_0$ is real analytic and transverse
to the fibers of $P$;
\item[(iii)] the boundary $\partial U_0$ is pseudoconvex in $U$.
\end{enumerate}
Then we restrict our attention to the fibration by discs $P_0 =
P\vert_{U_0} : U_0\rightarrow {\mathbb D}^n$. It is not difficult to
see that $U_0$ is Stein, but this fact will not really be used
below.

\begin{proposition} \label{yamaguchi} \cite{Ya1} \cite{Ya3}
The fiberwise Poincar\'e metric on $U_0\buildrel P_0\over\rightarrow
{\mathbb D}^n$ has a plurisubharmonic variation.
\end{proposition}

\begin{proof}
It is sufficient to consider the case $n=1$. The statement is local
on the base, and for every $z_0\in{\mathbb D}$ we can embed a
neighbourhood of $\overline{P_0^{-1}(z_0)}$ in $U$ into ${\mathbb
C}^2$ in such a way that $P$ becomes the projection to the first
coordinate (see, e.g., \cite[\S 3]{Suz}). Thus we may assume that
$U_0\subset {\mathbb D}\times{\mathbb C}$, $P_0(z,w)=z$, and
$P^{-1}(z) = D_z$ is a disc in $\{ z\} \times {\mathbb C} = {\mathbb
C}$, with real analytic boundary, depending on $z$ in a real
analytic and pseudoconvex way.

Take a holomorphic section $\alpha : {\mathbb D}\rightarrow U_0$ and
a holomorphic vertical vector field $v$ along $\alpha$, i.e. for
every $z\in{\mathbb D}$, $v(z)$ is a vector in $T_{\alpha (z)}U_0$
tangent to the fiber over $z$ (and nonvanishing). We need to prove
that $\log \| v(z)\|_{Poin(D_z)}$ is a subharmonic function on
${\mathbb D}$. By another change of coordinates, we may assume that
$\alpha (z) = (z,0)$ and $v(z) = \frac{\partial}{\partial
w}\vert_{(z,0)}$.

For every $z$, let
$$g(z,\cdot ): \overline{D_z}\rightarrow [0,+\infty ]$$ be the Green
function of $D_z$ with pole at $0$. That
is, $g(z,\cdot )$ is harmonic on $D_z\setminus\{ 0\}$, zero on
$\partial D_z$, and around $w=0$ it has the development
$$ g(z,w)= \log \frac{1}{|w|} + \lambda (z) + O(|w|).$$
The constant $\lambda (z)$ (Robin constant) is related to the
Poincar\'e metric of $D_z$: more precisely, we have
$$\lambda (z) = -\log \| \frac{\partial}{\partial
w}\vert_{(z,0)}\|_{Poin(D_z)}$$ (indeed, recall that the Green
function gives the radial part of a uniformisation of $D_z$). Hence,
we are reduced to show that $z\mapsto \lambda (z)$ is {\it
super}harmonic.

Fix $z_0\in {\mathbb D}$. By real analyticity of $\partial U_0$, the
function $g$ is (outside the poles) also real analytic, and thus
extensible (in a real analytic way) beyond $\partial U_0$. This
means that if $z$ is sufficiently close to $z_0$, then $g(z,\cdot )$
is actually defined on $\overline{D_{z_0}}$, and harmonic on
$D_{z_0}\setminus\{ 0\}$. Of course, $g(z,\cdot )$ does not need to
vanish on $\partial D_{z_0}$. The difference $g(z,\cdot ) -
g(z_0,\cdot )$ is harmonic on $D_{z_0}$ (the poles annihilate),
equal to $\lambda (z) - \lambda (z_0)$ at $0$, and equal to
$g(z,\cdot )$ on $\partial D_{z_0}$. By Green formula:
$$\lambda (z) - \lambda (z_0) = -\frac{1}{2\pi}\int_{\partial
D_{z_0}} g(z,w)\frac{\partial g}{\partial n}(z_0,w)ds$$ and
consequently:
$$\frac{\partial^2\lambda}{\partial z\partial\bar z}(z_0) =
-\frac{1}{2\pi}\int_{\partial D_{z_0}} \frac{\partial^2 g}{\partial
z\partial\bar z}(z_0,w)\frac{\partial g}{\partial n}(z_0,w)ds.$$

We now compute the $z$-laplacian of $g(\cdot ,w_0)$ when $w_0$ is a
point of the boundary $\partial D_{z_0}$.

The function $-g$ is, around $(z_0,w_0)$, a defining function for
$U_0$. By pseudoconvexity, the Levi form of $g$ at $(z_0,w_0)$ is
therefore nonpositive on the complex tangent space
$T_{(z_0,w_0)}^{\mathbb C}(\partial U_0)$, i.e. on the Kernel of
$\partial g$ at $(z_0,w_0)$ \cite[II.2]{Ran}. By developing, and
using also the fact that $g$ is $w$-harmonic, we obtain
$$\frac{\partial^2 g}{\partial z\partial\bar z}(z_0,w_0) \le 2 Re
\Big\{ \frac{\frac{\partial^2 g}{\partial w\partial\bar
z}(z_0,w_0)\cdot \frac{\partial g}{\partial
z}(z_0,w_0)}{\frac{\partial g}{\partial w}(z_0,w_0)}\Big\} .$$

We put this inequality into the expression of
$\frac{\partial^2\lambda}{\partial z\partial\bar z}(z_0)$ derived
above from Green formula, and then we apply Stokes theorem. We find
$$\frac{\partial^2\lambda}{\partial z\partial\bar z} (z_0) \le
-\frac{2}{\pi}\int_{D_{z_0}}\big| \frac{\partial^2g}{\partial
w\partial\bar z}(z_0,w)\big|^2 idw\wedge d\bar w \le 0$$ from which
we see that $\lambda $ is superharmonic.
\end{proof}

A similar result can be proved, by the same proof, even when we drop
the simply connectedness hypothesis on the fibers, for instance when
the fibers of $U_0$ are annuli instead of discs; however, the result
is that the Bergman fiberwise metric, and not the Poincar\'e one,
has a plurisubharmonic variation. This is because on a multiply
connected curve the Green function is more directly related to the
Bergman metric \cite{Ya3}. The case of the Poincar\'e metric is done
in \cite{Kiz}, by a covering argument. The general case of Theorem
\ref{kizuka} requires also to understand what happens when $\partial
U_0$ is still pseudoconvex but no more transverse to the fibers, so
that $U_0$ is no more a differentiably trivial family of curves.
This is rather delicate, and it is done in \cite{Ya1}. Then Theorem
\ref{kizuka} is proved by an exhaustion argument.

\subsection{Parabolic fibrations}

Theorem \ref{kizuka}, as stated, is rather empty when all the fibers
are isomorphic to ${\mathbb C}$. However, in that case Nishino
proved that if $U$ is Stein then it is isomorphic to ${\mathbb
D}^n\times {\mathbb C}$ \cite[II]{Nis}. A refinement of this was
found in \cite{Ya2}.

As before, we consider a fibration $P:U\rightarrow {\mathbb D}^n$
and we do not assume that $U$ is Stein. We suppose that there exists
an embedding $j : {\mathbb D}^n\times{\mathbb D}\rightarrow U$ such
that $P\circ j$ coincides with the projection from ${\mathbb
D}^n\times{\mathbb D}$ to ${\mathbb D}^n$ (this can always be done,
up to restricting the base). For every $\varepsilon \in [0,1)$, we
set
$$U_\varepsilon = U\setminus j({\mathbb D}^n \times
\overline{{\mathbb D}(\varepsilon )})$$ with $\overline{{\mathbb
D}(\varepsilon )} = \{ z\in{\mathbb C}\vert\quad |z|\le\varepsilon
\}$, and we denote by
$$P_\varepsilon : U_\varepsilon \rightarrow {\mathbb D}^n$$
the restriction of $P$. Thus, the fibers of $P_\varepsilon$ are
obtained from those of $P$ by removing a closed disc (if
$\varepsilon >0$) or a point (if $\varepsilon =0$).

\begin{theorem}\label{nishino} \cite[II]{Nis} \cite{Ya2}
Suppose that:
\begin{enumerate}
\item[(i)] for every $z\in{\mathbb D}^n$, the fiber $P^{-1}(z)$ is
isomorphic to ${\mathbb C}$;
\item[(ii)] for every $\varepsilon >0$ the fiberwise Poincar\'e
metric on $U_\varepsilon\buildrel P_\varepsilon\over\rightarrow
{\mathbb D}^n$ has a plurisubharmonic variation.
\end{enumerate}
Then $U$ is isomorphic to a product:
$$U\simeq {\mathbb D}^n\times{\mathbb C} .$$
\end{theorem}

\begin{proof}
For every $z\in{\mathbb D}^n$ we have a unique isomorphism
$$f(z,\cdot ):P^{-1}(z)\rightarrow {\mathbb C}$$
such that, using the coordinates given by $j$,
$$f(z,0)=0 \qquad {\rm and} \qquad f'(z,0)=1.$$
We want to prove that $f$ is holomorphic in $z$.

Set $R_\varepsilon (z)= f(z,P^{-1}_\varepsilon (z))\subset {\mathbb
C}$. By Koebe's Theorem, the distorsion of $f(z,\cdot )$ on compact
subsets of ${\mathbb D}$ is uniformly bounded, and so ${\mathbb
D}(\frac{1}{k}\varepsilon )\subset f(z,{\mathbb D}(\varepsilon
))\subset {\mathbb D}(k\varepsilon )$ for every $\varepsilon\in
(0,\frac{1}{2})$ and for some constant $k$, independent on $z$.
Therefore, for every $\varepsilon$ and $z$,
$${\mathbb C}\setminus {\mathbb D}(\frac{1}{k}\varepsilon ) \subset
R_\varepsilon (z) \subset {\mathbb C}\setminus {\mathbb
D}(k\varepsilon ).$$

In a similar way \cite[II]{Nis}, Koebe's Theorem gives also the
continuity of the above map $f$.

On the fibers of $P_0$, which are all isomorphic to ${\mathbb C}^*$,
we put the unique complete hermitian metric of zero curvature and
period (=length of closed simple geodesics) equal to $\sqrt{2}\pi$.
On the fibers of $P_\varepsilon$, $\varepsilon >0$, which are all
hyperbolic, we put the Poincar\'e metric multiplied by
$\log\varepsilon$, whose (constant) curvature is therefore equal to
$-\frac{1}{(\log\varepsilon )^2}$. By a simple and explicit
computation, the Poincar\'e metric on ${\mathbb C}\setminus {\mathbb
D}(c\varepsilon )$ multiplied by $\log\varepsilon$ converges
uniformly to the flat metric of period $\sqrt{2}\pi$ on ${\mathbb
C}^*$, as $\varepsilon\to 0$. Using this and the above bounds on
$R_\varepsilon (z)$, we obtain that our fiberwise metric on
$U_\varepsilon \buildrel P_\varepsilon\over\rightarrow {\mathbb
D}^n$ converges uniformly, as $\varepsilon\to 0$, to our fiberwise
metric on $U_0\buildrel P_0\over\rightarrow {\mathbb D}^n$ (see
\cite{Br4} for more explicit computations). Hence, from the
plurisubharmonic variation of the former we deduce the
plurisubharmonic variation of the latter.

Our flat metric on $P^{-1}_0(z)$ is the pull-back by $f(z,\cdot )$
of the metric $\frac{idx\wedge d\bar x}{4|x|^2}$ on $R_0(z) =
{\mathbb C}^*$. In the coordinates given by $j$, we have
$$f(z,w) = w\cdot e^{g(z,w)},$$
with $g$ holomorphic in $w$ and $g(z,0)=0$ for every $z$, by the
choice of the normalization. Hence, in these coordinates our metric
takes the form
$$\big| 1+w\frac{\partial g}{\partial w}(z,w)\big|^2 \cdot
\frac{idw\wedge d\bar w}{4|w|^2}.$$ Set $F=\log |1+w\frac{\partial
g}{\partial w}|^2$. We know, by the previous arguments, that $F$ is
plurisubharmonic. Moreover, $\frac{\partial^2 F}{\partial
w\partial\bar w}\equiv 0$, by flatness of the metric. By
semipositivity of the Levi form we then obtain $\frac{\partial^2
F}{\partial w\partial\bar z_k}\equiv 0$ for every $k$. Hence the
function $\frac{\partial F}{\partial w}$ is holomorphic, that is the
function $(\frac{\partial g}{\partial w} + w\frac{\partial^2
g}{\partial w^2})(1+w\frac{\partial g}{\partial w})^{-1}$ is
holomorphic. Taking into account that $g(z,0)\equiv 0$, we obtain
from this that $g$ also is fully holomorphic. Thus $f$ is fully
holomorphic in the chart given by $j$, and hence everywhere. It
follows that $U$ is isomorphic to a product.
\end{proof}

Remark that if $U$ is Stein then the hypothesis on the
plurisubharmonic variation is automatically satisfied, by Theorem
\ref{kizuka}, and because if $U$ is Stein then also $U_\varepsilon$
are Stein, for every $\varepsilon$. That was the situation
originally considered by Nishino and Yamaguchi.

A standard illustration of Theorem \ref{nishino} is the following
one. Take a continuous function $h:{\mathbb D}\rightarrow {\mathbb
P}$, let $\Gamma\subset{\mathbb D}\times{\mathbb P}$ be its graph,
and set $U=({\mathbb D}\times{\mathbb P})\setminus\Gamma$. Then $U$
fibers over ${\mathbb D}$ and all the fibers are isomorphic to
${\mathbb C}$. Clearly $U$ is isomorphic to a product ${\mathbb
D}\times{\mathbb C}$ if and only if $h$ is holomorphic, which in
turn is equivalent, by a classical result (due, once a time, to
Hartogs), to the Steinness of $U$.

\subsection{Foliations on Stein manifolds}

Even if we shall not need Il'yashen\-ko's results \cite{Il1}
\cite{Il2}, let us briefly explain them, as a warm-up for some basic
ideas.

Let $X$ be a Stein manifold, of dimension $n$, and let ${\mathcal
F}$ be a foliation by curves on $X$. In order to avoid some
technicalities (to which we will address later), let us assume that
${\mathcal F}$ is {\it nonsingular}, i.e. $Sing({\mathcal
F})=\emptyset$.

Take an embedded $(n-1)$-disc $T\subset X$ transverse to ${\mathcal
F}$. For every $t\in T$, let $L_t$ be the leaf of ${\mathcal F}$
through $t$, and let $\widetilde{L_t}$ be its universal covering
with basepoint $t$. Remark that, because $X$ is Stein, every
$\widetilde{L_t}$ is isomorphic either to ${\mathbb D}$ or to
${\mathbb C}$. In \cite{Il1} Il'yashenko proves that these universal
coverings $\{ \widetilde{L_t}\}_{t\in T}$ can be glued together to
get a complex manifold of dimension $n$, a sort of ``long flow
box''. More precisely, there exists a complex $n$-manifold $U_T$
with the following properties:
\begin{enumerate}
\item[(i)] $U_T$ admits a submersion $P_T:U_T\to T$ and a section
$p_T:T\to U_T$ such that, for every $t\in T$, the pointed fiber
$(P_T^{-1}(t),p_T(t))$ is identified (in a natural way) with
$(\widetilde{L_t},t)$;
\item[(ii)] $U_T$ admits an immersion (i.e., local biholomorphism)
$\Pi_T:U_T\to X$ which sends
each fiber $(\widetilde{L_t},t)$ to the corresponding leaf
$(L_t,t)$, as universal covering.
\end{enumerate}
We shall not prove here these facts, because we shall prove later
(Section 4) some closely related facts in the context of (singular)
foliations on compact K\"ahler manifolds.

\begin{theorem} \label{ilyashenko} \cite{Il1} \cite{Il2}
The manifold $U_T$ is Stein.
\end{theorem}

\begin{proof}
Following Suzuki \cite{Suz}, it is useful to factorize the immersion
$U_T\to X$ through another manifold $V_T$, which is constructed in a
similar way as $U_T$ except that the universal coverings
$\widetilde{L_t}$ are replaced by the holonomy coverings
$\widehat{L_t}$.

Here is Suzuki's construction. Fix a foliated chart $\Omega\subset
X$ around $T$, i.e. $\Omega\simeq {\mathbb D}^{n-1}\times{\mathbb
D}$, $T\simeq {\mathbb D}^{n-1}\times\{ 0\}$, ${\mathcal
F}\vert_\Omega =$ vertical foliation, with leaves $\{ \ast\}\times
{\mathbb D}$. Let ${\mathcal O}_{\mathcal F}(\Omega )$ be the set of
holomorphic functions on $\Omega$ which are constant on the leaves
of ${\mathcal F}\vert_\Omega$, i.e. which depend only on the first
$(n-1)$ coordinates. Let $\overline V_T$ be the {\it existence
domain} of ${\mathcal O}_{\mathcal F}(\Omega )$ over $X$: by
definition, this is the maximal holomorphically separable Riemann
domain
$$\overline V_T \rightarrow X$$
which contains $\Omega$ and such that every $f\in{\mathcal
O}_{\mathcal F}(\Omega )$ extends to some $\tilde f\in{\mathcal
O}(\overline V_T)$. The classical Cartan-Thullen-Oka theory
\cite{GuR} says that $\overline V_T$ is a Stein manifold.

The projection $\Omega\to T$ extends to a map
$$\overline Q_T : \overline V_T \rightarrow T$$
thanks to ${\mathcal O}_{\mathcal F}(\Omega )\hookrightarrow
{\mathcal O}(\overline V_T)$. Consider a fiber $\overline
Q_T^{-1}(t)$. It is not difficult to see that the connected
component of $\overline Q_T^{-1}(t)$ which cuts $\Omega$ ($\subset
\overline V_T$) is exactly the holonomy covering $\widehat{L_t}$ of
$L_t$, with basepoint $t$. The reason is the following one. Firstly,
if $\gamma :[0,1]\to L_t$ is a path contained in a leaf, with
$\gamma (0) = t$, then any function $f\in {\mathcal O}_{\mathcal
F}(\Omega )$ can be analytically prolonged along $\gamma$, by
preserving the constancy on the leaves. Secondly, if $\gamma_1$ and
$\gamma_2$ are two such paths with the same endpoint $s\in L_t$,
then the germs at $s$ obtained by the two continuations of $f$ along
$\gamma_1$ and $\gamma_2$ may be different. If the foliation has
trivial holonomy along $\gamma_1\ast\gamma_2^{-1}$, then the two
germs are certainly equal; conversely, if the holonomy is not
trivial, then we can find $f$ such that the two final germs are
different. This argument shows that $\widehat{L_t}$ is naturally
contained into $\overline Q_T^{-1}(t)$. The fact that it is a
connected component is just a ``maximality'' argument (note that
$\overline V_T$ is foliated by the pull-back of ${\mathcal F}$, and
fibers of $\overline Q_T$ are closed subvarieties invariant by this
foliation).

We denote by $V_T\subset\overline V_T$ (open subset) the union of
these holonomy coverings, and by $Q_T$ the restriction of $\overline
Q_T$ to $V_T$.

Let us return to $U_T$. We have a natural map (local biholomorphism)
$$F_T : U_T \rightarrow V_T$$
which acts as a covering between fibers (but not globally: see
Examples \ref{elliptic2} and \ref{elliptic3} below). In particular,
$U_T$ is a Riemann domain over the Stein manifold $\overline V_T$.

\begin{lemma}\label{separability}
$U_T$ is holomorphically separable.
\end{lemma}

\begin{proof}
Given $p,q\in U_T$, $p\not= q$, we want to construct $f\in {\mathcal
O}(U_T)$ such that $f(p)\not= f(q)$. The only nontrivial case ($V_T$
being holomorphically separable) is the case where $F_T(p)=F_T(q)$,
in particular $p$ and $q$ belong to the same fiber
$\widetilde{L_t}$.

We use the following procedure. Take a path $\gamma$ in
$\widetilde{L_t}$ from $p$ to $q$. It projects by $F_T$ to a closed
path $\gamma_0$ in $\widehat{L_t}$. Suppose that $[ \gamma_0 ]\not=
0$ in $H_1(\widehat{L_t}, {\mathbb R})$. Then we may find a
holomorphic 1-form $\omega\in\Omega^1(\widehat{L_t})$ such that
$\int_{\gamma_0}\omega =1$. This 1-form can be holomorphically
extended from $\widehat{L_t}$ to $V_T\subset\overline V_T$, because
$\overline V_T$ is Stein and $\widehat{L_t}$ is a closed submanifold
of it. Call $\widehat\omega$ such an extension, and
$\widetilde\omega = F_T^*(\widehat\omega )$ its lift to $U_T$. On
every (simply connected!) fiber $\widetilde{L_t}$ of $U_T$ the
1-form $\widetilde\omega$ is exact, and can be integrated giving a
holomorphic function $f_t(z)=\int_t^z
\widetilde\omega\vert_{\widetilde{L_t}}$. We thus obtain a
holomorphic function $f$ on $U_T$, which separates $p$ and $q$:
$f(p)-f(q)= \int_\gamma\widetilde\omega =
\int_{\gamma_0}\widehat\omega =1$.

This procedure does not work if $[ \gamma_0]=0$: in that case, every
$\omega\in\Omega^1(\widehat{L_t})$ has period equal to zero on
$\gamma_0$. But, in that case, we may find two 1-forms
$\omega_1,\omega_2\in\Omega^1(\widehat{L_t})$ such that the iterated
integral of $(\omega_1,\omega_2)$ along $\gamma_0$ is not zero (this
iterated integral \cite{Che} is just the integral along $\gamma$ of
$\phi_1d\phi_2$, where $\phi_j$ is a primitive of $\omega_j$ lifted
to $\widetilde{L_t}$). Then we can repeat the argument above: the
fiberwise iterated integral of
$(\widetilde\omega_1,\widetilde\omega_2)$ is a holomorphic function
on $U_T$ which separates $p$ and $q$.
\end{proof}

Having established that $U_T$ is a holomorphically separable Riemann
domain over $\overline V_T$, it is again a fundamental result of
Cartan-Thullen-Oka theory \cite{GuR} that there exists a Stein
Riemann domain
$$\overline F_T : \overline U_T \rightarrow \overline V_T$$
which contains $U_T$ and such that ${\mathcal O}(\overline U_T) =
{\mathcal O}(U_T)$. The map $P_T : U_T\to T$ extends to
$$\overline P_T : \overline U_T \to T ,$$
and $U_T$ can be identified with the
open subset of $\overline U_T$ composed by the connected components
of fibers of $\overline P_T$ which cut $\Omega\subset\overline U_T$.
But, in fact, much better is true:

\begin{lemma} \label{connectivity}
Every fiber of $\overline P_T$ is connected, that is:
$$\overline U_T = U_T .$$
\end{lemma}

\begin{proof}
If not, then, by a connectivity argument, we may find
$a_0,b_0\in\overline P_T^{-1}(t_0)$, $a_k,b_k\in\overline
P_T^{-1}(t_k)$, with $a_k\to a_0$ and $b_k\to b_0$, such that:
\begin{enumerate}
\item[(i)] $a_0\in\widetilde{L_{t_0}}$, $b_0\in\overline
P_T^{-1}(t_0)\setminus\widetilde{L_{t_0}}$;
\item[(ii)] $a_k,b_k\in\widetilde{L_{t_k}}$.
\end{enumerate}
Denote by ${\mathcal M}_{t_0}$ the maximal ideal of ${\mathcal
O}_{t_0}$ (on $T$), and for every $p\in \overline P_T^{-1}(t_0)$
denote by ${\mathcal I}_p\subset {\mathcal O}_p$ the ideal generated
by $(\overline P_T)^*({\mathcal M}_{t_0})$. At points of
$\widetilde{L_{t_0}}$, this is just the ideal of functions vanishing
along $\widetilde{L_{t_0}}$; whereas at points of $\overline
P_T^{-1}(t_0) \setminus\widetilde{L_{t_0}}$, at which $\overline
P_T$ may fail to be a submersion, this ideal may correspond to a
``higher order'' vanishing. Because $\overline U_T$ is Stein and
$\overline P_T^{-1}(t_0)$ is a closed subvariety, we may find a
holomorphic function $f\in{\mathcal O}(\overline U_T)$ such that:
\begin{enumerate}
\item[(iii)] $f\equiv 0$ on $\widetilde{L_{t_0}}$, $f\equiv 1$ on
$\overline P_T^{-1}(t_0)\setminus\widetilde{L_{t_0}}$;
\item[(iv)] for every $p\in\overline P_T^{-1}(t_0)$, the
differential $df_p$ of $f$ at $p$ belongs to the ideal ${\mathcal
I}_p\Omega^1_p$.
\end{enumerate}
Let $\{ z_1,\ldots ,z_{n-1}\}$ denote the coordinates on $T$ lifted
to $\overline U_T$. Then, by property (iv), we can factorize
$$df = \sum_{j=1}^{n-1} (z_j - z_j(t_0))\cdot\beta_j$$
where $\beta_j$ are holomorphic 1-forms on $\overline U_T$.

As in Lemma \ref{separability}, each $\beta_j$ can be integrated
along the simply connected fibers of $U_T$ (with starting point on
$T$), giving a function $g_j\in{\mathcal O}(U_T)$. This function can
be holomorphically extended to the envelope $\overline U_T$. By the
factorization above, and (ii), we have
$$f(b_k)-f(a_k) = \sum_{j=1}^{n-1} (z_j(t_k) - z_j(t_0))\cdot
(g_j(b_k) - g_j(a_k))$$ and this expression tends to 0 as $k\to
+\infty$. Therefore $f(b_0)=f(a_0)$, in contradiction with (i) and
(iii).
\end{proof}

It follows from this Lemma that $U_T=\overline U_T$ is Stein.
\end{proof}

\begin{remark}. {\rm
We do not know if $V_T$ also is Stein, i.e. if $V_T = \overline
V_T$.}
\end{remark}

This Theorem allows to apply the results of Nishino and Yamaguchi
discussed above to holomorphic foliations on Stein manifolds. For
instance: the set of parabolic leaves of such a foliation is either
full or complete pluripolar. A similar point of view is pursued in
\cite{Suz}.

\section{The unparametrized Hartogs extension lemma}

In order to construct the leafwise universal covering of a
foliation, we shall need an extension lemma of Hartogs type. This is
done in this Section.

Let $X$ be a compact K\"ahler manifold. Denote by $A_r$, $r\in
(0,1)$, the semiclosed annulus $\{ r < |w| \le 1\}$, with boundary
$\partial A_r = \{ |w| =1\}$. Given a holomorphic immersion
$$f : A_r \rightarrow X$$
we shall say that $f(A_r)$ {\bf extends to a disc} if there exists a
holomorphic map
$$g : \overline{\mathbb D} \rightarrow X ,$$
not necessarily immersive, such that $f$ factorizes as $g\circ j$
for some embedding $j: A_r\rightarrow \overline{\mathbb D}$, sending
$\partial A_r$ to $\partial{\mathbb D}$. That is, $f$ itself does
not need to extend to the full disc $\{ |w|\le 1\}$, but it extends
``after a reparametrization'', given by $j$.

Remark that if $f$ is an embedding, and $f(A_r)$ extends to a disc,
then we can find $g$ as above which is moreover injective outside a
finite subset. The image $g(\overline{\mathbb D})$ is a (possibly
singular) disc in $X$ with boundary $f(\partial A_r)$. Such an
extension $g$ or $g(\overline{\mathbb D})$ will be called {\bf
simple} extension of $f$ or $f(A_r)$. Note that such a $g$ is
uniquely defined up to a Mo\"ebius reparametrization of
$\overline{\mathbb D}$.

Given a holomorphic immersion
$$f : {\mathbb D}^k\times A_r \rightarrow X$$
we shall say that $f({\mathbb D}^k\times A_r)$ {\bf extends to a
meromorphic family of discs} if there exists a meromorphic map
$$g : W \dashrightarrow X$$
such that:
\begin{enumerate}
\item[(i)] $W$ is a complex manifold of dimension $k+1$ with
boundary, equipped with a holomorphic submersion $W\rightarrow
{\mathbb D}^k$ all of whose fibers $W_z$, $z\in{\mathbb D}^k$, are
isomorphic to $\overline{\mathbb D}$;
\item[(ii)] $f$ factorizes as $g\circ j$ for some embedding $j:
{\mathbb D}^k\times A_r \rightarrow W$, sending ${\mathbb
D}^k\times\partial A_r$ to $\partial W$ and $\{ z\}\times A_r$ into
$W_z$, for every $z\in {\mathbb D}^k$.
\end{enumerate}
In particular, the restriction of $g$ to the fiber $W_z$ gives,
after removal of indeterminacies, a disc which extends $f(z,A_r)$,
and these discs depend on $z$ in a meromorphic way. The manifold $W$
is differentiably a product of ${\mathbb D}^k$ with
$\overline{\mathbb D}$, but in general this does not hold
holomorphically. However, note that by definition $W$ is around its
boundary $\partial W$ isomorphic to a product ${\mathbb D}^k\times
A_r$.

We shall say that an immersion $f : {\mathbb D}^k\times A_r
\rightarrow X$ is an {\bf almost embedding} if there exists a proper
analytic subset $I\subset {\mathbb D}^k$ such that the restriction
of $f$ to $({\mathbb D}^k\setminus I)\times A_r$ is an embedding. In
particular, for every $z\in {\mathbb D}^k\setminus I$, $f(z,A_r)$ is
an embedded annulus in $X$, and $f(z,A_r)$, $f(z',A_r)$ are disjoint
if $z,z'\in{\mathbb D}^k\setminus I$ are different.

The following result is a sort of ``unparametrized'' Hartogs
extension lemma \cite{Siu} \cite{Iv1}, in which the extension of
maps is replaced by the extension of their images. Its proof is
inspired by \cite{Iv1} and \cite{Iv2}. The new difficulty is that we
need to construct not only a map but also the space where it is
defined. The necessity of this unparametrized Hartogs lemma for our
future constructions, instead of the usual parametrized one, has
been observed in \cite{ChI}.

\begin{theorem} \label{hartogs}
Let $X$ be a compact K\"ahler manifold and let $f:{\mathbb
D}^k\times A_r \rightarrow X$ be an almost embedding. Suppose that
there exists an open nonempty subset $\Omega\subset{\mathbb D}^k$
such that $f(z,A_r)$ extends to a disc for every $z\in\Omega$. Then
$f({\mathbb D}^k\times A_r)$ extends to a meromorphic family of
discs.
\end{theorem}

\begin{proof}
Consider the subset
$$Z = \{\ z\in{\mathbb D}^k\setminus I\ \vert\ f(z,A_r)\ {\rm
extends\ to\ a\ disc}\ \} .$$ Our first aim is to give to $Z$ a
complex analytic structure with countable base . This is a rather
standard fact, see \cite{Iv2} for related ideas and \cite{CaP} for a
larger perspective.

For every $z\in Z$, fix a simple extension
$$g_z : \overline{\mathbb D} \rightarrow X$$
of $f(z,A_r)$. We firstly put on $Z$ the following metrizable
topology: we define the distance between $z_1,z_2\in Z$ as the
Hausdorff distance in $X$ between the discs
$g_{z_1}(\overline{\mathbb D})$ and $g_{z_2}(\overline{\mathbb D})$.
Note that this topology may be {\it finer} than the topology induced
by the inclusion $Z\subset{\mathbb D}^k$: if $z_1,z_2\in Z$ are
close each other in ${\mathbb D}^k$ then $g_{z_1}(\overline{\mathbb
D})$, $g_{z_2}(\overline{\mathbb D})$ are close each other near
their boundaries, but their interiors may be far each other (think
to blow-up).

Take $z\in Z$ and take a Stein neighbourhood $U\subset X$ of
$g_{z}(\overline{\mathbb D})$. Consider the subset $A\subset
{\mathbb D}^k\setminus I$ of those points $z'$ such that the circle
$f(z',\partial A_r)$ is the boundary of a compact complex curve
$C_{z'}$ contained in $U$. Note that, by the maximum principle, such
a curve is Hausdorff-close to $g_{z}(\overline{\mathbb D})$, if $z'$
is close to $z$. According to a theorem of Wermer or Harvey-Lawson
\cite[Ch.19]{AWe}, this condition is equivalent to say that
$\int_{f(z',\partial A_r)}\beta = 0$ for every holomorphic 1-form
$\beta$ on $U$ (moment condition). These integrals depend
holomorphically on $z'$, for every $\beta$. We deduce (by
noetherianity) that $A$ is an analytic subset of ${\mathbb
D}^k\setminus I$, on a neighbourhood of $z$. For every $z'\in A$,
however, the curve $C_{z'}$ is not necessarily the image of a disc:
recall that $g_{z}(\overline{\mathbb D})$ may be singular and may
have selfintersections, and so a curve close to it may have positive
genus, arising from smoothing the singularities.

Set ${\mathcal A} = \{\ (z',x)\in A\times U\ \vert\ x\in C_{z'}\
\}$. By inspection of the proof of Wermer-Harvey-Lawson theorem
\cite[Ch.19]{AWe}, we see that $\mathcal A$ is an analytic subset of
$A\times U$ (just by the holomorphic dependence on parameters of the
Cauchy transform used in that proof to construct $C_{z'}$). We have
a tautological fibration $\pi : {\mathcal A}\to A$ and a
tautological map $\tau : {\mathcal A}\to U$ defined by the two
projections. Let $B\subset A$ be the subset of those points $z'$
such that the fiber $\pi^{-1}(z') = C_{z'}$ has geometric genus
zero. This is an analytic subset of $A$ (the function $z'\mapsto\{$
geometric genus of $\pi^{-1}(z')\}$ is Zariski lower
semicontinuous). By restriction, we have a tautological fibration
$\pi :{\mathcal B}\to B$ and a tautological map $\tau : {\mathcal
B}\to U\subset X$. Each fiber of $\pi$ over $B$ is a disc, sent by
$\tau$ to a disc in $U$ with boundary $f(z',\partial A_r)$. In
particular, $B$ is contained in $Z$.

Now, a neighbourhood of $z$ in $B$ can be identified with a
neighbourhood of $z$ in $Z$ (in the $Z$-topology above): if $z'\in
Z$ is $Z$-close to $z$ then $g_{z'}(\overline{\mathbb D})$ is
contained in $U$ and then $z'\in B$. In this way, the analytic
structure of $B$ is transferred to $Z$. Note that, with this complex
analytic structure, the inclusion $Z\hookrightarrow {\mathbb D}^k$
is holomorphic. More precisely, each irreducible component of $Z$ is
a {\it locally analytic subset} of ${\mathbb D}^k\setminus I$
(where, as usual, ``locally analytic'' means ``analytic in a
neighbourhood of it''; of course, a component does not need to be
closed in ${\mathbb D}^k\setminus I$).

Let us now prove that the complex analytic space $Z$ has a {\it
countable} number of irreducible components.

To see this, we use the area function ${\bf a} : Z\to {\mathbb
R}^+$, defined by
$${\bf a}(z) =\ {\rm area\ of}\ g_z(\overline{\mathbb D}) =
\int_{\overline{\mathbb D}} g_z^*(\omega )$$ ($\omega =$ K\"ahler
form of $X$). This function is continuous on $Z$. Let $c >0$ be the
minimal area of rational curves in $X$. Set, for every $m\in{\mathbb
N}$,
$$Z_m = \{\ z\in Z\ \vert\ {\bf a}(z)\in \big( m\frac{c}{2},
(m+2)\frac{c}{2}\big) \ \} ,$$ so that $Z$ is covered by
$\cup_{m=0}^{+\infty} Z_m$. Each $Z_m$ is open in $Z$, and we claim
that on each $Z_m$ the $Z$-topology coincides with the ${\mathbb
D}^k$-topology. Indeed, take a sequence $\{ z_n\}\subset Z_m$ which
${\mathbb D}^k$-converges to $z_\infty\in Z_m$. We thus have, in
$X$, a sequence of discs $g_{z_n}(\overline{\mathbb D})$ with
boundaries $f(z_n,\partial A_r)$ and areas in the interval
$(m\frac{c}{2}, (m+2)\frac{c}{2})$. By Bishop's compactness theorem
\cite{Bis} \cite[Prop.3.1]{Iv1}, up to subsequencing,
$g_{z_n}(\overline{\mathbb D})$ converges, in the Hausdorff
topology, to a compact complex curve of the form $D\cup Rat$, where
$D$ is a disc with boundary $f(z_\infty ,\partial A_r)$ and $Rat$ is
a finite union of rational curves (the bubbles). Necessarily,
$D=g_{z_\infty}(\overline{\mathbb D})$. Moreover,
$$\lim_{n\to +\infty} {\rm area}(g_{z_n}(\overline{\mathbb D})) =
{\rm area}(g_{z_\infty}(\overline{\mathbb D})) + {\rm area}(Rat).$$
From ${\bf a}(z_\infty ),{\bf a}(z_n)\in (m\frac{c}{2},
(m+2)\frac{c}{2})$ it follows that ${\rm area}(Rat) <c$, hence, by
definition of $c$, $Rat =\emptyset$. Hence
$g_{z_n}(\overline{\mathbb D})$ converges, in the Hausdorff
topology, to $g_{z_\infty}(\overline{\mathbb D})$, i.e. $z_n$
converges to $z_\infty$ in the $Z$-topology.

Therefore, if $L_m\subset Z_m$ is a countable ${\mathbb D}^k$-dense
subset then $L_m$ is also $Z$-dense in $Z_m$, and
$\cup_{m=0}^{+\infty}L_m$ is countable and $Z$-dense in $Z$. It
follows that $Z$ has countably many irreducible components.

After these preliminaries, we can really start the proof of the
theorem.

The hypotheses imply that the space $Z$ has (at least) one
irreducible component $V$ which is open in ${\mathbb D}^k\setminus
I$. Let us consider again the area function ${\bf a}$ on $V$. The
following lemma is classical, and it is at the base of every
extension theorem for maps into K\"ahler manifolds \cite{Siu}
\cite{Iv1}.

\begin{lemma}\label{stokes}
For every compact $K\subset{\mathbb D}^k$, the function ${\bf a}$ is
bounded on $V\cap K$.
\end{lemma}

\begin{proof}
If $z_0,z_1\in V$, then we can join them by a continuous path $\{
z_t\}_{t\in [0,1]}\subset V$, so that we have in $X$ a continuous
family of discs $g_{z_t}(\overline{\mathbb D})$, with boundaries
$f(z_t,\partial A_r)$. By Stokes formula, the difference between the
area of $g_{z_1}(\overline{\mathbb D})$ and
$g_{z_0}(\overline{\mathbb D})$ is equal to the integral of the
K\"ahler form $\omega$ on the ```tube'' $\cup_{t\in [0,1]}
g_{z_t}(\partial{\mathbb D})=$ $f(\cup_{t\in [0,1]}\{ z_t\}\times
\partial A_r)$. Now, for topological reasons, $f^*(\omega )$ admits
a primitive $\lambda$ on ${\mathbb D}^k\times A_r$. Therefore
$${\bf a}(z_1) - {\bf a}(z_0) = \int_{\{ z_1\}\times\partial A_r}
\lambda - \int_{\{ z_0\}\times\partial A_r} \lambda .$$ Remark that
the function $z\mapsto \int_{\{ z\}\times\partial A_r} \lambda$ is
defined (and smooth) on the full ${\mathbb D}^k$, not only on $V$,
and so it is bounded on every compact $K\subset {\mathbb D}^k$. The
conclusion follows immediately.
\end{proof}

We use this lemma to study the boundary of $V$, and to show that the
complement of $V$ is small.

Take a point $z_\infty\in ({\mathbb D}^k\setminus I)\cap\partial V$
and a sequence $z_n\in V$ converging to $z_\infty$. By the
boundedness of ${\bf a}(z_n)$ and Bishop compactness theorem, we
obtain a disc in $X$ with boundary $f(z_\infty ,\partial A_r)$
(plus, perhaps, some rational bubbles, but we may forget them). In
particular, the point $z_\infty$ belongs to $Z$. Obviously, the
irreducible component of $Z$ which contains $z_\infty$ is not open
in ${\mathbb D}^k\setminus I$, because $z_\infty\in\partial V$, and
so that component is a locally analytic subset of ${\mathbb
D}^k\setminus I$ of {\it positive} codimension. It follows that the
boundary $\partial V$ is a {\it thin subset} of ${\mathbb D}^k$,
i.e. it is contained in a countable union of locally analytic
subsets of positive codimension (certain components of $Z$, plus the
analytic subset $I$). Disconnectedness properties of thin subsets
show that also the complement ${\mathbb D}^k\setminus V$ ($=\partial
V$) {\it is thin in} ${\mathbb D}^k$.

Recall now that over $V$ we have the (normalized) tautological
fibration $\pi :{\mathcal V}\to V$, equipped with the tautological
map $\tau : {\mathcal V}\to X$. Basically, this provides the desired
extension of $f$ over the large open subset $V$. As in \cite{Iv1},
we shall get the extension over the full ${\mathbb D}^k$ by reducing
to the Thullen type theorem of Siu \cite{Siu}.

By construction, $\partial{\mathcal V}$ has a neigbourhood
isomorphic to $V\times A_r$, the isomorphism being realized by $f$.
Hence we can glue to ${\mathcal V}$ the space ${\mathbb D}^k\times
A_r$, using the same $f$. We obtain a new space ${\mathcal W}$
equipped with a fibration $\pi : {\mathcal W}\to {\mathbb D}^k$ and
a map $\tau : {\mathcal W}\to X$ such that:
\begin{enumerate}
\item[(i)] $\pi^{-1}(z)\simeq \overline{\mathbb D}$ for $z\in V$,
$\pi^{-1}(z)\simeq A_r$ for $z\in {\mathbb D}^k\setminus V$;
\item[(ii)] $f$ factorizes through $\tau$.
\end{enumerate}
In other words, and recalling how ${\mathcal V}$ was defined, up to
normalization ${\mathcal W}$ is simply the analytic subset of
${\mathbb D}^k\times X$ given by the union of all the discs $\{ z\}
\times g_z(\overline{\mathbb D})$, $z\in V$, and all the annuli $\{
z\} \times f(z,A_r)$, $z\in{\mathbb D}^k\setminus V$.

\begin{lemma}\label{koebe}
There exists an embedding ${\mathcal W}\rightarrow {\mathbb
D}^k\times {\mathbb P}$, which respects the fibrations over
${\mathbb D}^k$.
\end{lemma}

\begin{proof}
Set $B_r = \{\ w\in {\mathbb P}\ \vert\ |w| > r\ \}$. By
construction, $\partial{\mathcal W}$ has a neighbourhood isomorphic
to ${\mathbb D}^k\times A_r$. We can glue ${\mathbb D}^k\times B_r$
to ${\mathcal W}$ by identification of that neighbourhood with
${\mathbb D}^k\times A_r \subset {\mathbb D}^k\times B_r$, i.e. by
prolonging each annulus $A_r$ to a disc $B_r$. The result is a new
space $\widehat{\mathcal W}$ with a fibration $\widehat\pi :
\widehat{\mathcal W}\to{\mathbb D}^k$ such that:
\begin{enumerate}
\item[(i)] $\widehat\pi ^{-1}(z)\simeq {\mathbb P}$ for every $z\in V$;
\item[(ii)] $\widehat\pi^{-1}(z)\simeq B_r$ for every $z\in
{\mathbb D}^k\setminus V$.
\end{enumerate}
We shall prove that $\widehat{\mathcal W}$ (and hence ${\mathcal
W}$) embeds into ${\mathbb D}^k\times{\mathbb P}$ (incidentally,
note the common features with the proof of Theorem \ref{nishino}).

For every $z\in V$, there exists a unique isomorphism
$$\varphi_z : \widehat\pi^{-1}(z)\rightarrow {\mathbb P}$$
such that
$$\varphi_z(\infty )=0\ ,\quad \varphi_z'(\infty )=1\ ,\quad
\varphi_z(r)=\infty$$ where $\infty ,r\in\overline{B_r}\subset
\widehat\pi^{-1}(z)$ and the derivative at $\infty$ is computed
using the coordinate $\frac{1}{w}$. Every ${\mathbb P}$-fibration is
locally trivial, and so this isomorphism $\varphi_z$ depends
holomorphically on $z$. Thus we obtain a biholomorphism
$$\Phi : \widehat\pi^{-1}(V)\rightarrow V\times {\mathbb P}$$
and we want to prove that $\Phi$ extends to the full
$\widehat{\mathcal W}$.

By Koebe's Theorem, the distorsion of $\varphi_z$ on any compact
$K\subset B_r$ is uniformly bounded (note that
$\varphi_z(B_r)\subset {\mathbb C}$). Hence, for every $w_0\in B_r$
the holomorphic function $z\mapsto \varphi_z(w_0)$ is bounded on
$V$. Because the complement of $V$ in ${\mathbb D}^k$ is thin, by
Riemann's extension theorem this function extends holomorphically to
${\mathbb D}^k$. This permits to extends the above $\Phi$ also to
fibers over ${\mathbb D}^k\setminus V$. Still by bounded distorsion,
this extension is an embedding of $\widehat{\mathcal W}$ into
${\mathbb D}^k\times{\mathbb P}$.
\end{proof}

Now we can finish the proof of the theorem. Thanks to the previous
embedding, we may ``fill in'' the holes of ${\mathcal W}$ and obtain
a $\overline{\mathbb D}$-fibration $W$ over ${\mathbb D}^k$. Then,
by the Thullen type theorem of Siu \cite{Siu} (and transfinite
induction) the map $\tau : {\mathcal W}\to X$ can be meromorphically
extended to $W$. This is the meromorphic family of discs which
extends $f({\mathbb D}^k\times A_r)$.
\end{proof}

By comparison with the usual ``parametrized'' Hartogs extension
lemma \cite{Iv1}, one could ask if the almost embedding hypothesis
in Theorem \ref{hartogs} is really indispensable. In some sense, the
answer is yes. Indeed, we may easily construct a fibered immersion
$f : {\mathbb D}\times A_r \to {\mathbb D}\times{\mathbb P} \subset
{\mathbb P}\times{\mathbb P}$, $f(z,w)=(z,f_0(z,w))$, such that: (i)
for some $z_0\in{\mathbb D}$, $f_0(z_0,\partial A_r)$ is an embedded
circle in ${\mathbb P}$; (ii) for some other $z_1\in{\mathbb D}$,
$f_0(z_1,\partial A_r)$ is an immersed but not embedded circle in
${\mathbb P}$. Then, for some neighbourhhod $U\subset{\mathbb D}$ of
$z_0$, $f(U\times A_r)$ can be obviously extended to a meromorphic
(even holomorphic) family of discs, but such a $U$ cannot be
enlarged to contain $z_1$, because $f_0(z_1,\partial A_r)$ bounds no
disc in ${\mathbb P}$. Note, however, that $f_0(z_1,\partial A_r)$
bounds a so called holomorphic chain in ${\mathbb P}$
\cite[Ch.19]{AWe}: if $\Omega_1,...,\Omega_m$ are the connected
components of ${\mathbb P}\setminus f_0(z_1,\partial A_r)$, then
$f_0(z_1,\partial A_r)$ is the ``algebraic'' boundary of
$\sum_{j=1}^m n_j\Omega_j$, for suitable integers $n_j$. It is
conceivable that Theorem \ref{hartogs} holds under the sole
assumption that $f$ is an immersion, provided that the manifold $W$
is replaced by a (suitably defined) ``meromorphic family of
1-dimensional chains''.

\section{Holonomy tubes and covering tubes}

Here we shall define leaves, holonomy tubes and covering tubes,
following \cite{Br3}.

Let $X$ be a compact K\"ahler manifold, of dimension $n$, and let
${\mathcal F}$ be a foliation by curves on $X$. Set $X^0 =
X\setminus Sing({\mathcal F})$ and ${\mathcal F}^0 = {\mathcal
F}\vert_{X^0}$. We could define the ``leaves'' of the singular
foliation ${\mathcal F}$ simply as the usual leaves of the
nonsingular foliation ${\mathcal F}^0$. However, for our purposes we
shall need that the universal coverings of the leaves glue together
in a nice way, producing what we shall call covering tubes. This
shall require a sort of semicontinuity of the fundamental groups of
the leaves. With the na\"\i ve definition ``leaves of ${\mathcal F}$
= leaves of ${\mathcal F}^0$'', such a semicontinuity can fail, in
the sense that a leaf (of ${\mathcal F}^0$) can have a larger
fundamental group than nearby leaves (of ${\mathcal F}^0$). To
remedy to this, we give now a less na\"\i ve definition of leaf of
${\mathcal F}$, which has the effect of killing certain homotopy
classes of cycles, and the problem will be settled almost by
definition (but we will require also the unparametrized Hartogs
extension lemma of the previous Section).

\subsection{Vanishing ends}

Take a point $p\in X^0$, and let $L_p^0$ be the leaf of ${\mathcal
F}^0$ through $p$. It is a smooth complex connected curve, equipped
with an injective immersion
$$i_p^0 : L_p^0 \rightarrow X^0 ,$$
and sometimes we will tacitly identify $L_p^0$ with its image in
$X^0$ or $X$. Recall that, given a local transversal ${\mathbb
D}^{n-1}\hookrightarrow X^0$ to ${\mathcal F}^0$ at $p$, we have a
holonomy representation \cite{CLN}
$$hol_p : \pi_1(L_p^0,p)\to Diff ({\mathbb D}^{n-1},0)$$
of the fundamental group of $L_p^0$ with basepoint $p$ into the
group of germs of holomorphic diffeomorphisms of $({\mathbb
D}^{n-1},0)$.

Let $E\subset L_p^0$ be a {\it parabolic end} of $L_p^0$, that is a
closed subset isomorphic to the punctured closed disc
$\overline{\mathbb D}^* = \{\ 0 < |w| \le 1\ \}$, and suppose that
the holonomy of ${\mathcal F}^0$ along the cycle $\partial E$ is
trivial. Then, for some $r\in (0,1)$, the inclusion
$A_r\subset\overline{\mathbb D}^* = E$ can be extended to an
embedding ${\mathbb D}^{n-1}\times A_r \to X^0$ which sends each $\{
z\}\times A_r$ into a leaf of ${\mathcal F}^0$, and $\{ 0\}\times
A_r$ to $A_r\subset E$ (this is because $A_r$ is Stein, see for
instance \cite[\S 3]{Suz}). More generally, if the holonomy of
${\mathcal F}^0$ along $\partial E$ is finite, of order $k$, then we
can find an immersion ${\mathbb D}^{n-1}\times A_r\to X^0$ which
sends each $\{ z\}\times A_r$ into a leaf of ${\mathcal F}^0$ and
$\{ 0\}\times A_r$ to $A_{r'}\subset E$, in such a way that $\{
0\}\times A_r \to A_{r'}$ is a regular covering of order $k$. Such
an immersion is (or can be chosen as) an almost embedding: the
exceptional subset $I\subset {\mathbb D}^{n-1}$, outside of which
the map is an embedding, corresponds to leaves which intersect the
transversal, over which the holonomy is computed, at points whose
holonomy orbit has cardinality strictly less than $k$. This is an
analytic subset of the transversal. Such an almost embedding will be
called {\it adapted to} $E$.

We shall use the following terminology: a meromorphic map is a {\it
meromorphic immersion} if it is an immersion outside its
indeterminacy set.

\begin{definition}\label{vanishing}
Let $E\subset L_p^0$ be a parabolic end with finite holonomy, of
order $k\ge 1$. Then $E$ is a {\bf vanishing end}, of order $k$, if
there exists an almost embedding $f: {\mathbb D}^{n-1}\times A_r \to
X^0\subset X$ adapted to $E$ such that:
\begin{enumerate}
\item[(i)] $f({\mathbb D}^{n-1}\times A_r)$ extends to a meromorphic
family of discs $g:W\dashrightarrow X$;
\item[(ii)] $g$ is a meromorphic immersion.
\end{enumerate}
\end{definition}

In other words, $E$ is a vanishing end if, firstly, it can be
compactified in $X$ to a disc, by adding a singular point of
${\mathcal F}$, and, secondly, this disc-compactification can be
meromorphically and immersively deformed to nearby leaves, up to a
ramification given by the holonomy. This definition mimics, in our
context, the classical definition of vanishing cycle for real
codimension one foliations \cite{CLN}.

\begin{remark} {\rm If $g:W\dashrightarrow X$ is as in Definition
\ref{vanishing}, then the indeterminacy set $F=Indet(g)$ cuts each
fiber $W_z$, $z\in{\mathbb D}^{n-1}$, along a finite subset
$F_z\subset W_z$. The restricted map $g_z : W_z\to X$ sends
$W_z\setminus F_z$ into a leaf of ${\mathcal F}^0$, in an immersive
way, and $F_z$ into $Sing({\mathcal F})$. Each point of $F_z$
corresponds to a parabolic end of $W_z\setminus F_z$, which is sent
by $g_z$ to a parabolic end of a leaf; clearly, this parabolic end
is a vanishing one (whose order, however, may be smaller than $k$),
and the corresponding meromorphic family of discs is obtained by
restricting $g$. Remark also that, $F$ being of codimension at least
2, we have $F_z=\emptyset$ for every $z$ outside an analytic subset
of ${\mathbb D}^{n-1}$ of positive codimension. This means (as we
shall see better below) that ``most'' leaves have no vanishing end.}
\end{remark}

If $E\subset L_p^0$ is a vanishing end of order $k$, then we
compactify it by adding one point, i.e. by prolonging
$\overline{\mathbb D}^*$ to $\overline{\mathbb D}$. But we do such a
compactification in an {\it orbifold sense}: the added point has, by
definition, a {\it multiplicity} equal to $k$. By doing such a
end-compactification for every vanishing end of $L_p^0$, we finally
obtain a connected curve (with orbifold structure) $L_p$, which is
by definition the {\bf leaf} of ${\mathcal F}$ through $p$. The
initial inclusion $i_p^0 : L_p^0\to X^0$ can be extended to a
holomorphic map
$$i_p : L_p \rightarrow X$$
which sends the discrete subset $L_p\setminus L_p^0$ into
$Sing({\mathcal F})$. Note that $i_p$ may fail to be immersive at
those points. Moreover, it may happen that two different points of
$L_p\setminus L_p^0$ are sent by $i_p$ to the same singular point of
${\mathcal F}$ (see Example \ref{flop} below). In spite of this, we
shall sometimes identify $L_p$ with its image in $X$. For instance,
to say that a map $f:Z\to X$ ``has values into $L_p$'' shall mean
that $f$ factorizes through $i_p$.

Remark that we have not defined, and shall not define, leaves $L_p$
through $p\in Sing({\mathcal F})$: a leaf may pass through
$Sing({\mathcal F})$, but its basepoint must be chosen outside
$Sing({\mathcal F})$.

Let us see two examples.

\begin{example} \label{elliptic} {\rm
Take a compact K\"ahler surface $S$ foliated by an elliptic
fibration $\pi : S\to C$, and let $c_0\in C$ be such that the fiber
$F_0 = \pi^{-1}(c_0)$ is of Kodaira type $II$ \cite[V.7]{BPV}, i.e.
a rational curve with a cusp $q$. If $p\in F_0$, $p\not= q$, then
the leaf $L_p^0$ is equal to $F_0\setminus\{ q\}\simeq{\mathbb C}$.
This leaf has a parabolic end with trivial holonomy, which is {\it
not} a vanishing end. Indeed, this end can be compactified to a
cuspidal disc, which however cannot be meromorphically deformed {\it
as a disc} to nearby leaves, because nearby leaves have positive
genus close to $q$. Hence $L_p = L_p^0$.

Let now $\widetilde S\to S$ be the composition of three blow-ups
which transforms $F_0$ into a tree of four smooth rational curves
$\widetilde F_0 = G_1+G_2+G_3+G_6$ of respective selfintersections
$-1, -2, -3, -6$ \cite[V.10]{BPV}. Let $\widetilde\pi : \widetilde
S\to C$ be the new elliptic fibration/foliation. Set $p_j=G_1\cap
G_j$, $j=2,3,6$. If $p\in G_1$ is different from those three points,
then $L_p^0 = G_1\setminus\{ p_2,p_3,p_6\}$. The parabolic end of
$L_p^0$ corresponding to $p_2$ (resp. $p_3$, $p_6$) has holonomy of
order 2 (resp. 3, 6). This time, this is a vanishing end: a disc $D$
in $G_1$ through $p_2$ (resp. $p_3$, $p_6$) ramified at order 2
(resp. 3, 6) can be deformed to nearby leaves as discs close to
$2D+G_2$ (resp. $3D+G_3$, $6D+G_6$), and also the ``meromorphic
immersion'' condition can be easily respected. Thus $L_p$ is
isomorphic to the orbifold ``${\mathbb P}$ with three points of
multiplicity 2, 3, 6''. Note that the universal covering (in
orbifold sense) of $L_p$ is isomorphic to ${\mathbb C}$, and the
holonomy covering (defined below) is a smooth elliptic curve.

Finally, if $p\in G_j$, $p\not= p_j$, $j=2,3,6$, then $L_p^0$ has a
parabolic end with trivial holonomy, which is not a vanishing end,
and so $L_p=L_p^0\simeq{\mathbb C}$.

A more systematic analysis of the surface case, from a slightly
different point of view, can be found in \cite{Br1}.}
\end{example}

\begin{example} \label{flop} {\rm
Take a projective threefold $M$ containing a smooth rational curve
$C$ with normal bundle $N_C = {\mathcal O}(-1)\oplus {\mathcal
O}(-1)$. Take a foliation ${\mathcal F}$ on $M$, nonsingular around
$C$, such that: (i) for every $p\in C$, $T_p{\mathcal F}$ is
different from $T_pC$; (ii) $T_{\mathcal F}$ has degree -1 on $C$.
It is easy to see that there are a lot of foliations on $M$
satisfying these two requirements. Note that, on a neighbourhood of
$C$, we can glue together the local leaves (discs) of ${\mathcal F}$
through $C$, and obtain a smooth surface $S$ containing $C$;
condition (ii) means that the selfintersection of $C$ in $S$ is
equal to -1.

\begin{figure}[h]
\includegraphics[width=10cm,height=5cm]{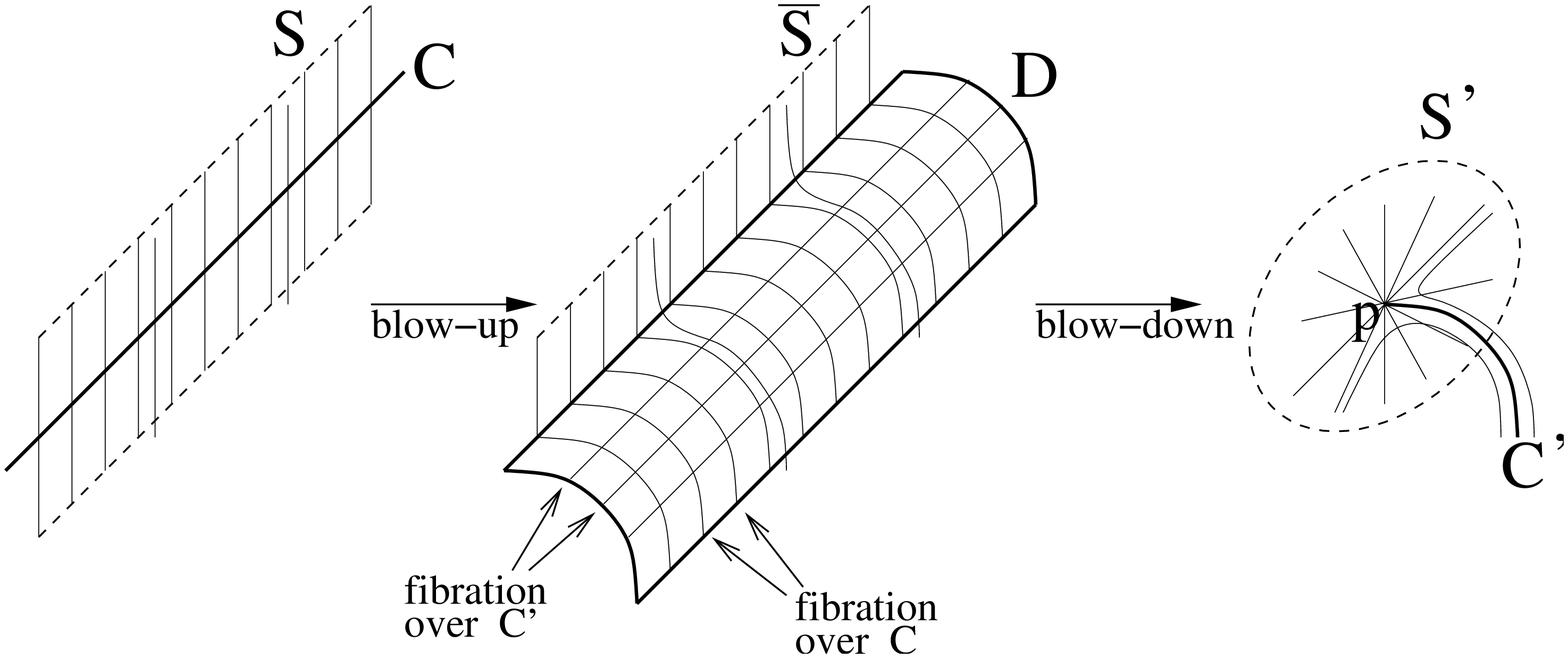}
\end{figure}

We now perform a {\it flop} of $M$ along $C$. That is, we firstly
blow-up $M$ along $C$, obtaining a threefold $\widetilde M$
containing an exceptional divisor $D$ naturally ${\mathbb
P}$-fibered over $C$. Because $N_C = {\mathcal O}(-1)\oplus
{\mathcal O}(-1)$, this divisor $D$ is in fact isomorphic to
${\mathbb P}\times{\mathbb P}$, hence it admits a second ${\mathbb
P}$-fibration, transverse to the first one. Each fibre of this
second fibration can be blow-down to a point (Moishezon's criterion
\cite{Moi}), and the result is a smooth threefold $M'$, containing a
smooth rational curve $C'$ with normal bundle $N_{C'}={\mathcal
O}(-1)\oplus {\mathcal O}(-1)$, over which $D$ fibers. (At this
point, $M'$ could be no more projective, nor K\"ahler, but this is
not an important fact in this example). The strict transform
$\overline S$ of $S$ in $\widetilde M$ cuts the divisor $D$ along
one of the fibers of the second fibration $D\to C'$, by condition
(ii) above, therefore its image $S'$ in $M'$ is a bidimensional disc
which cuts $C'$ transversely at some point $p$.

Let us look at the transformed foliation ${\mathcal F}'$ on $M'$.
The point $p$ is a singular point of ${\mathcal F}'$, the only one
on a neighbourhood of $C'$. The curve $C'$ is invariant by
${\mathcal F}'$. The surface $S'$ is tangent to ${\mathcal F}'$, and
over it the foliation has a radial type singularity. In fact, in
appropriate coordinates around $p$ the foliation is generated by the
vector field $x\frac{\partial}{\partial x} +
y\frac{\partial}{\partial y} - z\frac{\partial}{\partial z}$, with
$S'=\{ z=0\}$ and $C'=\{ x=y=0\}$.

If $L^0$ is a leaf of $({\mathcal F}')^0$, then each component $D^0$
of $L^0\cap S'$ is a parabolic end converging to $p$. It is a
vanishing end, of order 1: the meromorphic family of discs of
Definition \ref{vanishing} is obviously constructed from a flow box
of ${\mathcal F}$, around a suitably chosen point of $C$. Generic
fibers of this family are sent to discs in $M'$ close to $D^0\cup
C'$; other fibers are sent to discs in $S'$ passing through $p$, and
close to $D^0\cup\{ p\}$. Remark that it can happen that $L^0\cap
S'$ has several, or even infinitely many, components; in that case
the map $i:L\to M'$ sends several, or even infinitely many, points
to the same $p\in M'$.}
\end{example}

Having defined the leaf $L_p$ through $p\in X^0$, we can now define
its {\bf holonomy covering} $\widehat{L_p}$ and its {\bf universal
covering} $\widetilde{L_p}$. The first one is the covering defined
by the Kernel of the holonomy. More precisely, we start with the
usual holonomy covering $(\widehat{L_p^0},p)\to (L_p^0,p)$ with
basepoint $p$ (it is useful to think to $\widehat{L_p^0}$ as
equivalence classes of paths in $L_p^0$ starting at $p$, so that the
basepoint $p\in\widehat{L_p^0}$ is the class of the constant path).
If $E\subset L_p^0$ is a vanishing end of order $k$, then its
preimage in $\widehat{L_p^0}$ is a (finite or infinite) collection
of parabolic ends $\{ \widehat{E_j}\}$, each one regularly covering
$E$ with degree $k$. Each such map $\overline{\mathbb D}^*\simeq
\widehat{E_j} \to \overline{\mathbb D}^*\simeq E$ can be extended to
a map $\overline{\mathbb D} \to \overline{\mathbb D}$, with a
ramification at 0 of order $k$. By definition, $\widehat{L_p}$ is
obtained by compactifying all these parabolic ends of
$\widehat{L_p^0}$, over all the vanishing ends of $L_p^0$.
Therefore, we have a covering map
$$(\widehat{L_p},p) \rightarrow (L_p,p)$$
which ramifies over $L_p\setminus L_p^0$. However, from the orbifold
point of view such a map is a {\it regular} covering: $w=z^k$ is a
regular covering if $z=0$ has multiplicity 1 and $w=0$ has
multiplicity $k$. Note that we do not need anymore a orbifold
structure on $\widehat{L_p}$, in the sense that all its points have
multiplicity 1.

In a more algebraic way, the orbifold fundamental group
$\pi_1(L_p,p)$ is a quotient of $\pi_1(L_p^0,p)$, through which the
holonomy representation $hol_p$ factorizes. Then $\widehat{L_p}$ is
the covering defined by the Kernel of this representation of
$\pi_1(L_p,p)$ into $Diff({\mathbb D}^{n-1},0)$.

The universal covering $\widetilde{L_p}$ can be now defined as the
universal covering of $\widehat{L_p}$, or equivalently as the
universal covering, in orbifold sense, of $L_p$. We then have
natural covering maps
$$(\widetilde{L_p},p) \rightarrow (\widehat{L_p},p)
\rightarrow (L_p,p).$$ Recall that there are few exceptional
orbifolds (teardrops) which do not admit a universal covering. It is
a pleasant fact that in our context such orbifolds do not appear.

\subsection{Holonomy tubes}

We now analyze how the maps $p\mapsto \widehat{L_p}$ and then
$p\mapsto \widetilde{L_p}$ depend on $p$.

Let $T\subset X^0$ be a local transversal to ${\mathcal F}^0$.

\begin{proposition} \label{holonomytube}
There exists a complex manifold $V_T$ of dimension $n$, a
holomorphic submersion
$$Q_T : V_T \rightarrow T,$$
a holomorphic section
$$q_T : T \rightarrow V_T,$$
and a meromorphic immersion
$$\pi_T : V_T \dashrightarrow X$$
such that:
\begin{enumerate}
\item[(i)] for every $t\in T$, the pointed fiber $(Q_T^{-1}(t),
q_T(t))$ is isomorphic to $(\widehat{L_t},t)$;
\item[(ii)] the indeterminacy set $Indet(\pi_T)$ of $\pi_T$ cuts
each fiber $Q_T^{-1}(t)=\widehat{L_t}$ along the discrete subset
$\widehat{L_t}\setminus \widehat{L_t^0}$;
\item[(iii)] for every $t\in T$, the restriction of $\pi_T$ to
$Q_T^{-1}(t) = \widehat{L_t}$ coincides, after removal of
indeterminacies, with the holonomy covering $\widehat{L_t}\to L_t
\buildrel i_t\over\to i_t(L_t)\subset X$.
\end{enumerate}
\end{proposition}

\begin{proof}
We firstly prove a similar statement for the regular foliation
${\mathcal F}^0$ on $X^0$. We use Il'yashenko's methodology
\cite{Il1}; an alternative but equivalent one can be found in
\cite{Suz}, we have already seen it at the beginning of the proof of
Theorem \ref{ilyashenko}. In fact, in the case of a regular
foliation the construction of $V_T^0$ below is a rather classical
fact in foliation theory, which holds in the much more general
context of smooth foliations with real analytic holonomy.

Consider the space $\Omega_T^{{\mathcal F}^0}$ composed by
continuous paths $\gamma : [0,1]\to X^0$ tangent to ${\mathcal F}^0$
and such that $\gamma (0)\in T$, equipped with the uniform topology.
On $\Omega_T^{{\mathcal F}^0}$ we put the following equivalence
relation: $\gamma_1 \sim \gamma_2$ if $\gamma_1(0)=\gamma_2(0)$,
$\gamma_1(1)=\gamma_2(1)$, and the loop $\gamma_1\ast\gamma_2^{-1}$,
obtained by juxtaposing $\gamma_1$ and $\gamma_2^{-1}$, has trivial
holonomy.

Set $$V_T^0 = \Omega_T^{{\mathcal F}^0}\big/ \sim$$ with the
quotient topology. Note that we have natural continuous maps
$$Q_T^0 : V_T^0 \rightarrow T$$
and
$$\pi_T^0 : V_T^0 \rightarrow X^0$$
defined respectively by $[\gamma ]\mapsto \gamma (0)\in T$ and
$[\gamma ]\mapsto \gamma (1)\in X^0$. We also have a natural section
$$q_T^0 : T \rightarrow V_T^0$$
which associates to $t\in T$ the equivalence class of the constant
path at $t$. Clearly, for every $t\in T$ the pointed fiber
$((Q_T^0)^{-1}(t),q_T^0(t))$ is the same as $(\widehat{L_t^0},t)$,
by the very definition of holonomy covering, and $\pi_T^0$
restricted to that fiber is the holonomy covering map. Therefore, we
just have to find a complex structure on $V_T^0$ such that all these
maps become holomorphic.

We claim that $V_T^0$ is a Hausdorff space. Indeed, if
$[\gamma_1],[\gamma_2]\in V_T^0$ are two nonseparated points, then
$\gamma_1(0) = \gamma_2(0)=t$, $\gamma_1(1) = \gamma_2(1)$, and the
loop $\gamma_1\ast\gamma_2^{-1}$ in the leaf $L_t^0$ can be
uniformly approximated by loops $\gamma_{1,n}\ast\gamma_{2,n}^{-1}$
in the leaves $L_{t_n}^0$ ($t_n\to t$) with trivial holonomy (so
that $[\gamma_{1,n}]=[\gamma_{2,n}]$ is a sequence of points of
$V_T^0$ converging to both $[\gamma_1]$ and $[\gamma_2]$). But this
implies that also the loop $\gamma_1\ast\gamma_2^{-1}$ has trivial
holonomy, by the identity principle: if $h\in Diff({\mathbb
D}^{n-1},0)$ is the identity on a sequence of open sets accumulating
to $0$, then $h$ is the identity everywhere. Thus $[\gamma_1] =
[\gamma_2]$, and $V_T^0$ is Hausdorff.

Now, note that $\pi_T^0 : V_T\to X^0$ is a local homeomorphism.
Hence we can pull back to $V_T^0$ the complex structure of $X^0$,
and in this way $V_T^0$ becomes a complex manifold of dimension $n$
with all the desired properties. Remark that, at this point,
$\pi_T^0$ has not yet indeterminacy points, and so $V_T^0$ is a
so-called Riemann Domain over $X^0$.

In order to pass from $V_T^0$ to $V_T$, we need to add to each fiber
$\widehat{L_t^0}$ of $V_T^0$ the discrete set
$\widehat{L_t}\setminus\widehat{L_t^0}$.

Take a vanishing end $E\subset L_t^0$, of order $k$, let $f:{\mathbb
D}^{n-1}\times A_r\to X^0$ be an almost embedding adapted to $E$,
and let $g:W\dashrightarrow X$ be a meromorphic family of discs
extending $f$, immersive outside $F=Indet(g)$. Take also a parabolic
end $\widehat E\subset \widehat{L_t^0}$ projecting to $E$, with
degree $k$. By an easy holonomic argument, the immersion
$g\vert_{W\setminus F} : W\setminus F \rightarrow X^0$ can be lifted
to $V_T^0$, as a proper embedding
$$\widetilde g : W\setminus F \to V_T^0$$
which sends the central fiber $W_0\setminus F_0$ to $\widehat{E}$.
Each fiber $W_z\setminus F_z$ is sent by $\widetilde g$ to a closed
subset of a fiber $\widehat{L_{t(z)}^0}$, and each point of $F_z$
corresponds to a parabolic end of $\widehat{L_{t(z)}^0}$ projecting
to a vanishing end of $L_{t(z)}^0$.

Now we can glue $W$ to $V_T^0$ using $\widetilde g$: this
corresponds to compactify all parabolic ends of fibers of $V_T^0$
which project to vanishing ends and which are close to
$\widehat{E}$. By doing this operation for every $E$ and
$\widehat{E}$, we finally construct our manifold $V_T$, fibered over
$T$ with fibers $\widehat{L_t}$. The map $\pi_T$ extending
(meromorphically) $\pi_T^0$ is then deduced from the maps $g$ above.
\end{proof}

The manifold $V_T$ will be called {\bf holonomy tube} over $T$. The
meromorphic immersion $\pi_T$ is, of course, {\it very complicated}:
it contains all the dynamics of the foliation, so that it is,
generally speaking, very far from being, say, finite-to-one. Note,
however, that most fibers do not cut the indeterminacy set of
$\pi_T$, so that $\pi_T$ sends that fibers to leaves of ${\mathcal
F}^0$; moreover, most leaves have trivial holonomy (it is a general
fact \cite{CLN} that leaves with non trivial holonomy cut any
transversal along a thin subset), and so on most fibers $\pi_T$ is
even an isomorphism between the fiber and the corresponding leaf of
${\mathcal F}^0$. But be careful: a leaf may cut a transversal $T$
infinitely many times, and so $V_T$ will contain infinitely many
fibers sent by $\pi_T$ to the same leaf, as holonomy coverings
(possibly trivial) with different basepoints.

\subsection{Covering tubes}

The following proposition is similar, in spirit, to Proposition
\ref{holonomytube}, but, as we shall see, its proof is much more
delicate. Here the K\"ahler assumption becomes really indispensable,
via the unparametrized Hartogs extension lemma. Without the K\"ahler
hypothesis it is easy to find counterexamples (say, for foliations
on Hopf surfaces).

\begin{proposition}\label{coveringtube}
There exists a complex manifold $U_T$ of dimension $n$, a
holomorphic submersion
$$P_T : U_T \rightarrow T,$$
a holomorphic section
$$p_T : T \rightarrow U_T,$$
and a surjective holomorphic immersion
$$F_T : U_T \rightarrow V_T$$
such that:
\begin{enumerate}
\item[(i)] for every $t\in T$, the pointed fiber $(P_T^{-1}(t),
p_T(t))$ is isomorphic to $(\widetilde{L_t},t)$;
\item[(ii)] for every $t\in T$, $F_T$ sends the fiber
$(\widetilde{L_t},t)$ to the fiber $(\widehat{L_t},t)$, as universal
covering.
\end{enumerate}
\end{proposition}

\begin{proof}
We use the same methodology as in the first part of the previous
proof, with ${\mathcal F}^0$ replaced by the fibration $V_T$ and
$\Omega_T^{{\mathcal F}^0}$ replaced by $\Omega_T^{V_T} =$ space of
continuous paths $\gamma :[0,1]\to V_T$ tangent to the fibers and
starting from $q_T(T)\subset V_T$. But now the equivalence relation
$\sim$ is given by homotopy, not holonomy: $\gamma_1\sim\gamma_2$ if
they have the same extremities and the loop
$\gamma_1\ast\gamma_2^{-1}$ is homotopic to zero in the fiber
containing it. The only thing that we need to prove is that the
quotient space
$$U_T = \Omega_T^{V_T}\big/ \sim$$
is Hausdorff; then everything is completed as in the previous proof,
with $F_T$ associating to a homotopy class of paths its holonomy
class. The Hausdorff property can be spelled as follows
(``nonexistence of vanishing cycles''):
\begin{enumerate}
\item[(*)] if $\gamma : [0,1]\to \widehat{L_t}\subset V_T$ is a loop
(based at $q_T(t)$) uniformly approximated by loops $\gamma_n :
[0,1]\to \widehat{L_{t_n}}\subset V_T$ (based at $q_T(t_n)$)
homotopic to zero in $\widehat{L_{t_n}}$, then $\gamma$ is homotopic
to zero in $\widehat{L_t}$.
\end{enumerate}

Let us firstly consider the case in which $\gamma$ is a simple loop.
We may assume that $\Gamma = \gamma ([0,1])$ is a real analytic
curve in $\widehat{L_t}$, and we may find an embedding
$$f:{\mathbb D}^{n-1}\times A_r\to V_T$$
sending fibers to fibers and such that $\Gamma = f(0,\partial A_r)$.
Thus $\Gamma_n = f(z_n,\partial A_r)$ is homotopic to zero in its
fiber, for some sequence $z_n\to 0$. For evident reasons, if $z_n'$
is sufficiently close to $z_n$, then also $f(z_n',\partial A_r)$ is
homotopic to zero in its fiber. Thus, we have an open nonempty
subset $U\subset {\mathbb D}^{n-1}$ such that, for every $z\in U$,
$f(z,\partial A_r)$ is homotopic to zero in its fiber. Denote by
$D_z$ the disc in the fiber bounded by such $f(z,\partial A_r)$.

We may also assume that $\Gamma$ is disjoint from the discrete
subset $\widehat{L_t}\setminus\widehat{L_t^0}$, so that, after
perhaps restricting ${\mathbb D}^{n-1}$, the composite map
$$f' : \pi_T \circ f : {\mathbb D}^{n-1}\times A_r \rightarrow X$$
is holomorphic, and therefore it is an almost embedding. We already
know that, for every $z\in U$, $f'(z,A_r)$ extends to a disc, image
by $\pi_T$ of $D_z$. Therefore, by Theorem \ref{hartogs},
$f'({\mathbb D}^{n-1}\times A_r)$ extends to a meromorphic family of
discs
$$g : W \dashrightarrow X.$$
It may be useful to observe that such a $g$ is a meromorphic
immersion. Indeed, setting $F=Indet(g)$, the set of points of
$W\setminus F$ where $g$ is not an immersion is (if not empty) an
hypersurface. Such a hypersurface cannot cut a neighbourhood of the
boundary $\partial W$, where $g$ is a reparametrization of the
immersion $f'$. Also, such a hypersurface cannot cut the fiber $W_z$
when $z\in U$ is generic (i.e. $W_z\cap F =\emptyset$), because on a
neighbourhood of such a $W_z$ the map $g$ is a reparametrization of
the immersion $\pi_T$ on a neighbourhood of $D_z$. It follows that
such a hypersurface is empty.

As in the proof of Proposition \ref{holonomytube},
$g\vert_{W\setminus F}$ can be lifted, holomorphically, to $V_T^0$,
and then $g$ can be lifted to $V_T$, giving an embedding $\widetilde
g : W\to V_T$. Then $\widetilde g (W_0)$ is a disc in
$\widehat{L_t}$ with boundary $\Gamma$, and consequently $\gamma$ is
homotopic to zero in the fiber $\widehat{L_t}$.

Consider now the case in which $\gamma$ is possibly not simple. We
may assume that $\gamma$ is a smooth immersion with some points of
transverse selfintersection, and idem for $\gamma_n$. We reduce to
the previous simple case, by a purely topological argument.

Take the immersed circles $\Gamma =\gamma ([0,1])$ and $\Gamma_n
=\gamma_n([0,1])$. Let $R_n\subset \widehat{L_{t_n}}$ be the open
bounded subset obtained as the union of a small tubular
neighbourhood of $\Gamma_n$ and all the bounded components of
$\widehat{L_{t_n}}\setminus\Gamma_n$ isomorphic to the disc. Thus,
each connected component of $R_n\setminus\Gamma_n$ is either a disc
with boundary in $\Gamma_n$ (union of arcs between selfintersection
points), or an annulus with one boundary component in $\Gamma_n$ and
another one in $\partial R_n$; this last one is not the boundary of
a disc in $\widehat{L_{t_n}}\setminus R_n$. We have the following
elementary topological fact: if $\Gamma_n$ is homotopic to zero in
$\widehat{L_{t_n}}$, then it is homotopic to zero also in $R_n$.

Let $R\subset\widehat{L_t}$ be defined in a similar way, starting
from $\Gamma$. By the first part of the proof, if $D_n\subset
R_n\setminus\Gamma_n$ is a disc with boundary in $\Gamma_n$, then
for $t_n\to t$ such a disc converges to a disc
$D\subset\widehat{L_t}$ with boundary in $\Gamma$, i.e. to a disc
$D\subset R\setminus\Gamma$. Conversely, but by elementary reasons,
any disc $D\subset R\setminus\Gamma$ with boundary in $\Gamma$ can
be deformed to discs $D_n\subset R_n\setminus\Gamma_n$ with
boundaries in $\Gamma_n$. We deduce that $R$ is diffeomorphic to
$R_n$, or more precisely that the pair $(R,\Gamma )$ is
diffeomorphic to the pair $(R_n,\Gamma_n)$, for $n$ large. Hence
from $\Gamma_n$ homotopic to zero in $R_n$ we infer $\Gamma$
homotopic to zero in $R$, and a fortiori in $\widehat{L_t}$. This
completes the proof of the Hausdorff property (*).
\end{proof}

The manifold $U_T$ will be called {\bf covering tube} over $T$. We
have a meromorphic immersion
$$\Pi_T = \pi_T\circ F_T : U_T \dashrightarrow X$$
whose indeterminacy set $Indet(\Pi_T)$ cuts each fiber
$P_T^{-1}(t)=\widetilde{L_t}$ along the discrete subset which is the
preimage of $\widehat{L_t}\setminus\widehat{L_t^0}$ under the
covering map $\widetilde{L_t}\to\widehat{L_t}$. For every $t\in T$,
the restriction of $\Pi_T$ to $P^{-1}(t)$ coincides, after removal
of indeterminacies, with the universal covering $\widetilde{L_t}\to
L_t \buildrel i_t\over\to i_t(L_t)\subset X$.

The local biholomorphism $F_T : U_T\to V_T$ is a fiberwise covering,
but globally it may have a quite wild structure. Let us see two
examples.

\begin{example} \label{elliptic2}{\rm
We take again the elliptic fibration $S\buildrel\pi\over\to C$ of
Example \ref{elliptic}. Let $T\subset S$ be a small transverse disc
centered at $t_0\in F_0\setminus\{ q\}$. Then, because the holonomy
is trivial, $\widehat{L_t} = L_t$ for every $t$. We have already
seen that $L_{t_0}=F_0\setminus\{ q\}$, and obviously for $t\not=
t_0$, $L_t$ is the smooth elliptic curve through $t$. The covering
tube $V_T$ is simply $\pi^{-1}(\pi (T))\setminus\{ q\}$. Remark that
its central fiber is simply connected, but the other fibers are not.
All the fibers of $U_T$ are isomorphic to ${\mathbb C}$ (in fact,
one can see that $U_T\simeq T\times{\mathbb C}$). The map
$F_T:U_T\to V_T$, therefore, is injective on the central fiber, but
not on the other ones.

\begin{figure}[h]
\includegraphics[width=10cm,height=6cm]{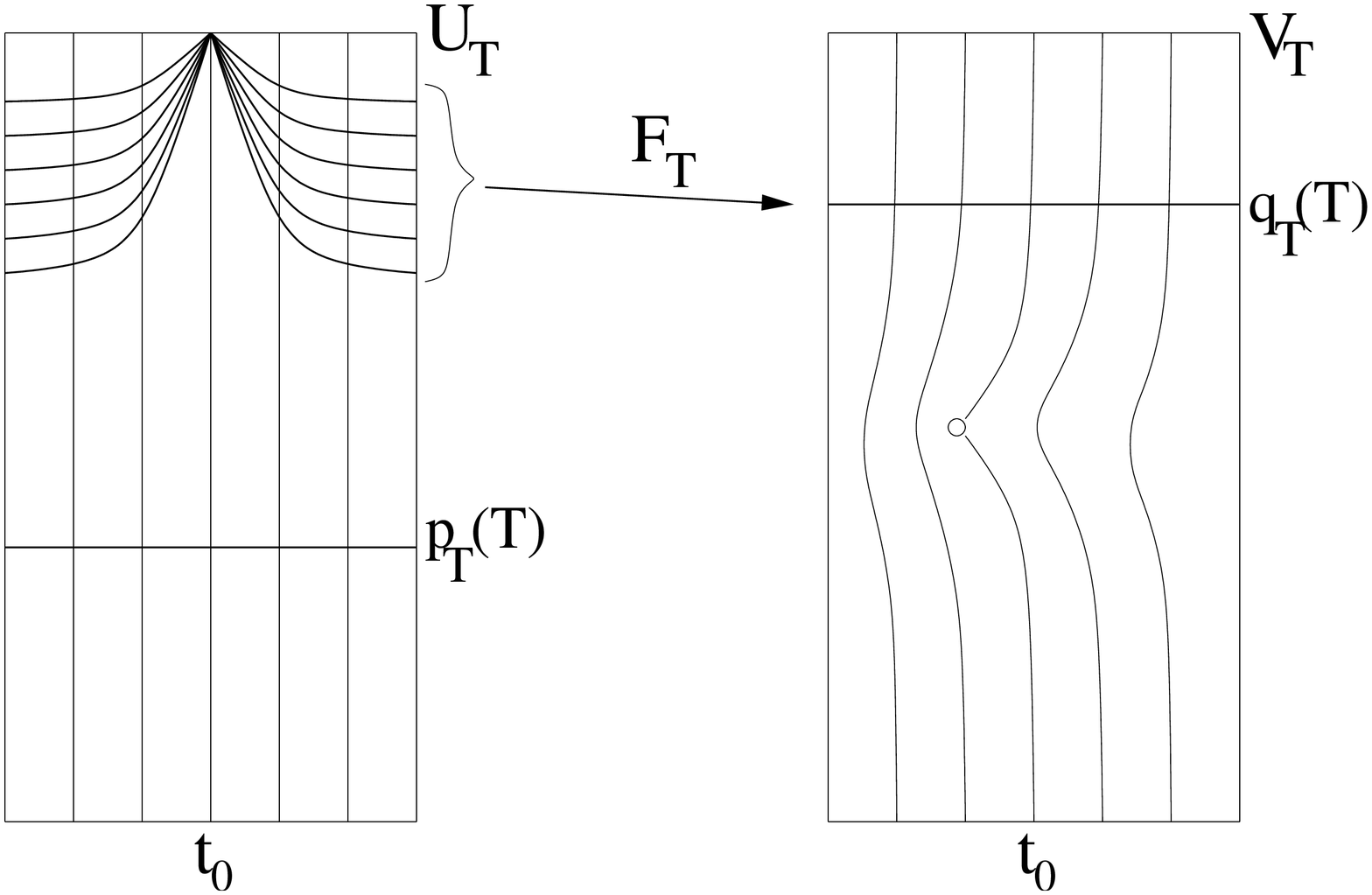}
\end{figure}

To see better what is happening, take the basepoints $q_T(T) \subset
V_T$ and consider the preimage $F_T^{-1}(q_T(T))\subset U_T$. This
preimage has infinitely many components: one of them is
$p_T(T)\subset U_T$, and each other one is the graph over
$T\setminus\{ t_0\}$ of a 6-valued section of $U_T$. This follows
from the fact that the monodromy of the elliptic fibration around a
fiber of type $II$ has order 6 \cite[V.10]{BPV}. The map $F_T$ sends
this 6-valued graph to $q_T(T\setminus\{ t_0\} )$, as a regular
6-fold covering. There is a ``virtual'' ramification of order 6 over
$q_T(t_0)$, which is however pushed-off $U_T$, to the point at
infinity of the central fiber.}
\end{example}

\begin{example} \label{elliptic3}{\rm
We take again an elliptic fibration $S\buildrel\pi\over\to C$, but
now with a fiber $\pi^{-1}(c_0) = F_0$ of Kodaira type $I_1$, i.e. a
rational curve with a node $q$. As before, $V_T$ coincides with
$\pi^{-1}(\pi (T))\setminus\{ q\}$, but now the central fiber is
isomorphic to ${\mathbb C}^*$. Again $U_T\simeq T\times{\mathbb C}$.
The map $F_T:U_T\to V_T$ is a ${\mathbb Z}$-covering over the
central fiber, a ${\mathbb Z}^2$-covering over the other fibers. The
preimage of $q_T(T)$ by $F_T$ has still infinitely many components.
One of them is $p_T(T)$. Some of them are graphs of (1-valued)
sections over $T$, passing through the (infinitely many) points of
$F_T^{-1}(q_T(t_0))$. But most of them are graphs of $\infty$-valued
sections over $T\setminus\{ t_0\}$ (like the graph of the
logarithm). Indeed, the monodromy of the elliptic fibration around a
fiber of type $I_1$ has infinite order \cite[V.10]{BPV}. If
$t\not=t_0$, then $F_T^{-1}(q_T(t))$ is a lattice in
$\widetilde{L_t}\simeq{\mathbb C}$, with generators $1$ and $\lambda
(t)\in{\mathbb H}$. For $t\to t_0$, this second generator diverges
to $+i\infty$, and the lattices reduces to ${\mathbb Z} =
F_T^{-1}(q_T(t_0))$. The monodromy acts as $(n,m\lambda (t))\mapsto
(n+m,m\lambda (t))$. Then each connected component of
$F_T^{-1}(q_T(T))$ intersects $\widetilde{L_t}$ either at a single
point $(n,0)$, fixed by the monodromy, or along an orbit
$(n+m{\mathbb Z}, m\lambda (t))$, $m\not=0$.}
\end{example}

More examples concerning elliptic fibrations can be found in
\cite{Br4}.

\begin{remark} {\rm
As we recalled in Section 2, similar constructions of $U_T$ and
$V_T$ have been done, respectively, by Il'yashenko \cite{Il1} and
Suzuki \cite{Suz}, in the case where the ambient manifold $X$ is a
Stein manifold. However, the Stein case is much simpler than the
compact K\"ahler one. Indeed, the meromorphic maps
$g:W\dashrightarrow X$ with which we work are automatically {\it
holomorphic} if $X$ is Stein. Thus, in the Stein case there are no
vanishing ends, i.e. $L_p=L_p^0$ for every $p$ and leaves of
${\mathcal F}$ = leaves of ${\mathcal F}^0$. Then the maps $\pi_T$
and $\Pi_T$ are {\it holomorphic} immersions of $V_T$ and $U_T$ into
$X^0$ (and so $V_T$ and $U_T$ are Riemann Domains over $X^0$). Also,
our unparametrized Hartogs extension lemma still holds in the Stein
case, but with a much simpler proof, because we do not need to worry
about ``rational bubbles'' arising in Bishop's Theorem.

In fact, there is a common framework for the Stein case and the
compact K\"ahler case: the framework of {\it holomorphically convex}
(not necessarily compact) K\"ahler manifolds. Indeed, the only form
of compactness that we need, in this Section and also in the next
one, is the following: for every compact $K\subset X$, there exists
a (larger) compact $\hat K\subset X$ such that every holomorphic
disc in $X$ with boundary in $K$ is fully contained in $\hat K$.
This property is obviously satisfied by any holomorphically convex
K\"ahler manifold, with $\hat K$ equal to the usual holomorphically
convex hull of $K$.}
\end{remark}

A more global point of view on holonomy tubes and covering tubes
will be developed in the last Section.

\subsection{Rational quasi-fibrations}

We conclude this section with a result which can be considered as an
analog, in our context, of the classical Reeb Stability Theorem for
real codimension one foliations \cite{CLN}.

\begin{proposition}\label{reeb}
Let $X$ be a compact connected K\"ahler manifold and let ${\mathcal
F}$ be a foliation by curves on $X$. Suppose that there exists a
rational leaf $L_p$ (i.e., $\widetilde{L_p}={\mathbb P}$). Then all
the leaves are rational. Moreover, there exists a compact connected
K\"ahler manifold $Y$, $\dim Y = \dim X -1$, a meromorphic map $B :
X\dashrightarrow Y$, and Zariski open and dense subsets $X_0\subset
X$, $Y_0\subset Y$, such that:
\begin{enumerate}
\item[(i)] $B$ is holomorphic on $X_0$ and $B(X_0)=Y_0$;
\item[(ii)] $B : X_0 \to Y_0$ is a proper submersive map, all of
whose fibers are smooth rational curves, leaves of ${\mathcal F}$.
\end{enumerate}
\end{proposition}

\begin{proof}
It is sufficient to verify that all the leaves are rational; then
the second part follows by standard arguments of complex analytic
geometry, see e.g. \cite{CaP}.

By connectivity, it is sufficient to prove that, given a covering
tube $U_T$, if some fiber is rational then all the fibers are
rational. We can work, equivalently, with the holonomy tube $V_T$.
Now, such a property was actually already verified in the proof of
Proposition \ref{coveringtube}, in the form of ``nonexistence of
vanishing cycles''. Indeed, the set of rational fibers of $V_T$ is
obviously open. To see that it is also closed, take a fiber
$\widehat{L_t}$ approximated by fibers $\widehat{L_{t_n}}\simeq
{\mathbb P}$. Take an embedded cycle $\Gamma\subset \widehat{L_t}$,
approximated by cycles $\Gamma_n\subset \widehat{L_{t_n}}$. Each
$\Gamma_n$ bounds in $\widehat{L_{t_n}}$ {\it two} discs, one on
each side. As in the proof of Proposition \ref{coveringtube}, we
obtain that $\Gamma$ also bounds in $\widehat{L_t}$ {\it two} discs,
one on each side. Hence $\widehat{L_t}$ is rational.
\end{proof}

Such a foliation will be called {\bf rational quasi-fibration}. A
meromorphic map $B$ as in Proposition \ref{reeb} is sometimes called
{\it almost holomorphic}, because the image of its indeterminacy set
is a proper subset of $Y$, of positive codimension, contained in
$Y\setminus Y_0$. If $\dim X=2$ then $B$ is necessarily holomorphic,
and the foliation is a rational fibration (with possibly some
singular fibers). In higher dimensions one may think that the
foliation is obtained from a rational fibration by a meromorphic
transformation which does not touch generic fibers (like flipping
along a codimension two subset).

Note that, as the proof shows, for a rational quasi-fibration every
holonomy tube and every covering tube is isomorphic to
$T\times{\mathbb P}$, provided that the transversal $T$ is
sufficiently small (every ${\mathbb P}$-fibration is locally
trivial).

There are certainly many interesting issues concerning rational
quasi-fibrations, but basically this is a chapter of Algebraic
Geometry. In the following, we shall forget about them, and we will
concentrate on foliations with parabolic and hyperbolic leaves.

\section{A convexity property of covering tubes}

Let $X$ be a compact K\"ahler manifold, of dimension $n$, and let
${\mathcal F}$ be a foliation by curves on $X$, different from a
rational quasi-fibration. Fix a transversal $T\subset X^0$ to
${\mathcal F}^0$, and consider the covering tube $U_T$ over $T$,
with projection $P_T:U_T\to T$, section $p_T:T\to U_T$, and
meromorphic immersion $\Pi_T:U_T\dashrightarrow X$. Each fiber of
$U_T$ is either ${\mathbb D}$ or ${\mathbb C}$.

We shall establish in this Section, following \cite{Br2} and
\cite{Br3}, a certain convexity property of $U_T$, which later will
allow us to apply to $U_T$ the results of Section 2 of Nishino and
Yamaguchi.

We fix also an embedded closed disc $S\subset T$
($S\simeq\overline{\mathbb D}$, and the embedding in $T$ is
holomorphic up to the boundary), and we denote by $U_S$, $P_S$,
$p_S$, $\Pi_S$ the corresponding restrictions. Set $\partial U_S =
P_S^{-1}(\partial S)$. We shall assume that $S$ satisfies the
following properties:
\begin{enumerate}
\item[(a)] $U_S$, as a subset of $U_T$, intersects $Indet(\Pi_T)$
along a discrete subset, necessarily equal to $Indet(\Pi_S)$, and
$\partial U_S$ does not intersect $Indet(\Pi_T)$;
\item[(b)] for every $z\in\partial S$, the area of the fiber
$P_S^{-1}(z)$ is infinite.
\end{enumerate}
In (b), the area is computed with respect to the pull-back by
$\Pi_S$ of the K\"ahler form $\omega$ of $X$. Without loss of
generality, we take $\omega$ real analytic. We will see later that
these assumptions (a) and (b) are ``generic'', in a suitable sense.

\begin{theorem}\label{convexity}
For every compact subset $K\subset \partial U_S$ there exists a real
analytic bidimensional torus $\Gamma\subset \partial U_S$ such that:
\begin{enumerate}
\item[(i)] $\Gamma$ is transverse to the fibers of $\partial
U_S\buildrel P_S\over\to \partial S$, and cuts each fiber
$P_S^{-1}(z)$, $z\in\partial S$, along a circle $\Gamma (z)$ which
bounds a disc $D(z)$ which contains $K\cap P_S^{-1}(z)$ and
$p_S(z)$;
\item[(ii)] $\Gamma$ is the boundary of a real analytic Levi-flat
hypersurface $M\subset U_S$, filled by a real analytic family of
holomorphic discs $D^\theta$, $\theta\in{\mathbb S}^1$; each
$D^\theta$ is the graph of a section $s^\theta : S\to U_S$,
holomorphic up to the boundary, with $s^\theta (\partial S)
\subset\Gamma$;
\item[(iii)] $M$ bounds in $U_S$ a domain $\Omega$, which cuts each
fiber $P_S^{-1}(z)$, $z\in S$, along a disc $\Omega (z)$ which
contains $p_S(z)$ ($\Omega (z) =D(z)$ when $z\in\partial S$).
\end{enumerate}
\end{theorem}

This statement should be understood as expressing a variant of
Hartogs-convexity \cite[II.2]{Ran}, in which the standard Hartogs
figure is replaced by $p_S(S)\cup (\cup_{z\in\partial S}D(z))$, and
its envelope is replaced by $\Omega$. By choosing a large compact
$K$, condition (i) says that $\cup_{z\in\partial S}D(z)$ almost fill
the lateral boundary $\partial U_S$; conditions (ii) and (iii) say
that the family of discs $D(z)$, $z\in\partial S$, can be pushed
inside $S$, getting a family of discs $\Omega (z)$, $z\in S$, in
such a way that the boundaries $\partial\Omega (z)$, $z\in S$, vary
with $z$ in a ``holomorphic'' manner (``variation analytique'' in
the terminology of \cite{Ya3}). It is a sort of ``geodesic''
convexity of $U_S$, in which the extremal points of the geodesic are
replaced by $\Gamma$ and the geodesic is replaced by $M$.

Theorem \ref{convexity} will be proved by solving a nonlinear
Riemann-Hilbert problem, see \cite{For} and \cite[Ch. 20]{AWe} and
reference therein for some literature on this subject. An important
difference with this classical literature, however, is that the
torus $\Gamma$ is not fixed a priori: we want just to prove that
{\it some} torus $\Gamma$, enclosing the compact $K$ as in (i), is
the boundary of a Levi-flat hypersurface $M$ as in (ii); we do not
pretend that {\it every} torus $\Gamma$ has such a property. Even
if, as we shall see below, we have a great freedom in the choice of
$\Gamma$.

We shall use the continuity method. The starting point is the
following special (but not so much) family of tori.

\begin{lemma} \label{family}
Given $K\subset\partial U_S$ compact, there exists a real analytic
embedding
$$F:\partial S\times\overline{\mathbb D}\rightarrow \partial U_S,$$
sending fibers to fibers, such that:
\begin{enumerate}
\item[(i)] $\partial S\times\{ 0\}$ is sent to $p_S(\partial
S)\subset\partial U_S$;
\item[(ii)] $\partial S\times\{ |w|=t \}$, $t\in (0,1]$, is sent
to a real analytic torus $\Gamma_t\subset \partial U_S$ transverse
to the fibers of $P_S$, so that for every $z\in\partial S$,
$\Gamma_t(z)=\Gamma_t\cap P_S^{-1}(z)$ is a circle bounding a disc
$D_t(z)$ containing $p_S(z)$;
\item[(iii)] $D_1(z)$ contains $K\cap P_S^{-1}(z)$, for every
$z\in\partial S$;
\item[(iv)] for every $t\in (0,1]$ the function
$${\bf a}_t : \partial S \rightarrow {\mathbb R}^+\quad , \qquad
{\bf a}_t(z) = {\rm area} (D_t(z))$$ is constant (the constant
depending on $t$, of course).
\end{enumerate}
\end{lemma}

\begin{proof}
Because the fibers over $\partial S$ have infinite area, we can
certainly find a smooth torus $\Gamma '\subset\partial U_S$ which
encloses $K$ and $p_S(\partial S)$, and such that all the discs
$D'(z)$, bounded by $\Gamma '(z)$, have the same area, say 1. We may
approximate $\Gamma '$ with a real analytic torus $\Gamma ''$; the
corresponding discs $D''(z)$ have now variable area, but close to 1,
say between 1 and $1+\varepsilon$.

For every $z\in\partial S$ we have in $\overline{D''(z)}\setminus
p_S(z)$ a canonical foliation by circles,  the standard circles
under the uniformisation $(\overline{D''(z)},p_S(z)) \simeq
(\overline{\mathbb D},0)$. For every $t\in (0,1]$, let $\Gamma_t(z)$
be the circle of that foliation which bounds a disc of area equal to
$t$. Then, because all the data ($\Gamma ''$, $\omega$,...) are real
analytic, the union $\Gamma_t = \cup_{z\in\partial S}\Gamma_t(z)$ is
a real analytic torus, and these tori glue together in a real
analytic way. If the initial perturbation is sufficiently small,
$\Gamma_1$ encloses $K$. And the function ${\bf a}_t$ is constantly
equal to $t$.
\end{proof}

Given $F$ as in Lemma \ref{family}, we shall say that a real
analytic embedding
$$G : S\times\overline{\mathbb D} \rightarrow U_S$$
is a {\bf Levi-flat extension} of $F$ if $G$ sends fibers to fibers
and:
\begin{enumerate}
\item[(i)] $G(S\times\{ 0\}) = p_S(S)$;
\item[(ii)] $G(S\times\{ |w|=t \})$, $t\in (0,1]$, is a real analytic
Levi-flat hypersurface $M_t\subset U_S$ with boundary $\Gamma_t$,
filled by graphs of holomorphic sections over $S$ with boundary
values in $\Gamma_t$.
\end{enumerate}
Our aim is to construct such a $G$. Then $\Gamma = \Gamma_1$ and
$M=M_1$ gives Theorem \ref{convexity}.

The continuity method consists in analyzing the set of those $t_0\in
(0,1]$ such that a similar $G$ can be constructed over
$S\times\overline{{\mathbb D}(t_0)}$. We need to show that this set
is nonempty, open and closed.

Nonemptyness is a consequence of classical results \cite{For}. Just
note that a neighbourhood of $p_S(S)$ can be embedded in ${\mathbb
C}^2$, in such a way that $P_S$ becomes the projection to the first
coordinate, and $p_S(S)$ becomes the closed unit disc in the first
axis. Hence $\Gamma_t$, $t$ small, becomes a torus in
$\partial{\mathbb D}\times{\mathbb C}$ enclosing $\partial{\mathbb
D}\times\{ 0\}$. Classical results on the Riemann-Hilbert problem in
${\mathbb C}^2$ imply that, for $t_0 > 0$ sufficiently small, there
exists a Levi-flat extension on $S\times\overline{{\mathbb
D}(t_0)}$.

Openness is a tautology. By definition, a real analytic embedding
defined on $S\times\overline{{\mathbb D}(t_0)}$ is in fact defined
on $S\times {\mathbb D}(t_0 + \varepsilon )$, for some $\varepsilon
> 0$, and obviously if $G$ is a Levi-flat extension on $S\times
\overline{{\mathbb D}(t_0)}$, then it is a Levi-flat extension also
on $S\times\overline{{\mathbb D}(t_0+\frac{\varepsilon}{2})}$.

The heart of the matter is closedness. In other words, we need to
prove that

{\it if a Levi-flat extension exists on $S\times {\mathbb D}(t_0)$,
then it exists also on $S\times\overline{{\mathbb D}(t_0)}$.}

The rest of this Section is devoted to the proof of this statement.

\subsection{Boundedness of areas}

We shall denote by $D_t^\theta$, $\theta\in{\mathbb S}^1$, the
closed holomorphic discs filling $M_t$, $0 < t <t_0$. Each
$D_t^\theta$ is the graph of a section $s_t^\theta : S\to U_S$,
holomorphic up to the boundary, with boundary values in $\Gamma_t$.

Consider the areas of these discs. These areas are computed with
respect to $\Pi_S^*(\omega )=\omega_0$, which is a real analytic
K\"ahler form on $U_S\setminus Indet(\Pi_S)$. Because $H^2
(U_S\setminus Indet(\Pi_S), {\mathbb R})=0$ (for $U_S$ is a
contractible complex surface and $Indet(\Pi_S)$ is a discrete
subset), this K\"ahler form is exact:
$$\omega_0 = d\lambda$$
for some real analytic 1-form $\lambda$ on $U_S\setminus
Indet(\Pi_S)$. If $D_t^\theta$ is disjoint from $Indet(\Pi_S)$, then
its area $\int_{D_t^\theta}\omega_0$ is simply equal, by Stokes
formula, to $\int_{\partial D_t^\theta}\lambda$. If $D_t^\theta$
intersects $Indet(\Pi_S)$, this is no more true, but still we have
the inequality
$${\rm area }(D_t^\theta) = \int_{D_t^\theta}\omega \le
\int_{\partial D_t^\theta}\lambda .$$ The reason is the following:
by the meromorphic map $\Pi_S$ the disc $D_t^\theta$ is mapped not
really to a disc in $X$, but rather to a disc {\it plus} some
rational bubbles coming from indeterminacy points of $\Pi_S$; then
$\int_{\partial D_t^\theta}\lambda$ is equal to the area of the disc
{\it plus} the areas of these rational bubbles, whence the
inequality above. Remark that, by our standing assumptions, the
boundary of $D_t^\theta$ is contained in $\partial U_S$ and hence it
is disjoint from $Indet(\Pi_S)$.

Now, the important fact is that, thanks to the crucial condition
(iv) of Lemma \ref{family}, we may get a {\it uniform} bound of
these areas.

\begin{lemma} \label{areabound}
There exists a constant $C > 0$ such that for every $t\in (0,t_0)$
and every $\theta\in{\mathbb S}^1$:
$${\rm area}(D_t^\theta) \le C .$$
\end{lemma}

\begin{proof}
By the previous remarks, we just have to bound the integrals
$\int_{\partial D_t^\theta}\lambda$. The idea is the following one.
For $t$ fixed the statement is trivial, and we need just to
understand what happens for $t\to t_0$. Look at the curves $\partial
D_t^\theta\subset\Gamma_t$. They are graphs of sections over
$\partial S$. For $t\to t_0$ these graphs could oscillate more and
more. But, using condition (iv) of Lemma \ref{family}, we will see
that these oscillations do not affect the integral of $\lambda$.
This would be evident if the tori $\Gamma_t$ were lagrangian (i.e.
$\omega_0\vert_{\Gamma_t}\equiv 0$, i.e. $\lambda\vert_{\Gamma_t}$
closed), so that the integrals of $\lambda$ would have a
cohomological meaning, not affected by the oscillations. Our
condition (iv) of Lemma \ref{family} expresses a sort of
half-lagrangianity in the direction along which oscillations take
place, and this is sufficient to bound the integrals.

Fix real analytic coordinates $(\varphi ,\psi ,r)\in {\mathbb S}^1
\times{\mathbb S}^1\times (-\varepsilon ,\varepsilon )$ around
$\Gamma_{t_0}$ in $\partial U_S$ such that:
\begin{enumerate}
\item[(i)] $P_S : \partial U_S\to\partial S$ is given by $(\varphi
,\psi ,r)\mapsto \varphi$;
\item[(ii)] $\Gamma_t =\{ r=t-t_0\}$ for every $t$ close to $t_0$.
\end{enumerate}
Each curve $\partial D_t^\theta$, $t < t_0$ close to $t_0$, is
therefore expressed by
$$\partial D_t^\theta = \{ \psi = h_t^\theta (\varphi ),\ r=t-t_0\}$$
for some real analytic function $h_t^\theta :{\mathbb S}^1 \to
{\mathbb S}^1$. Because the discs $D_t^\theta$ form a continuous
family, all these functions $h_t^\theta$ have the same degree, and
we may suppose that it is zero up to changing $\psi$ to $\psi +
\ell\varphi$.

The 1-form $\lambda$, restricted to $\partial U_S$, in these
coordinates is expressed by
$$\lambda = a(\varphi ,\psi ,r)d\varphi + b(\varphi ,\psi ,r)d\psi
+ c(\varphi ,\psi ,r)dr$$ for suitable real analytic functions
$a,b,c$ on ${\mathbb S}^1\times{\mathbb S}^1\times (-\varepsilon
,\varepsilon )$. Setting $b_0(\varphi ,r)= \int_{{\mathbb
S}^1}b(\varphi ,\psi ,r)d\psi$, we can write $b(\varphi ,\psi ,r) =
b_0(\varphi ,r)+ \frac{\partial b_1}{\partial\psi}(\varphi ,\psi
,r)$, for some real analytic function $b_1$ (the indefinite integral
of $b-b_0$ along $\psi$), and therefore
$$\lambda = a_0(\varphi ,\psi ,r)d\varphi + b_0(\varphi ,r)d\psi
+ c_0(\varphi ,\psi ,r)dr + db_1$$ with $a_0 = a - \frac{\partial
b_1}{\partial\varphi}$ and $c_0 = c - \frac{\partial b_1}{\partial
r}$.

Remark now that $b_0(\varphi ,r)$ is just equal to $\int_{\partial
D_t(z)}\lambda$, for $r=t-t_0$ and $\varphi =$ the coordinate of
$z\in\partial S$. By Stokes formula, this is equal to the area of
the disc $D_t(z)$, and by condition (iv) of Lemma \ref{family} this
does not depend on $\varphi$. That is, the function $b_0$ depends
only on $r$, and not on $\varphi$:
$$b_0(\varphi ,r)=b_0(r) .$$

In particular, if we restrict $\lambda$ to a torus $\Gamma_t$ we
obtain, up to an exact term, a 1-form $a_0 (\varphi ,\psi,
t-t_0)d\varphi + b_0 (t-t_0)d\psi$ which is perhaps not closed (this
would be the lagrangianity of $\Gamma_t$), but its component along
$\psi$ is closed. And note that the oscillations of the curves
$\partial D_t^\theta$ are directed along $\psi$.

If we now integrate $\lambda$ along $\partial D_t^\theta$ we obtain
$$\int_{\partial D_t^\theta}\lambda = \int_{{\mathbb S}^1}
a_0(\varphi ,h_t^\theta(\varphi ), t-t_0)d\varphi + b_0(t-t_0)\cdot
\int_{{\mathbb S}^1} \frac{\partial h_t^\theta}{\partial\varphi}
(\varphi )d\varphi .$$ The first integral is bounded by $C= \sup
|a_0|$, and the second integral is equal to zero because the degree
of $h_t^\theta$ is zero.
\end{proof}

Take now any sequence of discs
$$D_n = D_{t_n}^{\theta_n} , \ n\in{\mathbb N},$$
with $t_n\to t_0$. Our next aim is to prove that $\{ D_n\}$
converges (up to subsequencing) to some disc $D_\infty\subset U_S$,
with boundary in $\Gamma_{t_0}$. The limit discs so obtained will be
then glued together to produce the Levi-flat hypersurface $M_{t_0}$.

\subsection{Convergence around the boundary}

We firstly prove that everything is good around the boundary. Recall
that every disc $D_n$ is the graph of a section
$s_n=s_{t_n}^{\theta_n} : S \to U_S$ with boundary values in
$\Gamma_n= \Gamma_{t_n}$.

\begin{lemma}\label{boundary}
There exists a neighbourhood $V\subset S$ of $\partial S$ and a
section
$$s_\infty : V \to U_S$$
such that $s_n\vert_V$ converges uniformly to $s_\infty$ (up to
subsequencing).
\end{lemma}

\begin{proof}
We want to apply Bishop compactness theorem \cite{Bis} \cite{Chi} to
the sequence of analytic subsets of bounded area $D_n\subset U_S$.
This requires some care due to the boundary.

Let us work on some slightly larger open disc $S'\subset T$
containing the closed disc $S$. Every torus $\Gamma_t\subset U_{S'}$
has a neighbourhood $W_t\subset U_{S'}$ over which we have a well
defined Schwarz reflection with respect to $\Gamma_t$ (which is
totally real and of half dimension in $U_{S'}$). Thus, the complex
curve $D_t^\theta\cap W_t$ with boundary in $\Gamma_t$ can be
doubled to a complex curve without boundary $A_t^\theta$, properly
embedded in $W_t$. Moreover, using the fact that the tori $\Gamma_t$
form a real analytic family up to $t_0$, we see that the size of the
neighbourhoods $W_t$ is uniformly bounded from below. That is, there
exists a neighbourhood $W$ of $\Gamma_{t_0}$ in $U_{S'}$ which is
contained in every $W_t$, for $t$ sufficiently close to $t_0$, and
therefore every $A_t^\theta$ restricts to a properly embedded
complex curve in $W$, still denoted by $A_t^\theta$. Set
$$\widehat D_t^\theta = D_t^\theta\cup A_t^\theta .$$
Because the Schwarz reflection respects the fibration of $U_{S'}$,
it is clear that $\widehat D_t^\theta$ is still the graph of a
section $\widehat s_t^\theta$, defined over some open subset
$R_t^\theta\subset S'$ which contains $S$. The area of $A_t^\theta$
is roughly the double of the area of $D_t^\theta\cap W$, and
therefore the analytic subsets $\widehat D_t^\theta\subset U_S\cup
W$ also have uniformly bounded areas.

Having in mind this uniform extension of the discs $D_t^\theta$ into
the neighbourhood $W$ of $\Gamma_{t_0}$, we now apply Bishop Theorem
to the sequence $\{ D_n\}$. Remark that $\partial D_n\subset
\Gamma_n$ cannot exit from $W$, as $n\to +\infty$, because
$\Gamma_n$ converges to $\Gamma_{t_0}$. Up to subsequencing, we
obtain that $D_n$ Hausdorff-converges to a complex curve
$D_\infty\subset U_S$ with boundary in $\Gamma_{t_0}$. Moreover, and
taking into account that $D_n$ are graphs over $S$ (compare with
\cite[Prop. 3.1]{Iv1}):
\begin{enumerate}
\item[(i)] $D_\infty = D_\infty^0\cup E_1\cup\ldots\cup E_m \cup
F_1\cup\ldots\cup F_\ell$;
\item[(ii)] $D_\infty^0$ is the graph of a section $s_\infty : V\to
U_S$, over some open subset $V\subset S$ which contains $\partial
S$;
\item[(iii)] each $E_j$ is equal to $P_S^{-1}(p_j)$, for some
$p_j\in S\setminus\partial S$ (interior bubble);
\item[(iv)] each $F_j$ is equal to the closure of a connected
component of $P_S^{-1}(q_j)\setminus\Gamma_{t_0}(q_j)$, for some
$q_j\in \partial S$ (boundary bubble);
\item[(v)] for every compact $K\subset V\setminus \{ p_1,\ldots
,p_m,q_1,\ldots ,q_\ell\}$, $s_n\vert_K$ converges uniformly to
$s_\infty\vert_K$, as $n\to +\infty$.
\end{enumerate}
We have just to prove that there are no boundary bubbles, i.e. that
the set $\{ q_1,\ldots ,q_\ell\}$ is in fact empty.

Consider the family of Levi-flat hypersurfaces $M_t\subset U_S$ with
boundary $\Gamma_t$, for $t < t_0$. Each $M_t$ is a ``lower
barrier'', which prevents the approaching of $D_n$ to the bounded
component of $P_S^{-1}(q)\setminus\Gamma_{t_0}(q)$, for every
$q\in\partial S$. More precisely, for any compact $R$ in that
bounded component we may select $t_1 < t_0$ such that $\cup_{0\le t
< t_1}\Gamma_t$ contains $R$, and so $\cup_{0\le t < t_1}M_t$ is a
neighbourhood of $R$ in $U_S$. For $n$ sufficiently large (so that
$t_n > t_1$), $D_n\subset M_{t_n}$ is disjoint from that
neighbourhood of $R$. Hence the sequence $D_n$ cannot accumulate to
the bounded component of $P_S^{-1}(q)\setminus\Gamma_{t_0}(q)$.

But neither $D_n$ can accumulate to the unbounded component of
$P_S^{-1}(q)\setminus\Gamma_{t_0}(q)$, because that component has
infinite area, by our standing assumptions. Therefore, as desired,
$$\{ q_1,\ldots ,q_\ell\}=\emptyset .$$
Remark that by the same barrier argument we have also
$$\{ p_1,\ldots ,p_m\}=\emptyset $$
but this will not be used below.
\end{proof}

The proof above shows in fact the following: there is a maximal $V$
over which $s_\infty$ is defined, and for every $z\not\in V$ the
sequence $s_n(z)$ is divergent in the fiber $P_S^{-1}(z)$.

\subsection{Convergence on the interior}

In order to extend the convergence above from $V$ to the full $S$,
we need to use the map $\Pi_S$ into $X$. Consider the discs
$$f_n = \Pi_S\circ s_n : S \rightarrow X$$
in the compact K\"ahler manifold $X$. They have bounded area, and,
once a time, we apply to them Bishop compactness theorem \cite{Bis}
\cite[Prop.3.1]{Iv1}. We obtain a holomorphic map
$$f_\infty : S\cup B \rightarrow X$$
which obviously coincides with $\Pi_S\circ s_\infty$ on the
neighbourhood $V$ of $\partial S$ of Lemma \ref{boundary}. The set
$B$ is a union of trees of rational curves, each one attached to
some point of $S$ outside $V$. We will prove that $f_\infty\vert_S$
can be lifted to $U_S$, providing the extension of the section
$s_\infty$ to the full $S$.

The map $f_\infty$ is an immersion around $\partial S$. Let us even
suppose that it is an embedding (anyway, this is true up to moving a
little $\partial S$ inside $S$, and this does not affect the
following reasoning). In some sufficiently smooth tubular
neighbourhood $X_0\subset X$ of $f_\infty (\partial S)$, we have a
properly embedded complex surface with boundary $Y$, given by the
image by $\Pi_S$ of a neighbourhood of $s_\infty (\partial S)$ in
$U_S$. The boundary $\partial Y$ of $Y$ in $X_0$ is filled by the
images by $\Pi_S$ of part of the tori $\Gamma_t$, $t$ close to
$t_0$; denote them by $\Gamma_t'$ (with a good choice of $X_0$, each
$\Gamma_t'$ is a real annulus). Thus $f_n$, $n$ large, sends $S$ to
a disc in $X$ whose (embedded) boundary is contained in $\Gamma_n' =
\Gamma_{t_n}'$, and $f_\infty$ sends $S\cup B$ to a disc with
rational bubbles in $X$ whose (embedded) boundary is contained in
$\Gamma_\infty ' =\Gamma_{t_0}'$. Inspired by \cite{IvS}, but
avoiding any infinite dimensional tool due to our special context,
we now prove that $f_\infty$ and $f_n$, for some large $n$, can be
holomorphically interpolated by discs with boundaries in $\partial
Y$.

\begin{lemma} \label{interpolation}
There exists a complex surface with boundary $W$, a proper map $\pi
: W\to {\mathbb D}$, a holomorphic map $g:W\to X$, such that:
\begin{enumerate}
\item[(i)] for every $w\not= 0$, the fiber $W_w = \pi^{-1}(w)$ is
isomorphic to $S$, and $g$ sends that fiber to a disc in $X$ with
boundary in $\partial Y$;
\item[(ii)] for some $e\not=0$, $g$ coincides on $W_e =\pi^{-1}(e)$
with $f_n$, for some $n$ (large);
\item[(iii)] $W_0 = \pi^{-1}(0)$ is
isomorphic to $S\cup B$, and $g$ on that fiber coincides with
$f_\infty$.
\end{enumerate}
\end{lemma}

\begin{proof}
Let us work on the complex manifold $\widehat X = X\times{\mathbb
D}(t_0,\varepsilon )$, where the second factor is a small disc in
${\mathbb C}$ centered at $t_0$. The real surfaces $\Gamma_t'$ in
$X_0$ can be seen as a single real analytic submanifold of dimension
three $\Gamma '$ in $\widehat X_0 = X_0\times{\mathbb
D}(t_0,\varepsilon )$, by considering $\Gamma_t'$ as a subset of
$X_0\times\{ t\}$. Remark that $\Gamma '$ is totally real.
Similarly, the discs $f_n(S)$ can be seen as discs in $X\times\{
t_n\}\subset\widehat X$, and the disc with bubbles $f_\infty (S\cup
B)$ can be seen as a disc with bubbles in $X\times\{
t_0\}\subset\widehat X$; all these discs have boundaries in $\Gamma
'$. In $\widehat X_0$ we also have a complex submanifold of
dimension three with boundary $\widehat{Y} = Y\times{\mathbb
D}(t_0,\varepsilon )$, which is ``half'' of the complexification of
$\Gamma '$.

Around the circle $f_\infty (\partial S)\subset \widehat X$, we may
find holomorphic coordinates $z_1,\ldots ,z_{n+1}$, with $| z_j | <
\delta$ for $j\le n$, $1-\delta < | z_{n+1}| < 1+\delta$, such that:
\begin{enumerate}
\item[(i)] $f_\infty (\partial S) = \{ z_1=\ldots =z_n =0,\ |
z_{n+1}| = 1 \}$;
\item[(ii)] $\Gamma ' = \{ z_1 =\ldots =z_{n-2}=0,\ Im\ z_{n-1} = Im\
z_n = 0,\ |z_{n+1}| = 1\}$;
\item[(iii)] $\widehat{Y} = \{ z_1 =\ldots =z_{n-2}=0,\ |z_{n+1}| \le
1\}$.
\end{enumerate}

We consider, in these cordinates, the Schwarz reflection
$(z_1,\ldots ,z_n,z_{n+1})\mapsto (\bar z_1,\ldots \bar z_n,
\frac{1}{\bar z_{n+1}})$. It is a antiholomorphic involution, which
fixes in particular every point of $\Gamma '$. Using it, we may
double a neighbourhood $Z_0$ of $f_\infty (S\cup B)$ in $\widehat
X$: we take $Z_0$ and $\overline Z_0$ (i.e., $Z_0$ with the opposite
complex structure), and we glue them together using the Schwarz
reflection. Call $Z$ this double of $Z_0$. Then $Z$ naturally
contains a tree of rational curves $R_\infty$ which comes from
doubling $f_\infty (S\cup B)$, because this last has boundary in the
fixed point set of the Schwarz reflection. Similarly, each $f_n(S)$
doubles to a rational curve $R_n\subset Z$, close to $R_\infty$ for
$n$ large. Moreover, in some neighbourhood $N\subset Z$ of the
median circle of $R_\infty$ (arising from $f_\infty (\partial S)$),
we have a complex threefold $\widetilde Y\subset N$, arising by
doubling $\widehat Y$ or by complexifying $\Gamma '$, which contains
$R_\infty\cap N$ and every $R_n\cap N$.

Now, the space of trees of rational curves in $Z$ close to
$R_\infty$ has a natural structure of complex analytic space
${\mathcal R}$, see e.g. \cite{CaP} or \cite{IvS}. Those trees
which, in $N$, are contained in $\widetilde Y$ form a complex
analytic subspace ${\mathcal R}_0\subset{\mathcal R}$. The curve
$R_n$ above correspond to points of ${\mathcal R}_0$ converging to a
point corresponding to $R_\infty$. Therefore we can find a disc in
${\mathcal R}_0$ centered at $R_\infty$ and passing through some
$R_n$. This gives a holomorphic family of trees of rational curves
in $Z$ interpolating $R_\infty$ and $R_n$. Restricting to
$Z_0\subset Z$ and projecting to $X$, we obtain the desired family
of discs $g$.
\end{proof}

Note the the doubling trick used in the previous lemma is not so far
from the similar trick used in the proof of Lemma \ref{koebe}. In
both cases, a problem concerning discs is reduced to a more
tractable problem concerning rational curves.

The map $g$ can be lifted to $U_S$ around $\partial W$ and around
$W_e$. In this way, and up to a reparametrization, we obtain an
embedding
$$h : H \rightarrow U_S,$$
where $H =\{ (z,w)\in S\times{\mathbb D}\ \vert\ z\in V \ {\rm or}\
|w-e| < \varepsilon\}$ (for some $\varepsilon >0$ small), such that:
\begin{enumerate}
\item[(i)] $h(\cdot ,w)$ is a section of $U_S$ over $S$ (if $|w-e|
< \varepsilon$) or over $V$ (if $|w-e| \ge \varepsilon$);
\item[(ii)] $h(z,e)=s_n(z)$ for every $z\in S$;
\item[(iii)] $h(z,0)=s_\infty (z)$ for every $z\in V$.
\end{enumerate}
(note, however, that generally speaking the section $h(\cdot ,w)$
has not boundary values in some torus $\Gamma_t$, when $w\not=
0,e$). In some sense, we are in a situation similar to the one
already encountered in the construction of covering tubes in Section
4, but rotated by 90 degrees.

\begin{figure}[h]
\includegraphics[width=8cm,height=6cm]{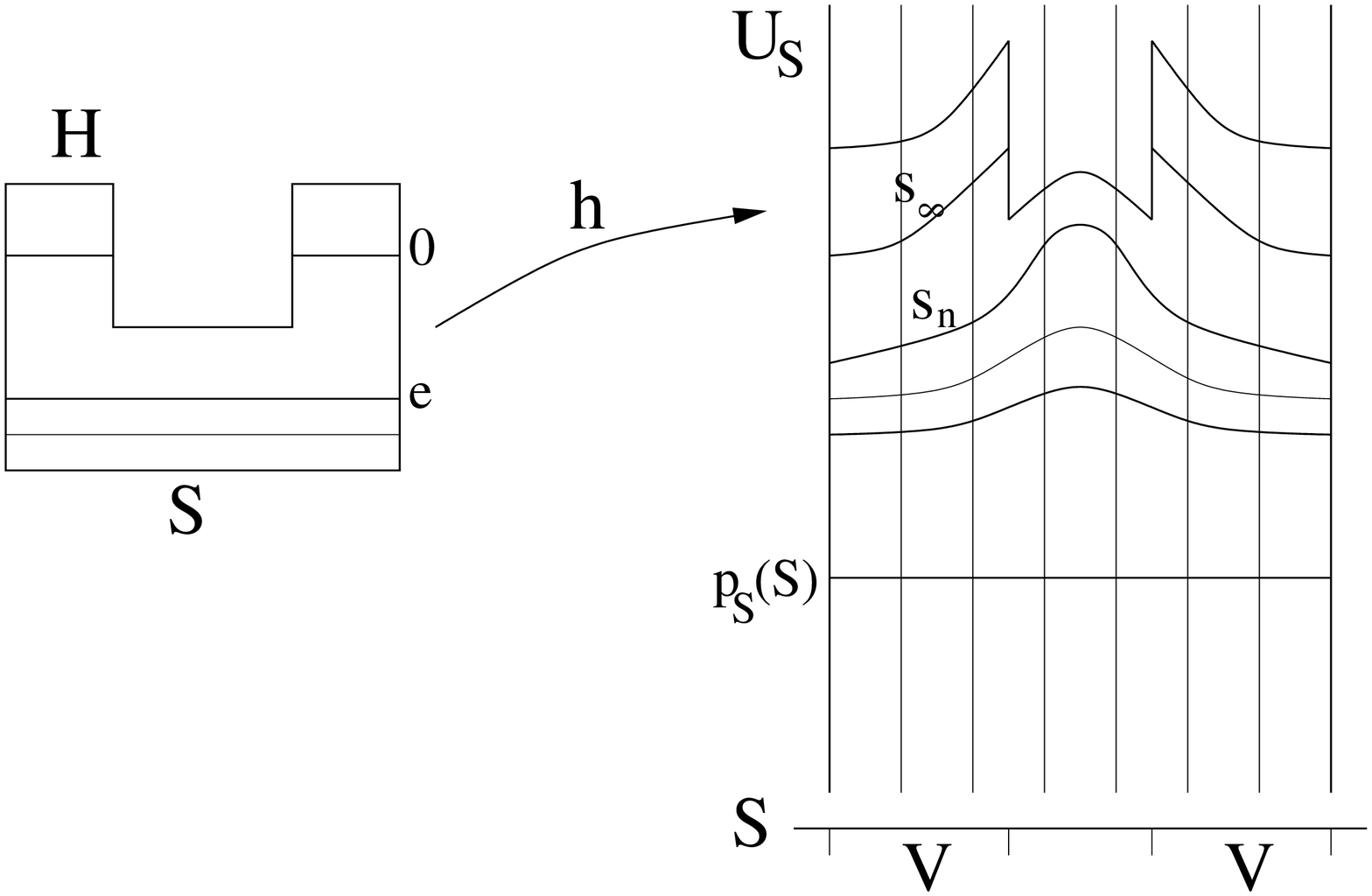}
\end{figure}

Consider now the meromorphic immersion $\Pi_S \circ h :
H\dashrightarrow X$. By \cite{Iv1}, this map can be meromorphically
extended to the envelope $S\times {\mathbb D}$, and clearly this
extension is still a meromorphic immersion. Each vertical fiber $\{
z\}\times{\mathbb D}$ is sent to a disc tangent to the foliation
${\mathcal F}$, and possibly passing through $Sing({\mathcal F})$.
But for every $z\in V$ we already have, by construction, that such a
disc can be lifted to $U_S$. By our definition and construction of
$U_S$, it then follows that the same holds for every $z\in S$: every
intersection point with $Sing({\mathcal F})$ is a vanishing end.
Hence the full family $S\times{\mathbb D}$ can be lifted to $U_S$,
or in other words the embedding $h: H\to U_S$ can be extended to
$\widehat h : S\times{\mathbb D}\to U_S$.

Take now $\widehat h (,\cdot ,0)$: it is a section over $S$ which
extends $s_\infty$. Thus, the section $s_\infty$ from Lemma
\ref{boundary} can be extended from $V$ to $S$, and the sequence of
discs $D_n\subset U_S$ uniformly converges to $D_\infty = s_\infty
(S)$.

\subsection{Construction of the limit Levi-flat hypersurface}

Let us resume. We are assuming that our Levi-flat extension exists
over $S\times{\mathbb D}(t_0)$, providing an embedded real analytic
family of Levi-flat hypersurfaces $M_t\subset U_S$ with boundaries
$\Gamma_t$, $t < t_0$. Given any sequence of holomorphic discs
$D_{t_n}^{\theta_n}\subset M_{t_n}$, $t_n\to t_0$, we have proved
that (up to subsequencing) $D_{t_n}^{\theta_n}$ converges uniformly
to some disc $D_\infty$ with $\partial D_\infty\subset
\Gamma_{t_0}$. Given any point $p\in\Gamma_{t_0}$, we may choose the
sequence $D_{t_n}^{\theta_n}$ so that $\partial D_\infty$ will
contain $p$. It remains to check that all the discs so constructed
glue together in a real analytic way, giving $M_{t_0}$, and that
this $M_{t_0}$ glues to $M_t$, $t < t_0$, in a real analytic way,
giving the Levi-flat extension over $S\times\overline{{\mathbb
D}(t_0)}$.

This can be seen using a Lemma from \cite[\S 5]{BeG}. It says that
if $D$ is an embedded disc in a complex surface $Y$ with boundary in
a real analytic totally real surface $\Gamma\subset Y$, and if the
winding number (Maslov index) of $\Gamma$ along $\partial D$ is
zero, then $D$ belongs to a unique embedded real analytic family of
discs $D^\varepsilon$, $\varepsilon\in
(-\varepsilon_0,\varepsilon_0)$, $D^0=D$, with boundaries in
$\Gamma$ (incidentally, in our real analytic context this can be
easily proved by the doubling argument used in Lemma
\ref{interpolation}, which reduces the statement to the well known
fact that a smooth rational curve of zero selfintersection belongs
to a unique local fibration by smooth rational curves). Moreover, if
$\Gamma$ is moved in a real analytic way, then the family
$D^\varepsilon$ also moves in a real analytic way.

For our discs $D_t^\theta\subset M_t$, $t < t_0$, the winding number
of $\Gamma_t$ along $\partial D_t^\theta$ is zero. By continuity of
this index, if $D_\infty$ is a limit disc then the winding number of
$\Gamma_{t_0}$ along $\partial D_\infty$ is also zero. Thus,
$D_\infty$ belongs to a unique embedded real analytic family
$D_\infty^\varepsilon$, with $\partial D_\infty^\varepsilon\subset
\Gamma_{t_0}$. This family can be deformed, real analytically, to a
family $D_t^\varepsilon$ with $\partial D_t^\varepsilon\subset
\Gamma_t$, for every $t$ close to $t_0$. When $t=t_n$, such a family
$D_{t_n}^\varepsilon$ necessarily contains $D_{t_n}^{\theta_n}$, and
thus coincides with $D_{t_n}^\theta$ for $\theta$ in a suitable
interval around $\theta_n$. Hence, for every $t < t_0$ the family
$D_t^\varepsilon$ coincides with $D_t^\theta$, for $\theta$ in a
suitable interval.

In this way, for every limit disc $D_\infty$ we have constructed a
piece
$$\bigcup_{\varepsilon\in
(-\varepsilon_0,\varepsilon_0)}D_\infty^\varepsilon$$ of our limit
$M_{t_0}$, this piece is real analytic and glues to $M_t$, $t <
t_0$, in a real analytic way. Because each $p\in\Gamma_{t_0}$
belongs to some limit disc $D_\infty$, we have completed our
construction of $M_{t_0}$, and the proof of Theorem \ref{convexity}.

\section{Hyperbolic foliations}

We can now draw the first consequences of the convexity of covering
tubes given by Theorem \ref{convexity}, still following \cite{Br2}
and \cite{Br3}.

As in the previous Section, let $X$ be a compact K\"ahler manifold
of dimension $n$, equipped with a foliation by curves ${\mathcal F}$
which is not a rational quasi-fibration. Let $T\subset X^0$ be local
transversal to ${\mathcal F}^0$. We firstly need to discuss the
pertinence of hypotheses (a) and (b) that we made at the beginning
of Section 5.

Concerning (a), let us simply observe that $Indet(\Pi_T)$ is an
analytic subset of codimension at least two in $U_T$, and therefore
its projection to $T$ by $P_T$ is a countable union of locally
analytic subsets of positive codimension in $T$ (a thin subset of
$T$). Hypothesis (a) means that the closed disc $S\subset T$ is
chosen so that it is not contained in that projection, and its
boundary $\partial S$ is disjoint from that projection.

Concerning (b), let us set
$$R = \{ z\in T\ \vert\ {\rm area}(P_T^{-1}(z)) < +\infty\} .$$

\begin{lemma} \label{areafibers}
Either $R$ is a countable union of analytic subsets of $T$ of
positive codimension, or $R=T$. In this second case, $U_T$ is
isomorphic to $T\times{\mathbb C}$.
\end{lemma}

\begin{proof}
If $z\in R$, then $\widetilde{L_z}$ has finite area and, a fortiori,
$L_z^0$ has finite area. In particular, $L_z^0$ is properly embedded
in $X^0$: otherwise, $L_z^0$ should cut some foliated chart, where
${\mathcal F}^0$ is trivialized, along infinitely many plaques, and
so $L_z^0$ would have infinite area. Because $X\setminus X^0$ is an
analytic subset of $X$, the fact that $L_z^0\subset X^0$ is properly
embedded and with finite area implies that its closure
$\overline{L_z^0}$ in $X$ is a complex compact curve, by Bishop
extension theorem \cite{Siu} \cite{Chi}. This closure coincides with
$\overline{L_z}$, the closure of $L_z$.

The finiteness of the area of $\widetilde{L_z}$ implies also that
the covering $\widetilde{L_z}\to L_z$ has finite order, i.e. the
orbifold fundamental group of $L_z$ is finite. By the previous
paragraph, $L_z$ can be compactified (as a complex curve) by adding
a finite set. This excludes the case $\widetilde{L_z}={\mathbb D}$:
a finite quotient of the disc does not enjoy such a property. Also,
the case $\widetilde{L_z}={\mathbb P}$ is excluded by our standing
assumptions. Therefore $\widetilde{L_z}={\mathbb C}$. Moreover,
again the finiteness of the orbifold fundamental group implies that
$L_z$ is equal to ${\mathbb C}$ with at most one multiple point. The
closure $\overline{L_z}$ is a rational curve in $X$.

Now, by general principles of analytic geometry \cite{CaP}, rational
curves in $X$ (K\"ahler) constitute an analytic space with countable
base, each irreducible component of which can be compactified by
adding points corresponding to trees of rational curves. It follows
easily from this fact that the subset
$$R' = \{ z\in T\ \vert\ \overline{L_z}\ {\rm is\ rational}\}$$
is either a countable union of analytic subsets of $T$ of positive
codimension, or it is equal to the full $T$. Moreover, if $A'$ is a
component of $R'$ then we can find a meromorphic map
$A'\times{\mathbb P}\dashrightarrow X$ sending $\{ z\}\times{\mathbb
P}$ to $\overline{L_z}$, for every $z\in A'$ (compare with the
arguments used at the beginning of the proof of Theorem
\ref{hartogs}).

Not every $z\in A'$, however, belongs to $R$:  a point $z\in A'$
belongs to $R$ if and only if among the points of $\{ z\}\times
{\mathbb P}$ sent to $Sing({\mathcal F})$ only one does not
correspond to a vanishing end of $L_z^0$, and at most one
corresponds to a vanishing end of order $m \ge 2$. By a simple
semicontinuity argument, $A'\cap R = A$ is an analytic subset of
$A'$. Hence $R$ also satisfies the above dichotomy.

Finally, if $R=T$ then we have a map $T\times{\mathbb
P}\dashrightarrow X$ sending each fiber $\{ z\}\times{\mathbb P}$ to
$\overline{L_z}$ and $(z,\infty )$ to the unique nonvanishing end of
$L_z^0$. It follows that $U_T = T\times{\mathbb C}$.
\end{proof}

Let now $U\subset X$ be an open connected subset where ${\mathcal
F}$ is generated by a holomorphic vector field $v\in\Theta (U)$,
vanishing precisely on $Sing({\mathcal F})\cap U$. Set $U^0 =
U\setminus (Sing({\mathcal F})\cap U)$, and consider the real
function
$$F : U^0 \rightarrow [-\infty ,+\infty )$$
$$F(q) = \log \| v(q)\|_{Poin}$$
where, as usual, $\| v(q)\|_{Poin}$ is the norm of $v(q)$ measured
with the Poincar\'e metric on $L_q$. Recall that this ``metric'' is
identically zero when $L_q$ is parabolic, so that $F$ is equal to
$-\infty$ on the intersection of $U^0$ with parabolic leaves.

\begin{proposition} \label{psh}
The function $F$ above is either plurisubharmonic or identically
$-\infty$.
\end{proposition}

\begin{proof}
Let $T\subset U^0$ be a transversal to ${\mathcal F}^0$, and let
$U_T$ be the corresponding covering tube. Put on the fibers of $U_T$
their Poincar\'e metric. The vector field $v$ induces a nonsingular
vertical vector field on $U_T$ along $p_T(T)$, which we denote again
by $v$. Due to the arbitrariness of $T$, and by a connectivity
argument, we need just to verify that the function on $T$ defined by
$$F(z) = \log \| v(p_T(z)\|_{Poin}$$
is either plurisubharmonic or identically $-\infty$. That is, the
fiberwise Poincar\'e metric on $U_T$ has a plurisubharmonic
variation.

The upper semicontinuity of $F$ being evident (see e.g. \cite[\S
3]{Suz} or \cite{Kiz}), let us consider the submean inequality over
discs in $T$.

Take a closed disc $S\subset T$ as in Theorem \ref{convexity}, i.e.
satisfying hypotheses (a) and (b) of Section 5. By that Theorem, and
by choosing an increasing sequence of compact subsets $K_j$ in
$\partial U_S$, we can find a sequence of relatively compact domains
$\Omega_j\subset U_S$, $j\in{\mathbb N}$, such that:

\begin{figure}[h]
\includegraphics[width=7cm,height=8cm]{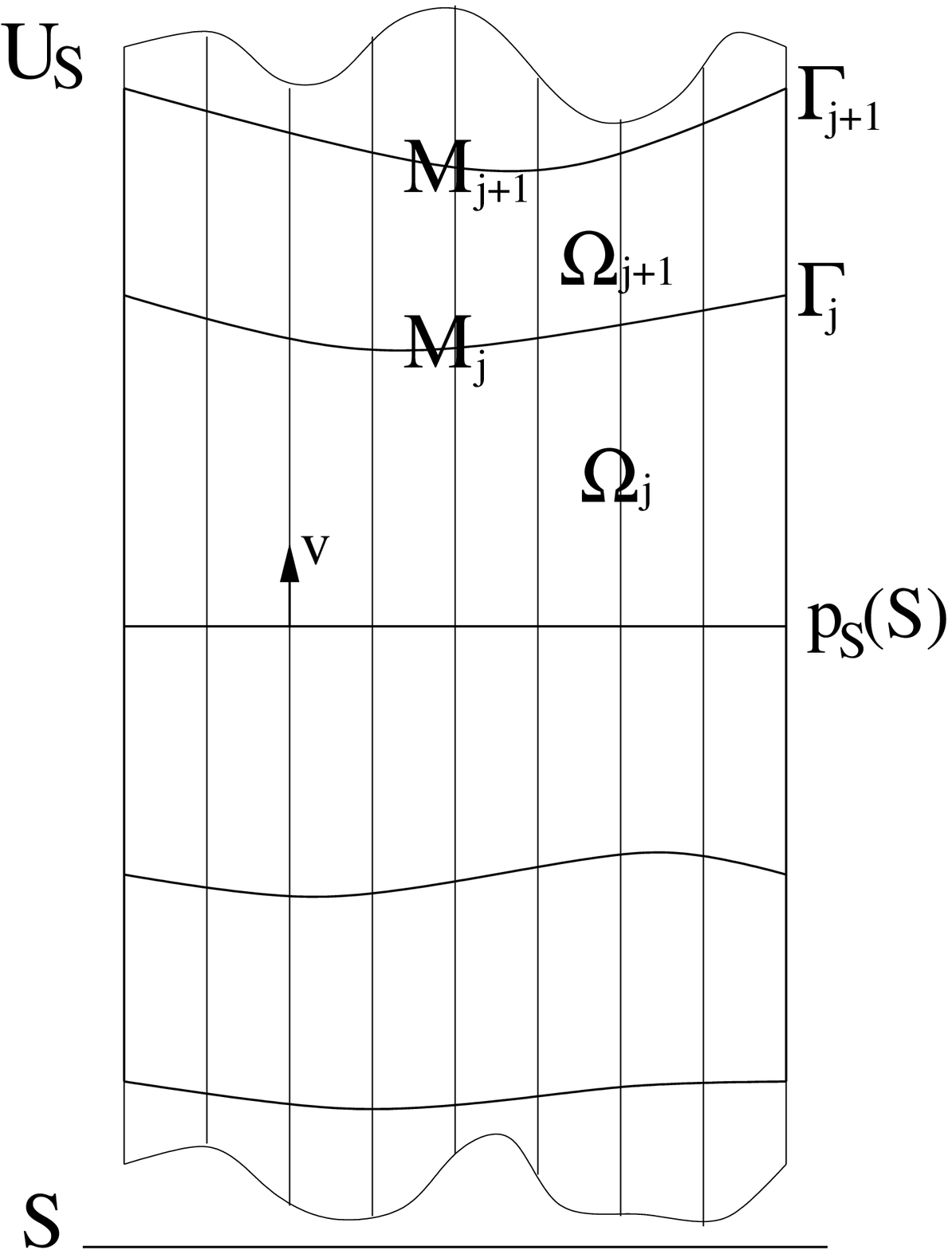}
\end{figure}

\begin{enumerate}
\item[(i)] the relative boundary of $\Omega_j$ in $U_S$ is a real
analytic Levi-flat hypersurface $M_j\subset U_S$, with boundary
$\Gamma_j\subset\partial U_S$, filled by a ${\mathbb S}^1$-family of
graphs of holomorphic sections of $U_S$ with boundary values in
$\Gamma_j$;
\item[(ii)] for every $z\in S$, the fiber $\Omega_j(z) =
\Omega_j\cap P_S^{-1}(z)$ is a disc, centered at $p_S(z)$; moreover,
for $z\in\partial S$ we have $\cup_{j=1}^{+\infty} \Omega_j(z) =
P_S^{-1}(z)$.
\end{enumerate}
Note that one cannot hope that the exhaustive property in (ii) holds
also for $z$ in the interior of $S$.

We may apply to $\Omega_j$, whose boundary is Levi-flat and hence
pseudoconvex, the result of Yamaguchi discussed in Section 2, more
precisely Proposition \ref{yamaguchi}. It says that the function on
$S$
$$F_j(z) = \log \| v(p_S(z)\|_{Poin(j)},$$
where $\| v(p_S(z)\|_{Poin(j)}$ is the norm with respect to the
Poincar\'e metric on the disc $\Omega_j(z)$, is plurisubharmonic.
Hence we have at the center 0 of $S\simeq\overline{\mathbb D}$ the
submean inequality:
$$F_j(0) \le \frac{1}{2\pi}\int_0^{2\pi} F_j(e^{i\theta})d\theta .$$

We now pass to the limit $j\to +\infty$. For every $z\in\partial S$
we have $F_j(z)\to F(z)$, by the exhaustive property in (ii) above.
Moreover, we may assume that $\Omega_j(z)$ is an increasing sequence
for every $z\in\partial S$ (and in fact for every $z\in S$, but this
is not important), so that $F_j(z)$ converges to $F(z)$ in a
decreasing way, by the monotonicity property of the Poincar\'e
metric. It follows that the boundary integral in the submean
inequality above converges, as $j\to +\infty$, to
$\frac{1}{2\pi}\int_0^{2\pi} F(e^{i\theta})d\theta$ (which may be
$-\infty$, of course).

Concerning $F_j(0)$, it is sufficient to observe that, obviously,
$F(0)\le F_j(0)$, because $\Omega_j(0)\subset P_S^{-1}(0)$, and so
$F(0)\le\liminf_{j\to +\infty} F_j(0)$. In fact, and because
$\Omega_j(0)$ is increasing, $F_j(0)$ converges to some value $c$ in
$[-\infty ,+\infty )$, but we may have the strict inequality $F(0) <
c$ if $\Omega_j(0)$ do not exhaust $P_S^{-1}(0)$. Therefore the
above submean inequality gives, at the limit,
$$F(0) \le \frac{1}{2\pi}\int_0^{2\pi} F(e^{i\theta})d\theta$$
that is, the submean inequality for $F$ on $S$.

Take now an arbitrary closed disc $S\subset T$, centered at some
point $p\in T$. By Lemma \ref{areafibers} and the remarks before it,
we may approximate $S$ by a sequence of closed discs $S_j$ with the
same center $p$ and satisfying moreover hypotheses (a) and (b)
before Theorem \ref{convexity} (unless $R=T$, but in that case $U_T
= T\times{\mathbb C}$ and $F\equiv -\infty$). More precisely, if
$\varphi :\overline{\mathbb D}\to T$ is a parametrization of $S$,
$\varphi (0)=p$, then we may uniformly approximate $\varphi$ by a
sequence of embeddings $\varphi_j :\overline{\mathbb D}\to T$,
$\varphi_j(0)=p$, such that $S_j=\varphi_j(\overline{\mathbb D})$
satisfies the assumptions of Theorem \ref{convexity}. Hence we have,
by the previous arguments and for every $j$,
$$F(p)\le \frac{1}{2\pi}\int_0^{2\pi}
F(\varphi_j(e^{i\theta}))d\theta$$ and passing to the limit, using
Fatou Lemma, and taking into account the upper semicontinuity of
$F$, we finally obtain
$$F(p) \le \limsup_{j\to +\infty}\frac{1}{2\pi}\int_0^{2\pi}
F(\varphi_j(e^{i\theta}))d\theta \le \frac{1}{2\pi}\int_0^{2\pi}
\limsup_{j\to +\infty}F(\varphi_j(e^{i\theta}))d\theta \le$$
$$\le\frac{1}{2\pi}\int_0^{2\pi} F(\varphi (e^{i\theta}))d\theta .$$
This is the submean inequality on an arbitrary disc $S\subset T$,
and so $F$ is, if not identically $-\infty$, plurisubharmonic.
\end{proof}

Because $U\setminus U^0$ is an analytic subset of codimension at
least two, the above function $F$ on $U^0$ admits a (unique)
plurisubharmonic extension to the full $U$, given explicitely by
$$F(q) = \limsup_{p\in U^0,\ p\to q} F(p)\ ,\quad q\in U\setminus
U^0.$$

\begin{proposition} \label{poles}
We have $F(q)=-\infty$ for every $q\in U\setminus U^0$.
\end{proposition}

\begin{proof}
The vector field $v$ on $U$ has a local flow: a holomorphic map
$$\Phi : {\mathcal D}\rightarrow U$$
defined on a domain of the form
$${\mathcal D} = \{ (p,t)\in U\times{\mathbb C}\ \vert\ |t| < \rho
(p)\}$$ for a suitable lower semicontinuous function $\rho : U\to
(0,+\infty ]$, such that $\Phi (p,0)=p$,
$\frac{\partial\Phi}{\partial t}(p,0)=v(p)$, and $\Phi (p,t_1+t_2) =
\Phi (\Phi (p,t_1),t_2)$ whenever it makes sense. Standard results
on ordinary differential equations show that we may choose the
function $\rho$ so that $\rho\equiv +\infty$ on $U\setminus U^0 =$
the zero set of $v$.

Take $q\in U\setminus U^0$ and $p\in U^0$ close to it. Then $\Phi
(p,\cdot )$ sends the large disc ${\mathbb D}(\rho (p))$ into
$L_p^0\cap U^0$, and consequently into $L_p$, with derivative at 0
equal to $v(p)$. It follows, by monotonicity of the Poincar\'e
metric, that the Poincar\'e norm of $v(p)$ is bounded from above by
something like $\frac{1}{\rho (p)}$, which tends to 0 as $p\to q$.
We therefore obtain that $\log \| v(p)\|_{Poin}$ tends to $-\infty$
as $p\to q$.
\end{proof}

The functions $F : U\to [-\infty ,+\infty )$ so constructed can be
seen \cite{Dem} as local weights of a (singular) hermitian metric on
the tangent bundle $T_{\mathcal F}$ of ${\mathcal F}$, and by
duality on the canonical bundle $K_{\mathcal F}= T_{\mathcal F}^*$.
Indeed, if $v_j\in\Theta (U_j)$ are local generators of ${\mathcal
F}$, for some covering $\{ U_j\}$ of $X$, with $v_j = g_{jk} v_k$
for a multiplicative cocycle $g_{jk}$ generating $K_{\mathcal F}$,
then the functions $F_j =\log\| v_j\|_{Poin}$ are related by
$F_j-F_k = \log | g_{jk} |$. The curvature of this metric on
$K_{\mathcal F}$ is the current on $X$, of bidegree $(1,1)$, locally
defined by $\frac{i}{\pi}\partial\bar\partial F_j$. Hence
Propositions \ref{psh} and \ref{poles} can be restated in the
following more intrinsic form, where we set $Parab({\mathcal F}) =
\{ p\in X^0\ \vert\ \widetilde{L_p} = {\mathbb C}\}$.

\begin{theorem} \label{hyperbolic}
Let $X$ be a compact connected K\"ahler manifold and let ${\mathcal
F}$ be a foliation by curves on $X$. Suppose that ${\mathcal F}$ has
at least one hyperbolic leaf. Then the Poincar\'e metric on the
leaves of ${\mathcal F}$ induces a hermitian metric on the canonical
bundle $K_{\mathcal F}$ whose curvature is positive, in the sense of
currents. Moreover, the polar set of this metric coincides with
$Sing({\mathcal F})\cup Parab({\mathcal F})$.
\end{theorem}

A foliation with at least one hyperbolic leaf will be called {\bf
hyperbolic foliation}. The existence of a hyperbolic leaf (and the
connectedness of $X$) implies that ${\mathcal F}$ is not a rational
quasi-fibration, and all the local weights $F$ introduced above are
plurisubharmonic, and not identically $-\infty$.

Let us state two evident but important Corollaries.

\begin{corollary} \label{pseudoeffective}
The canonical bundle $K_{\mathcal F}$ of a hyperbolic foliation
${\mathcal F}$ is pseudoeffective.
\end{corollary}

\begin{corollary} \label{pluripolar}
Given a hyperbolic foliation ${\mathcal F}$, the subset
$$Sing({\mathcal F})\cup Parab({\mathcal F})$$ is complete pluripolar
in $X$.
\end{corollary}

We think that the conclusion of this last Corollary could be
strengthened. The most optimistic conjecture is that $Sing({\mathcal
F})\cup Parab({\mathcal F})$ is even an {\it analytic subset} of
$X$. At the moment, however, we are very far from proving such a
fact (except when $\dim X =2$, where special techniques are
available, see \cite{MQ1} and \cite{Br1}). Even the {\it closedness}
of $Sing({\mathcal F})\cup Parab({\mathcal F})$ seems an open
problem! This is related to the more general problem of the
continuity of the leafwise Poincar\'e metric (which would give, in
particular, the closedness of its polar set). Let us prove a partial
result in this direction, following a rather standard hyperbolic
argument \cite{Ghy} \cite{Br2}. Recall that a complex compact
analytic space $Z$ is {\it hyperbolic} if every holomorphic map of
${\mathbb C}$ into $Z$ is constant \cite{Lan}.

\begin{theorem} \label{continuity}
Let ${\mathcal F}$ be a foliation by curves on a compact connected
K\"ahler manifold $M$. Suppose that:
\begin{enumerate}
\item[(i)] every leaf is hyperbolic;
\item[(ii)] $Sing({\mathcal F})$ is hyperbolic.
\end{enumerate}
Then the leafwise Poincar\'e metric is continuous.
\end{theorem}

\begin{proof}
Let us consider the function
$$F : U^0\rightarrow {\mathbb R} \quad ,\quad F(q)=\log\|
v(q)\|_{Poin}$$ introduced just before Proposition \ref{psh}. We
have to prove that $F$ is continuous (the continuity on the full $U$
is then a consequence of Proposition \ref{poles}). We have already
observed, during the proof of Proposition \ref{psh}, that $F$ is
upper semicontinuous, hence let us consider its lower
semicontinuity.

Take $q_\infty\in U^0$ and take a sequence $\{ q_n\}\subset U^0$
converging to $q_\infty$. For every $n$, let $\varphi_n : {\mathbb
D}\to X$ be a holomorphic map into $L_{q_n}\subset X$, sending
$0\in{\mathbb D}$ to $q_n\in L_{q_n}$. For every compact subset
$K\subset {\mathbb D}$, consider
$$I_K = \{ \| \varphi_n'(t)\| \ \vert\ t\in K, n\in{\mathbb N}\}
\subset {\mathbb R}$$ (the norm of $\varphi_n'$ is here computed
with the K\"ahler metric on $X$).

{\it Claim}: $I_K$ is a bounded subset of ${\mathbb R}$.

Indeed, in the opposite case we may find a subsequence $\{
n_j\}\subset{\mathbb N}$ and a sequence $\{ t_j\}\subset K$ such
that $\| \varphi_{n_j}'(t_j)\| \to +\infty$ as $j\to +\infty$. By
Brody's Reparametrization Lemma \cite[Ch. III]{Lan}, we may
reparametrize these discs so that they converge to an entire curve:
there exists maps $h_j : {\mathbb D}(r_j)\to {\mathbb D}$, with
$r_j\to +\infty$, such that the maps
$$\psi_j = \varphi_j \circ h_j : {\mathbb D}(r_j)\to X$$
converge, uniformly on compact subsets, to a nonconstant map
$$\psi : {\mathbb C}\to X.$$
It is clear that $\psi$
is tangent to ${\mathcal F}$, more precisely $\psi '(t)\in T_{\psi
(t)}{\mathcal F}$ whenever $\psi (t)\not\in Sing({\mathcal F})$,
because each $\psi_j$ has the same property. Moreover, by hypothesis
(ii) we have that the image of $\psi$ is not contained in
$Sing({\mathcal F})$. Therefore, $S=\psi^{-1}(Sing({\mathcal F}))$
is a discrete subset of ${\mathbb C}$, and $\psi ({\mathbb
C}\setminus S)$ is contained in some leaf $L^0$ of ${\mathcal F}^0$.

Take now $t_0\in S$. It corresponds to a parabolic end of $L^0$. On
a small compact disc $B$ centered at $t_0$, $\psi\vert_B$ is uniform
limit of $\psi_j\vert_B : B\to X$, which are maps into leaves of
${\mathcal F}$. If $U_T$ is a covering tube associated to some
transversal $T$ cutting $L^0$, then the maps $\psi_j\vert_B$ can be
lifted to $U_T$, in such a way that they converge on $\partial B$ to
some map which lifts $\psi\vert_{\partial B}$. The structure of
$U_T$ (absence of vanishing cycles) implies that, in fact, we have
convergence on the full $B$, to a map which lifts $\psi\vert_B$. By
doing so at every $t_0\in S$, we see that $\psi :{\mathbb C}\to X$
can be fully lifted to $U_T$, i.e. $\psi ({\mathbb C})$ is contained
in the leaf $L$ of ${\mathcal F}$ obtained by completion of $L^0$.
But this contradicts hypothesis (i), and proves the Claim.

The Claim implies now that, up to subsequencing, the maps $\varphi_n
:{\mathbb D}\to X$ converge, uniformly on compact subsets, to some
$\varphi_\infty : {\mathbb D}\to X$, with $\varphi_\infty (0) =
q_\infty$. As before, we obtain $\varphi_\infty ({\mathbb D})\subset
L_{q_\infty}$.

Recall now the extremal propery of the Poincar\'e metric: if we
write $\varphi_n'(0) = \lambda_n\cdot v(q_n)$, then $\|
v(q_n)\|_{Poin} \le \frac{1}{| \lambda_n| }$, and equality is
atteined if $\varphi_n$ is a uniformization of $L_{q_n}$. Hence,
with this choice of $\{ \varphi_n\}$, we see that
$$\| v(q_\infty )\|_{Poin} \le \frac{1}{| \lambda_\infty |} =
\lim_{n\to +\infty} \frac{1}{| \lambda_n| } = \lim_{n\to +\infty} \|
v(q_n)\|_{Poin}$$ i.e. $F(q_\infty )\le \lim_{n\to +\infty} F(q_n)$.
Due to the arbitrariness of the initial sequence $\{q_n\}$, this
gives the lower semicontinuity of $F$.
\end{proof}

Of course, due to hypothesis (i) such a result says nothing about
the possible closedness of $Sing({\mathcal F})\cup Parab({\mathcal
F})$, when $Parab({\mathcal F})$ is not empty, but at least it
leaves some hope. The above proof breaks down when there are
parabolic leaves, because Brody's lemma does not allow to control
where the limit entire curve $\psi$ is located: even if each
$\psi_j$ passes through $q_{n_j}$, it is still possible that $\psi$
does not pass through $q_\infty$, because the points in
$\psi_j^{-1}(q_{n_j})$ could exit from every compact subset of
${\mathbb C}$. Hence, the only hypothesis ``$L_{q_\infty}$ is
hyperbolic'' (instead of ``all the leaves are hyperbolic'') is not
sufficient to get a contradiction and prove the Claim. In other
words, the (parabolic) leaf $L$ appearing in the Claim above could
be ``far'' from $q_\infty$, but still could have some influence on
the possible discontinuity of the leafwise Poincar\'e metric at
$q_\infty$.

The subset $Sing({\mathcal F})\cup Parab({\mathcal F})$ being
complete pluripolar, a natural question concerns the computation of
its Lelong numbers. For instance, if these Lelong numbers were
positive, then, by Siu Theorem \cite{Dem}, we should get that
$Sing({\mathcal F})\cup Parab({\mathcal F})$ is a countable union of
analytic subsets, a substantial step toward the conjecture above.
However, we generally expect that these Lelong numbers are zero,
even when $Sing({\mathcal F})\cup Parab({\mathcal F})$ is analytic.

\begin{example} \label{turbulent} {\rm
Let $E$ be an elliptic curve and let $X={\mathbb P}\times E$. Let
$\alpha = f(z)dz$ be a meromorphic 1-form on ${\mathbb P}$, with
poles $P=\{ z_1,\ldots ,z_k\}$ of orders $\{ \nu_1,\ldots ,\nu_k\}$.
Consider the (nonsingular) foliation ${\mathcal F}$ on $X$ defined
by the (saturated) Kernel of the meromorphic 1-form $\beta = f(z)dz
- dw$, i.e. by the differential equation $\frac{dw}{dz}=f(z)$. Then
each fiber $\{ z_j\}\times E$, $z_j\in P$, is a leaf of ${\mathcal
F}$, whereas each other fiber $\{ z\}\times E$, $z\not\in P$, is
everywhere transverse to ${\mathcal F}$. In \cite{Br1} such a
foliation is called {\it turbulent}. Outside the elliptic leaves
$P\times E$, every leaf is a regular covering of ${\mathbb
P}\setminus P$, by the projection $X\to {\mathbb P}$. Hence, if
$k\ge 3$ then these leaves are hyperbolic, and their Poincar\'e
metric coincides with the pull-back of the Poincar\'e metric on
${\mathbb P}\setminus P$.

Take a point $(z_j,w)\in P\times E = Parab({\mathcal F})$. Around
it, the foliation is generated by the holomorphic and nonvanishing
vector field $v = f(z)^{-1}\frac{\partial}{\partial z} +
\frac{\partial}{\partial w}$, whose $z$-component has at $z=z_j$ a
zero of order $\nu_j$. The weight $F=\log\| v\|_{Poin}$ is nothing
but than the pull-back of $\log\| f(z)^{-1}\frac{\partial}{\partial
z}\|_{Poin}$, where the norm is measured in the Poincar\'e metric of
${\mathbb P}\setminus P$. Recalling that the Poincar\'e metric of
the punctured disc ${\mathbb D}^*$ is $\frac{idz\wedge d\bar
z}{|z|^2 (\log |z|^2)^2}$, we see that $F$ is something like
$$\log |z-z_j|^{\nu_j-1} - \log \big| \log |z-z_j|^2 \big| .$$
Hence the Lelong number along $\{ z_j\}\times E$ is positive if and
only if $\nu_j\ge 2$, which can be considered as an ``exceptional''
case; in the ``generic'' case $\nu_j=1$ the pole of $F$ along $\{
z_j\}\times E$ is a weak one, with vanishing Lelong number.}
\end{example}

\begin{remark} {\rm
We used the convexity property stated by Theorem \ref{convexity} as
a substitute of the Stein property required by the results of
Nishino, Yamaguchi, Kizuka discussed in Section 2. One could ask if,
after all, such a convexity property can be used to prove the
Steinness of $U_T$, when $T$ is Stein. If the ambient manifold $X$
is Stein, instead of K\"ahler compact, Il'yashenko proved in
\cite{Il1} and \cite{Il2} (see Section 2) that indeed $U_T$ is
Stein, using Cartan-Thullen-Oka convexity theory over Stein
manifolds. See also \cite{Suz} for a similar approach to $V_T$,
\cite{Br6} for some result in the case of projective manifolds,
close in spirit to \cite{Il2}, and \cite{Nap} and \cite{Ohs} for
related results in the case of proper fibrations by curves.

For instance, suppose that all the fibers of $U_T$ are hyperbolic,
and that the fiberwise Poincar\'e metric is of class $C^2$. Then we
can take the function $\psi : U_T\to{\mathbb R}$ defined by $\psi =
\psi_0 +\varphi\circ P_T$, where $\psi_0(q)$ is the squared
hyperbolic distance (in the fiber) between $q$ and the basepoint
$p_T(P_T(q))$, and $\varphi :T\to{\mathbb R}$ is a strictly
plurisubharmonic exhaustion of $T$. A computation shows that $\psi$
is strictly plurisubharmonic (thanks to the plurisubharmonic
variation of the fiberwise Poincar\'e metric on $U_T$), and being
also exhaustive we deduce that $U_T$ is Stein. Probably, this can be
done also if the fiberwise Poincar\'e metric is less regular, say
$C^0$. But when there are parabolic fibers such a simple argument
cannot work, because $\psi$ is no more exhaustive (one can try
perhaps to use a renormalization argument like the one used in the
proof of Theorem \ref{nishino}). However, if {\it all} the fibers
are parabolic then we shall see later that $U_T$ is a product
$T\times{\mathbb C}$ (if $T$ is small), and hence it is Stein.

A related problem concerns the existence on $U_T$ of holomorphic
functions which are not constant on the fibers. By Corollary
\ref{pseudoeffective}, $K_{\mathcal F}$ is pseudoeffective, if
${\mathcal F}$ is hyperbolic. Let us assume a little more, namely
that it is effective. Then any nontrivial section of $K_{\mathcal
F}$ over $X$ can be lifted to $U_T$, giving a holomorphic section of
the relative canonical bundle of the fibration. As in Lemmata
\ref{separability} and \ref{connectivity}, this section can be
integrated along the (simply connected and pointed) fibers, giving a
holomorphic function on $U_T$ not constant on generic fibers.}
\end{remark}

\section{Extension of meromorphic maps from line bundles}

In order to generalize Corollary \ref{pseudoeffective} to cover
(most) parabolic foliations, we need an extension theorem for
certain meromorphic maps. This is done in the present Section,
following \cite{Br5}.

\subsection{Volume estimates}

Let us firstly recall some results of Dingoyan \cite{Din}, in a
slightly simplified form due to our future use.

Let $V$ be a connected complex manifold, of dimension $n$, and let
$\omega$ be a smooth closed semipositive $(1,1)$-form on $V$ (e.g.,
the pull-back of a K\"ahler form by some holomorphic map from $V$).
Let $U\subset V$ be an open subset, with boundary $\partial U$
compact in $V$. Suppose that the mass of $\omega^n$ on $U$ is
finite: $\int_U\omega^n < +\infty$. We look for some condition
ensuring that also the mass on $V$ is finite: $\int_V\omega^n <
+\infty$. In other words, we look for the boundedness of the
$\omega^n$-volume of the ends $V\setminus U$.

Set
$$P_\omega (V,U) = \{ \varphi : V\to [-\infty ,+\infty )\ {\rm u.s.c}
\ \vert\ dd^c\varphi +\omega \ge 0,\ \varphi\vert_U\le 0\}$$ where
u.s.c. means upper semicontinuous, and the first inequality is in
the sense of currents. This first inequality defines the so-called
{\it $\omega$-plurisubharmonic functions}. Note that locally the
space of $\omega$-plurisubharmonic functions can be identified with
a translation of the space of the usual plurisubharmonic functions:
locally the form $\omega$ admits a smooth potential $\phi$ ($\omega
= dd^c\phi$), and so $\varphi$ is $\omega$-plurisubharmonic is and
only if $\varphi + \phi$ is plurisubharmonic. In this way, most
local problems on $\omega$-plurisubharmonic functions can be reduced
to more familiar problems on plurisubharmonic functions.

Remark that the space $P_\omega (V,U)$ is not empty, for it contains
at least all the constant nonpositive functions on $V$.

Suppose that $P_\omega (V,U)$ satisfies the following condition:
\begin{enumerate}
\item[(A)] the functions in $P_\omega (V,U)$ are locally uniformly
bounded from above: for every $z\in V$ there exists a neighbourhood
$V_z\subset V$ of $z$ and a constant $c_z$ such that
$\varphi\vert_{V_z}\le c_z$ for every $\varphi\in P_\omega (V,U)$.
\end{enumerate}
Then we can introduce the upper envelope
$$\Phi (z) = \sup_{\varphi\in P_\omega (V,U)}\varphi (z) \qquad
\forall z\in V$$ and its upper semicontinuous regularization
$$\Phi^*(z) = \limsup_{w\to z}\Phi (w)\qquad \forall z\in V.$$
The function
$$\Phi^* : V\to [0,+\infty )$$
is identically zero on $U$, upper semicontinuous, and
$\omega$-plurisubharmonic (Brelot-Cartan \cite{Kli}), hence it
belongs to the space $P_\omega (V,U)$. Moreover, by results of
Bedford and Taylor \cite{BeT} \cite{Kli} the wedge product
$(dd^c\Phi^* +\omega )^n$ is well defined, as a locally finite
measure on $V$, and it is identically zero outside $\overline U$:
$$(dd^c\Phi^* +\omega )^n \equiv 0\qquad {\rm on}\
V\setminus\overline U.$$

Indeed, let $B\subset V\setminus\overline U$ be a ball around which
$\omega$ has a potential. Let $P_\omega (B,\Phi^*)$ be the space of
$\omega$-plurisubharmonic functions $\psi$ on $B$ such that
$\limsup_{z\to w}\psi (z)\le \Phi^*(w)$ for every $w\in\partial B$.
Let $\Psi^*$ be the regularized upper envelope of the family
$P_\omega (B,\Phi^*)$ (which is bounded from above by the maximum
principle). Remark that $\Phi^*\vert_B$ belongs to $P_\omega
(B,\Phi^*)$, and so $\Psi^*\ge \Phi^*$ on $B$. By \cite{BeT},
$\Psi^*$ satisfies the homogeneous Monge-Amp\`ere equation
$(dd^c\Psi^* +\omega )^n =0$ on $B$, with Dirichlet boundary
condition $\limsup_{z\to w}\Psi^*(z) = \Phi^*(w)$, $w\in\partial B$
(``balayage''). Then the function $\widetilde\Phi^*$ on $V$, which
is equal to $\Psi^*$ on $B$ and equal to $\Phi^*$ on
$V\setminus\overline B$, still belongs to $P_\omega (V,U)$, and it
is everywhere not smaller than $\Phi^*$. Hence, by definition of
$\Phi^*$, we must have $\widetilde\Phi^* = \Phi^*$, i.e.
$\Phi^*=\Psi^*$ on $B$ and so $\Phi^*$ satisfies the homogeneous
Monge-Amp\`ere equation on $B$.

Suppose now that the following condition is also satisfied:
\begin{enumerate}
\item[(B)] $\Phi^*:V\rightarrow [0,+\infty )$ is exhaustive on
$V\setminus U$: for every $c > 0$, the subset $\{ \Phi^* <
c\}\setminus U$ is relatively compact in $V\setminus U$.
\end{enumerate}
Roughly speaking, this means that the function $\Phi^*$ solves on
$V\setminus \overline U$ the homogeneous Monge-Amp\`ere equation,
with boundary conditions 0 on $\partial U$ and $+\infty$ on the
``boundary at infinity'' of $V\setminus U$.

\begin{theorem} \label{dingoyan} \cite{Din}
Under assumptions (A) and (B), the $\omega^n$-volume of $V$ is
finite:
$$\int_V \omega^n < +\infty .$$
\end{theorem}

\begin{proof}
The idea is that, using $\Phi^*$, we can push all the mass of
$\omega^n$ on $V\setminus U$ to the compact set $\partial U$. Note
that we certainly have
$$\int_V (dd^c\Phi^* +\omega )^n < +\infty$$
because, after decomposing $V= U\cup (V\setminus U)\cup \partial U$:
(i) $\Phi^*\equiv 0$ on $U$, and $\int_U\omega^n$ is finite by
standing assumptions; (ii) $(dd^c\Phi^* +\omega )^n\equiv 0$ on
$V\setminus U$; (iii) $\partial U$ is compact (but, generally
speaking, $\partial U$ is charged by the measure $(dd^c\Phi^*+\omega
)^n$).

Hence the theorem follows from the next inequality.

\begin{lemma}\label{inequality} \cite[Lemma 4]{Din}
$$\int_V \omega^n \le \int_V (dd^c\Phi^* +\omega )^n .$$
\end{lemma}

\begin{proof}
More generally, we shall prove that for every $k=0,\ldots ,n-1$:
$$\int_V (dd^c\Phi^* +\omega )^{k+1}\wedge\omega^{n-k-1} \ge \int_V
(dd^c\Phi^* +\omega )^k\wedge\omega^{n-k} ,$$ so that the desired
inequality follows by concatenation. We can decompose the integral
on the left hand side as
$$\int_V (dd^c\Phi^* +\omega )^k\wedge\omega^{n-k} + \int_V
dd^c\Phi^*\wedge (dd^c\Phi^* +\omega )^k\wedge\omega^{n-k-1}$$ and
so we need to prove that, setting $\eta =(dd^c\Phi^* +\omega
)^k\wedge\omega^{n-k-1}$, we have
$$I = \int_V dd^c\Phi^*\wedge\eta \ge 0 .$$
Here all the wedge products are well defined, because $\Phi^*$ is
locally bounded, and moreover $\eta$ is a closed positive current of
bidegree $(n-1,n-1)$ \cite{Kli}.

Take a sequence of smooth functions $\chi_n : {\mathbb R}\to [0,1]$,
$n\in{\mathbb N}$, such that $\chi_n(t)=1$ for $t\le n$,
$\chi_n(t)=0$ for $t\ge n+1$, and $\chi_n'(t)\le 0$ for every $t$.
Thus, for every $z\in V$ we have $(\chi_n\circ\Phi^*)(z)=0$ for
$n\le\Phi^*(z)-1$ and $(\chi_n\circ\Phi^*)(z)=1$ for
$n\ge\Phi^*(z)$. Hence it is sufficient to prove that
$$I_n = \int_V (\chi_n\circ\Phi^*)\cdot dd^c\Phi^*\wedge\eta \ge 0$$
for every $n$. By assumption (B), the support of $\chi_n\circ\Phi^*$
intersects $V\setminus U$ along a compact subset. Moreover, $\Phi^*$
is identically zero on $U$. Thus, the integrand above has compact
support in $V$. Hence, by Stokes formula,
$$I_n = -\int_V d(\chi_n\circ\Phi^*)\wedge d^c\Phi^*\wedge\eta =
-\int_V (\chi_n'\circ\Phi^*)\cdot d\Phi^*\wedge d^c\Phi^*\wedge\eta
.$$ Now, $d\Phi^*\wedge d^c\Phi^*$ is a positive current, and its
product with $\eta$ is a positive measure. From $\chi_n'\le 0$ we
obtain $I_n\ge 0$, for every $n$.
\end{proof}
This inequality completes the proof of the theorem.
\end{proof}

\subsection{Extension of meromorphic maps}

As in \cite[\S 6]{Din}, we shall use the volume estimate of Theorem
\ref{dingoyan} to get an extension theorem for certain meromorphic
maps into K\"ahler manifolds.

Consider the following situation. It is given a compact connected
K\"ahler manifold $X$, of dimension $n$, and a line bundle $L$ on
$X$. Denote by $E$ the total space of $L$, and by $\Sigma\subset E$
the graph of the null section of $L$. Let $U_\Sigma\subset E$ be a
connected (tubular) neighbourhood of $\Sigma$, and let $Y$ be
another compact K\"ahler manifold, of dimension $m$.

\begin{theorem} \label{extension} \cite{Br5}
Suppose that $L$ is not pseudoeffective. Then any meromorphic map
$$f : U_\Sigma \setminus\Sigma \dashrightarrow Y$$
extends to a meromorphic map
$$\bar f : U_\Sigma \dashrightarrow Y .$$
\end{theorem}

Before the proof, let us make a link with \cite{BDP}. In the special
case where $X$ is projective, and not only K\"ahler, the non
pseudoeffectivity of $L$ translates into the existence of a covering
family of curves $\{ C_t\}_{t\in B}$ on $X$ such that $L\vert_{C_t}$
has negative degree for every $t\in B$ \cite{BDP}. This means that
the normal bundle of $\Sigma$ in $E$ has negative degree on every
$C_t\subset\Sigma\simeq X$. Hence the restriction of $E$ over $C_t$
is a surface $E_t$ which contains a compact curve $\Sigma_t$ whose
selfintersection is negative, and thus contractible to a normal
singularity. By known results \cite{Siu} \cite{Iv1}, every
meromorphic map from $U_t\setminus\Sigma_t$ ($U_t$ being a
neighbourhood of $\Sigma_t$ in $E_t$) into a compact K\"ahler
manifold can be meromorphically extended to $U_t$. Because the
curves $C_t$ cover the full $X$, this is sufficient to extends from
$U_\Sigma\setminus \Sigma$ to $U_\Sigma$.

Of course, if $X$ is only K\"ahler then such a covering family of
curves could not exist, and we need a more global approach, which
avoids any restriction to curves. Even in the projective case, this
seems a more natural approach than evoking \cite{BDP}.

\begin{proof}
We begin with a simple criterion for pseudoeffectivity, analogous to
the well known fact that a line bundle is ample if and only if its
dual bundle has strongly pseudoconvex neighbourhoods of the null
section. Recall that an open subset $W$ of a complex manifold $E$ is
{\it locally pseudoconvex in} $E$ if for every $w\in\partial W$
there exists a neighbourhood $U_w\subset E$ of $w$ such that $W\cap
U_w$ is Stein.

\begin{lemma}\label{criterion}
Let $X$ be a compact connected complex manifold and let $L$ be a
line bundle on $X$. The following two properties are equivalent:
\begin{enumerate}
\item[(i)] $L$ is pseudoeffective;
\item[(ii)] in the total space $E^*$ of the dual line bundle $L^*$
there exists a neighbourhood $W\not= E^*$ of the null section
$\Sigma^*$ which is locally pseudoconvex in $E^*$.
\end{enumerate}
\end{lemma}

\begin{proof}
The implication (i) $\Rightarrow$ (ii) is quite evident. If $h$ is a
(singular) hermitian metric on $L$ with positive curvature
\cite{Dem}, then in a local trivialization $E\vert_{U_j}\simeq
U_j\times{\mathbb C}$ the unit ball is expressed by $\{ (z,t)\
\vert\ |t| < e^{h_j(z)}\}$, where $h_j:U_j\to [-\infty ,+\infty )$
is the plurisubharmonic weight of $h$. In the dual local
trivialization $E^*\vert_{U_j}\simeq U_j\times{\mathbb C}$, the unit
ball of the dual metric is expressed by $\{ (z,s)\ \vert\ |s| <
e^{-h_j(z)}\}$. The plurisubharmonicity of $h_j$ gives (and is
equivalent to) the Steinness of such an open subset of
$U_j\times{\mathbb C}$ (recall Hartogs Theorem on Hartogs Tubes
mentioned in Section 2). Hence we get (ii), with $W$ equal to the
unit ball in $E^*$.

The implication (ii) $\Rightarrow$ (i) is not more difficult. Let
$W\subset E^*$ be as in (ii). On $E^*$ we have a natural ${\mathbb
S}^1$-action, which fixes $\Sigma^*$ and rotates each fiber. For
every $\vartheta\in{\mathbb S}^1$, let $W_\vartheta$ be the image of
$W$ by the action of $\vartheta$. Then
$$W' = \cap_{\vartheta\in{\mathbb S}^1} W_\vartheta$$
is still a nontrivial locally pseudoconvex neighbourhood of
$\Sigma^*$, for local pseudoconvexity is stable by intersections.
For every $z\in X$, $W'$ intersects the fiber $E_z^*$ along an open
subset which is ${\mathbb S}^1$-invariant, a connected component of
which is therefore a disc $W_z^0$ centered at the origin (possibly
$W_z^0=E_z^*$ for certain $z$, but not for all); the other
components are annuli around the origin. Using the local
pseudoconvexity of $W'$, i.e. its Steinness in local trivializations
$E^*\vert_{U_j}\simeq U_j\times{\mathbb C}$, it is easy to see that
these annuli and discs cannot merge when $z$  moves in $X$. In other
words,
$$W'' = \cup_{z\in X}W_z^0$$ is a connected component of
$W'$, and of course it is still a nontrivial pseudoconvex
neighbourhood of $\Sigma^*$. We may use $W''$ as unit ball for a
metric on $L^*$. As in the first part of the proof, the
corresponding dual metric on $L$ has positive curvature, in the
sense of currents.
\end{proof}

Consider now, in the space $U_\Sigma\times Y$, the graph $\Gamma_f$
of the meromorphic map $f: U_\Sigma^0 = U_\Sigma\setminus\Sigma
\dashrightarrow Y$. By definition of meromorphicity, $\Gamma_f$ is
an irreducible analytic subset of $U_\Sigma^0\times Y \subset
U_\Sigma\times Y$, whose projection to $U_\Sigma^0$ is proper and
generically bijective. It may be singular, and in that case we
replace it by a resolution of its singularities, still denoted by
$\Gamma_f$. The (new) projection
$$\pi :\Gamma_f \to U_\Sigma^0$$
is a proper map, and it realizes an isomorphism between
$\Gamma_f\setminus Z$ and $U_\Sigma^0\setminus B$, for suitable
analytic subsets $Z\subset \Gamma_f$ and $B\subset U_\Sigma^0$, with
$B$ of codimension at least two.

The manifold $U_\Sigma\times Y$ is K\"ahler. The K\"ahler form
restricted to the graph of $f$ and pulled-back to its resolution
gives a smooth, semipositive, closed $(1,1)$-form $\omega$ on
$\Gamma_f$. Fix a smaller (tubular) neighbourhood $U_\Sigma'$ of
$\Sigma$, and set $U_0 = U_\Sigma\setminus\overline{U_\Sigma'}$,
$U=\pi^{-1}(U_0)\subset \Gamma_f$. Up to restricting a little the
initial $U_\Sigma$, we may assume that the $\omega^{n'}$-volume of
the shell $U$ is finite ($n'=n+1=\dim\Gamma_f$). Our aim is to prove
that
$$\int_{\Gamma_f} \omega^{n'} < +\infty .$$
Indeed, this is the volume of the graph of $f$. Its finiteness,
together with the analyticity of the graph in $U_\Sigma^0\times Y$,
imply that the closure of that graph in $U_\Sigma\times Y$ is still
an analytic subset of dimension $n'$, by Bishop's extension theorem
\cite{Siu} \cite{Chi}. This closure, then, is the graph of the
desired meromorphic extension $\bar f : U_\Sigma\dashrightarrow Y$.

We shall apply Theorem \ref{dingoyan}. Hence, consider the space
$P_\omega (\Gamma_f,U)$ of $\omega$-plurisub\-harmonic functions on
$\Gamma_f$, nonpositive on $U$, and let us check conditions (A) and
(B) above.

Consider the open subset $\Omega\subset\Gamma_f$ where the functions
of $P_\omega (\Gamma_f,U)$ are locally uniformly bounded from above.
It contains $U$, and it is a general fact that it is locally
pseudoconvex in $\Gamma_f$ \cite[\S 3]{Din}. Therefore $\Omega
'=\Omega\cap (\Gamma_f\setminus Z)$ is locally pseudoconvex in
$\Gamma_f\setminus Z$. Its isomorphic projection $\pi (\Omega ')$ is
therefore locally pseudoconvex in $U_\Sigma^0\setminus B$. Classical
characterizations of pseudoconvexity \cite[II.5]{Ran} show that
$\Omega_0 = {\rm interior}\{ \pi (\Omega ')\cup B\}$ is locally
pseudoconvex in $U_\Sigma^0$. From $\Omega\supset U$ we also have
$\Omega_0\supset U_0$.

\begin{figure}[h]
\includegraphics[width=8cm,height=5cm]{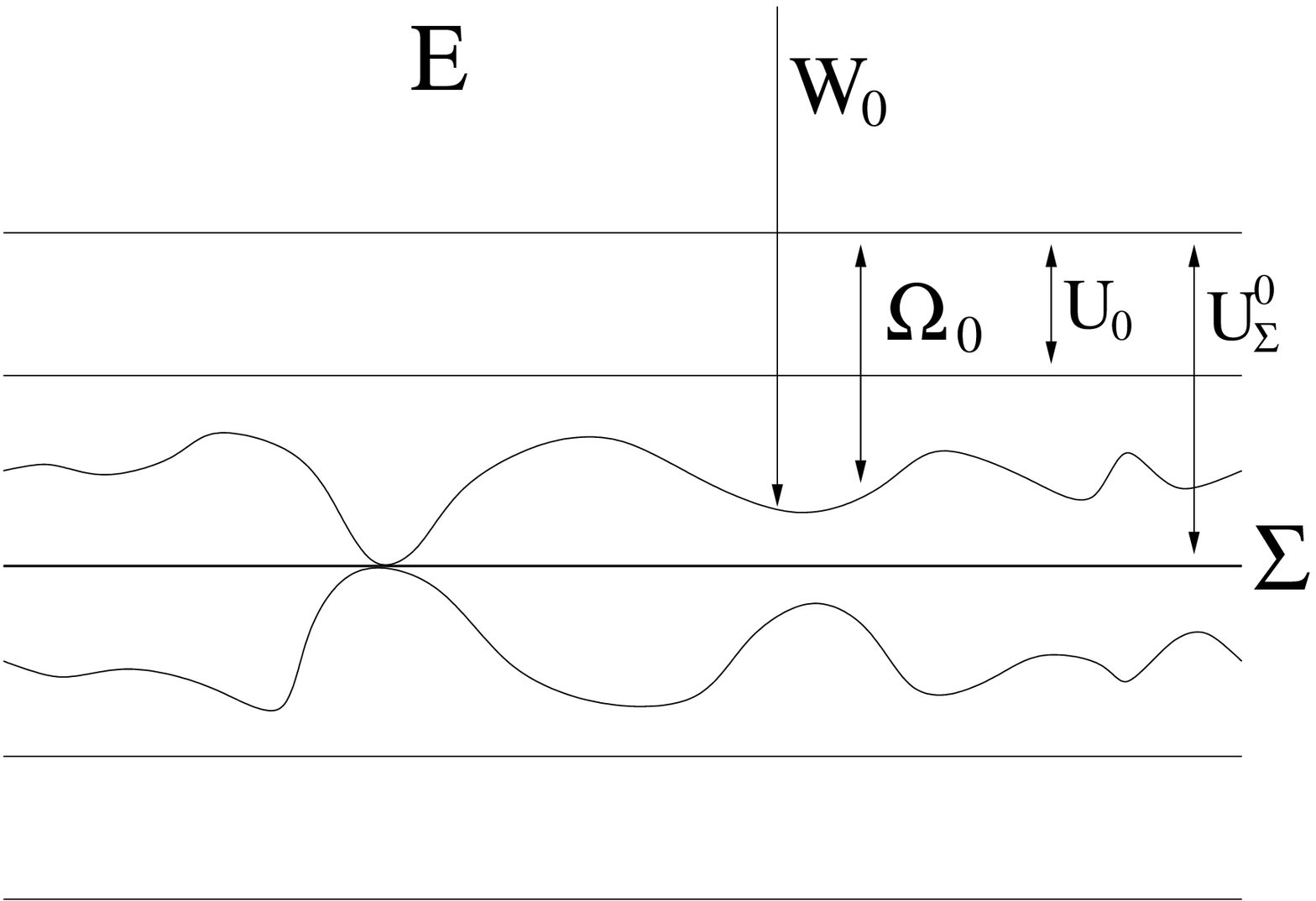}
\end{figure}

Take now in $E$ the neighbourhood of infinity $W_0 = \Omega_0\cup
(E\setminus U_\Sigma )$. Because $E\setminus \Sigma$ is naturally
isomorphic to $E^*\setminus\Sigma^*$, the isomorphism exchanging
null sections and sections at infinity, we can see $W_0$ as an open
subset of $E^*$, so that $W = W_0\cup\Sigma^*$ is a neighbourhood of
$\Sigma^*$, locally pseudoconvex in $E^*$. Because $L$ is not
pseudoeffective by assumption, Lemma \ref{criterion} says that
$W=E^*$. That is, $\Omega_0 = U_\Sigma^0$.

This implies that the original $\Omega\subset\Gamma_f$ contains, at
least, $\Gamma_f\setminus Z$. But, by the maximum principle, a
family of $\omega$-plurisubharmonic functions locally bounded
outside an analytic subset is automatically bounded also on the same
analytic subset. Therefore $\Omega = \Gamma_f$, and condition (A) of
Theorem \ref{dingoyan} is fulfilled.

Condition (B) is simpler \cite[\S 4]{Din}. We just have to exhibit a
$\omega$-pluri\-sub\-harmonic function on $\Gamma_f$ which is
nonpositive on $U$ and exhaustive on $\Gamma_f\setminus U$. On
$U_\Sigma^0$ we take the function
$$\psi (z) = -\log dist (z,\Sigma )$$
where $dist(\cdot ,\Sigma )$ is the distance function from $\Sigma$,
with respect to the K\"ahler metric $\omega_0$ on $U_\Sigma$.
Classical estimates (Takeuchi) give $dd^c\psi \ge -C\cdot\omega_0$,
for some positive constant $C$. Thus
$$dd^c(\psi\circ\pi )\ge -C\cdot\pi^*(\omega_0) \ge -C\cdot
\omega $$ because $\omega\ge\pi^*(\omega_0)$. Hence
$\frac{1}{C}(\psi\circ\pi )$ is $\omega$-plurisubharmonic on
$\Gamma_f$. For a sufficiently large $C'>0$,
$\frac{1}{C}(\psi\circ\pi ) -C'$ is moreover negative on $U$, and it
is exhaustive on $\Gamma_f\setminus U$. Thus condition (B) is
fulfilled.

Finally we can apply Theorem \ref{dingoyan}, obtain the finiteness
of the volume of the graph of $f$, and conclude the proof of the
theorem.
\end{proof}

\begin{remark} {\rm
We think that Theorem \ref{extension} should be generalizable to the
following ``nonlinear'' statement: if $U_\Sigma$ is any K\"ahler
manifold and $\Sigma\subset U_\Sigma$ is a compact hypersurface
whose normal bundle is not pseudoeffective, then any meromorphic map
$f:U_\Sigma\setminus\Sigma\dashrightarrow Y$ ($Y$ K\"ahler compact)
extends to $\bar f :U_\Sigma\dashrightarrow Y$. The difficulty is to
show that a locally pseudoconvex subset of $U_\Sigma^0 =
U_\Sigma\setminus\Sigma$ like $\Omega_0$ in the proof above can be
``lifted'' in the total space of the normal bundle of $\Sigma$,
preserving the local pseudoconvexity.}
\end{remark}

\section{Parabolic foliations}

We can now return to foliations.

As usual, let $X$ be a compact connected K\"ahler manifold, $\dim
X=n$, and let ${\mathcal F}$ be a foliation by curves on $X$
different from a rational quasi-fibration. Let us start with some
general remarks, still following \cite{Br5}.

\subsection{Global tubes}

The construction of holonomy tubes and covering tubes given in
Section 4 can be easily modified by replacing the transversal
$T\subset X^0$ with the full $X^0$. That is, all the holonomy
coverings $\widehat{L_p}$ and universal coverings $\widetilde{L_p}$,
$p\in X^0$, can be glued together, without the restriction $p\in T$.
The results are complex manifolds $V_{\mathcal F}$ and $U_{\mathcal
F}$, of dimension $n+1$, equipped with submersions
$$Q_{\mathcal F} : V_{\mathcal F}\rightarrow X^0\quad ,
\qquad P_{\mathcal F} : U_{\mathcal F}\rightarrow X^0$$ sections
$$q_{\mathcal F} : X^0\rightarrow V_{\mathcal F}\quad ,
\qquad p_{\mathcal F} : X^0\rightarrow U_{\mathcal F}$$ and
meromorphic maps
$$\pi_{\mathcal F} : V_{\mathcal F}\dashrightarrow X\quad ,
\qquad \Pi_{\mathcal F} : U_{\mathcal F}\dashrightarrow X$$ such
that, for any transversal $T\subset X^0$, we have $Q_{\mathcal
F}^{-1}(T)=V_T$, $q_{\mathcal F}\vert_T = q_T$, $\pi_{\mathcal
F}\vert_{Q_{\mathcal F}^{-1}(T)}=\pi_T$, etc.

Remark that if $D\subset X^0$ is a small disc contained in some leaf
$L_p$ of ${\mathcal F}$, $p\in D$, then $Q_{\mathcal F}^{-1}(D)$ is
naturally isomorphic to the product $\widehat{L_p}\times D$: for
every $q\in D$, $\widehat{L_q}$ is the same as $\widehat{L_p}$, but
with a different basepoint. More precisely, thinking to points of
$\widehat{L_q}$ as equivalence classes of paths starting at $q$, we
see that for every $q\in D$ the isomorphism between $\widehat{L_q}$
and $\widehat{L_p}$ is completely canonical, once $D$ is fixed and
because $D$ is contractible. This means that $D$ can be lifted, in a
canonical way, to a foliation by discs in $Q_{\mathcal F}^{-1}(D)$,
transverse to the fibers. In this way, by varying $D$ in $X^0$, we
get in the full space $V_{\mathcal F}$ a nonsingular foliation by
curves $\widehat{\mathcal F}$, which projects by $Q_{\mathcal F}$ to
${\mathcal F}^0$.

If $\gamma :[0,1]\to X^0$ is a loop in a leaf, $\gamma (0)=\gamma
(1)=p$, then this foliation $\widehat{\mathcal F}$ permits to define
a {\it monodromy map} of the fiber $\widehat{L_p}$ into itself. This
monodromy map is just the covering transformation of $\widehat{L_p}$
corresponding to $\gamma$ (which may be trivial, if the holonomy of
the foliation along $\gamma$ is trivial).

In a similar way, in the space $U_{\mathcal F}$ we get a canonically
defined nonsingular foliation by curves $\widetilde{\mathcal F}$,
which projects by $P_{\mathcal F}$ to ${\mathcal F}^0$. And we have
a fiberwise covering
$$F_{\mathcal F} : U_{\mathcal F}\rightarrow V_{\mathcal F}$$
which is a local diffeomorphism, sending $\widetilde{\mathcal F}$ to
$\widehat{\mathcal F}$.

In the spaces $U_{\mathcal F}$ and $V_{\mathcal F}$ we also have the
graphs of the sections $p_{\mathcal F}$ and $q_{\mathcal F}$. They
are {\it not} invariant by the foliations $\widetilde{\mathcal F}$
and $\widehat{\mathcal F}$: in the notation above, with $D$ in a
leaf and $p,q\in D$, the basepoint $q_{\mathcal
F}(q)\in\widehat{L_q}$ corresponds to the constant path $\gamma
(t)\equiv q$, whereas the point of $\widehat{L_q}$ in the same leaf
(of $\widehat{\mathcal F}$) of $q_{\mathcal F}(p)\in\widehat{L_p}$
corresponds to the class of a path in $D$ from $q$ to $p$. In fact,
the graphs $p_{\mathcal F}(X^0)\subset U_{\mathcal F}$ and
$q_{\mathcal F}(X^0)\subset V_{\mathcal F}$ are hypersurfaces
everywhere transverse to $\widetilde{\mathcal F}$ and
$\widehat{\mathcal F}$.

\begin{figure}[h]
\includegraphics[width=11cm,height=5cm]{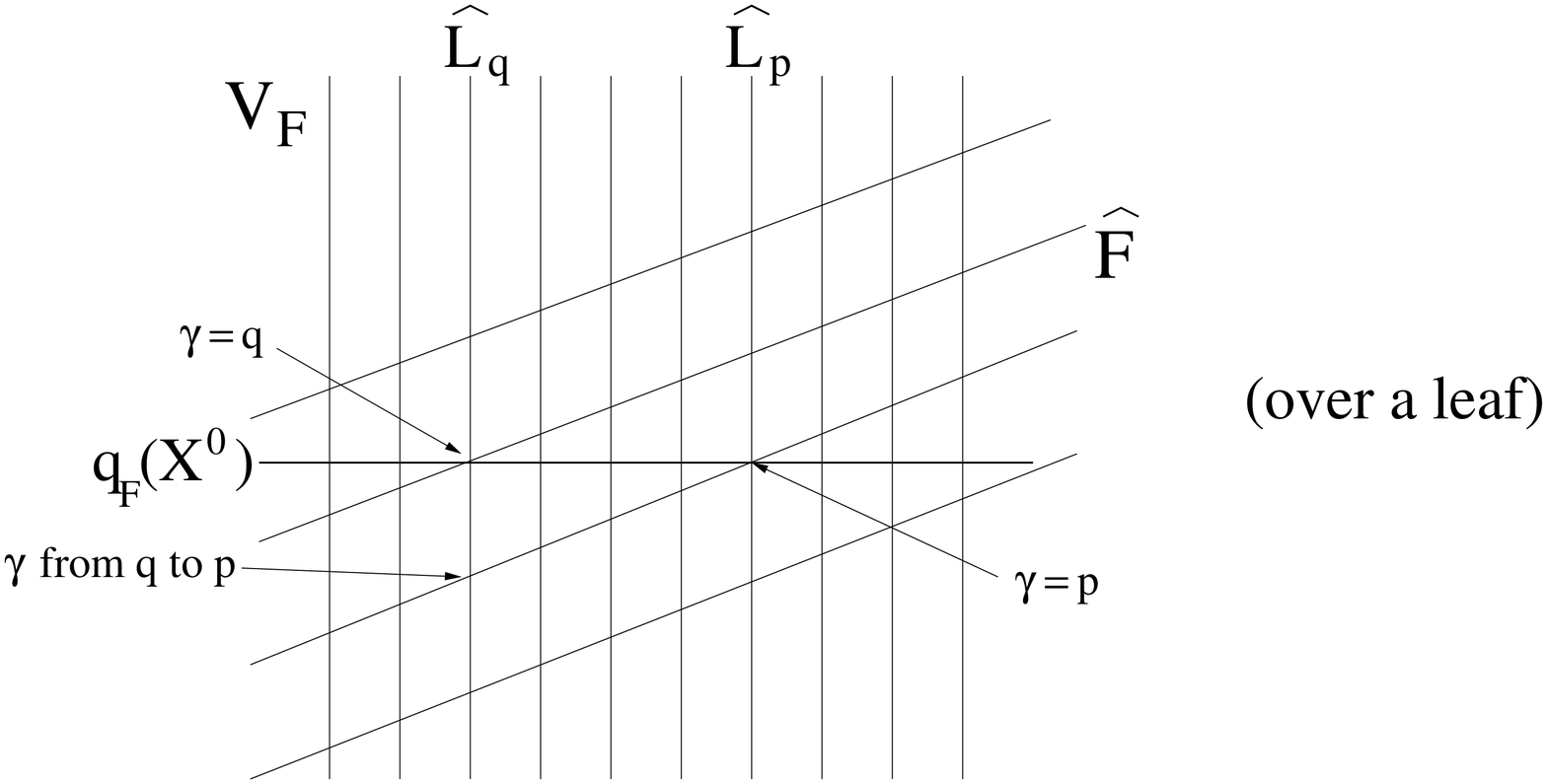}
\end{figure}

A moment of reflection shows also the following fact: the normal
bundle of the hypersurface $p_{\mathcal F}(X^0)$ in $U_{\mathcal F}$
(or $q_{\mathcal F}(X^0)$ in $V_{\mathcal F}$) is naturally
isomorphic to $T_{\mathcal F}\vert_{X^0}$, the tangent bundle of the
foliation restricted to $X^0$. That is, the manifold $U_{\mathcal
F}$ (resp. $V_{\mathcal F}$) can be thought as an ``integrated
form'' of the (total space of the) tangent bundle of the foliation,
in which tangent lines to the foliation are replaced by universal
coverings (resp. holonomy coverings) of the corresponding leaves.
From this perspective, which will be useful below, the map
$\Pi_{\mathcal F}:U_{\mathcal F}\dashrightarrow X$ is a sort of
``skew flow'' associated to ${\mathcal F}$, in which the ``time''
varies not in ${\mathbb C}$ but in the universal covering of the
leaf.

Let us conclude this discussion with a trivial but illustrative
example.

\begin{example}\label{trivial} {\rm
Suppose $n=1$, i.e. $X$ is a compact connected curve and ${\mathcal
F}$ is the foliation with only one leaf, $X$ itself. The manifold
$V_{\mathcal F}$ is composed by equivalence classes of paths in $X$,
where two paths are equivalent if they have the same starting point
and the same ending point (here holonomy is trivial!). Clearly,
$V_{\mathcal F}$ is the product $X\times X$, $Q_{\mathcal F}$ is the
projection to the first factor, $q_{\mathcal F}$ is the diagonal
embedding of $X$ into $X\times X$, and $\pi_{\mathcal F}$ is the
projection to the second factor. Note that the normal bundle of the
diagonal $\Delta\subset X\times X$ is naturally isomorphic to $TX$.
The foliation $\widehat{\mathcal F}$ is the horizontal foliation,
and note that its monodromy is trivial, corresponding to the fact
that the holonomy of the foliation is trivial. The manifold
$U_{\mathcal F}$ is the fiberwise universal covering of $V_{\mathcal
F}$, with basepoints on the diagonal. It is {\it not} the product of
$X$ with the universal covering $\widetilde{X}$ (unless $X={\mathbb
P}$, of course). It is only a locally trivial $\widetilde{X}$-bundle
over $X$. The foliation $\widetilde{\mathcal F}$ has nontrivial
monodromy: if $\gamma :[0,1]\to X$ is a loop based at $p$, then the
monodromy of $\widetilde{\mathcal F}$ along $\gamma$ is the covering
transformation of the fiber over $p$ (i.e. the universal covering of
$X$ with basepoint $p$) associated to $\gamma$. The foliation
$\widetilde{\mathcal F}$ can be described as the suspension of the
natural representation $\pi_1 (X)\to Aut(\widetilde{X})$
\cite{CLN}.}
\end{example}

\subsection{Parabolic foliations}

After these preliminaries, let us concentrate on the class of {\bf
parabolic foliations}, i.e. let us assume that all the leaves of
${\mathcal F}$ are uniformised by ${\mathbb C}$. In this case, the
Poincar\'e metric on the leaves is identically zero, hence quite
useless. But our convexity result Theorem \ref{convexity} still
gives a precious information on covering tubes.

\begin{theorem} \label{trivialtube}
Let $X$ be a compact connected K\"ahler manifold and let ${\mathcal
F}$ be a parabolic foliation on $X$. Then the global covering tube
$U_{\mathcal F}$ is a locally trivial ${\mathbb C}$-fibration over
$X^0$, isomorphic to the total space of $T_{\mathcal F}$ over $X^0$,
by an isomorphism sending $p_{\mathcal F}(X^0)$ to the null section.
\end{theorem}

\begin{proof}
By the discussion above (local triviality of $U_{\mathcal F}$ along
the leaves), the first statement is equivalent to say that, if
$T\subset X^0$ is a small transversal (say, isomorphic to ${\mathbb
D}^{n-1}$), then $U_T\simeq T\times{\mathbb C}$.

We use for this Theorem \ref{nishino} of Section 2. We may assume
that there exists an embedding $T\times{\mathbb D}\buildrel
j\over\to U_T$ sending fibers to fibers and $T\times\{ 0\}$ to
$p_T(T)$. Then we set
$$U_T^\varepsilon = U_T\setminus\{j(T\times\overline{{\mathbb D}
(\varepsilon )})\} .$$ We need to prove that the fiberwise
Poincar\'e metric on $U_T^\varepsilon$ has a plurisubharmonic
variation, for every $\varepsilon > 0$ small.

But this follows from Theorem \ref{convexity} in exactly the same
way as we did in Proposition \ref{psh} of Section 6. We just
replace, in that proof, the open subsets $\Omega_j\subset U_S$ (for
$S\subset T$ a generic disc) with
$$\Omega_j^\varepsilon = \Omega_j\setminus \{
j(S\times\overline{{\mathbb D}(\varepsilon )})\} .$$ Then the
fibration $\Omega_j^\varepsilon\to S$ is, for $j$ large, a fibration
by annuli, and its boundary in $U_S$ has two components: one is the
Levi flat $M_j$, and the other one is the Levi-flat
$j(S\times\partial {\mathbb D}(\varepsilon ))$. Then Theorem
\ref{kizuka} of Section 2, or more simply the annular generalization
of Proposition \ref{yamaguchi}, gives the desired plurisubharmonic
variation on $\Omega_j^\varepsilon$, and then on $U_S^\varepsilon$
by passing to the limit and finally on $U_T^\varepsilon$.

Hence $U_T\simeq T\times{\mathbb C}$ and $U_{\mathcal F}$ is a
locally trivial ${\mathbb C}$-fibration over $X^0$.

Let us now define explicitely the isomorphism between $U_{\mathcal
F}$ and the total space $E_{\mathcal F}$ of $T_{\mathcal F}$ over
$X^0$.

Take $p\in X^0$ and let $v_p\in E_{\mathcal F}$ be a point over $p$.
Then $v_p$ is a tangent vector to $L_p$ at $p$, and it can be lifted
to $\widetilde{L_p}$ as a tangent vector $\widetilde{v_p}$ at $p$.
Suppose $\widetilde{v_p}\not= 0$. Then, because
$\widetilde{L_p}\simeq{\mathbb C}$, $\widetilde{v_p}$ can be
extended, in a uniquely defined way, to a complete holomorphic and
nowhere vanishing vector field $\widetilde{v}$ on $\widetilde{L_p}$.
Take $q\in\widetilde{L_p}$ equal to the image of $p$ by the time-one
flow of $\widetilde{v}$, and take $q=p$ if $\widetilde{v_p}=0$. We
have in this way defined a map $(E_{\mathcal F})_p\to
\widetilde{L_p}$, $v_p\mapsto q$, which obviously is an isomorphism,
sending the origin of $(E_{\mathcal F})_p$ to the basepoint of
$\widetilde{L_p}$. In other words: because $L_p$ is parabolic, we
have a canonically defined isomorphism between $(T_pL_p,0)$ and
$(\widetilde{L_p}, p)$.

By varying $p$ in $X^0$ we thus have a map
$$U_{\mathcal F}\rightarrow E_{\mathcal F}\vert_{X^0}$$
sending $p_{\mathcal F}(X^0)$ to the null section, and we need to
verify that this map is {\it holomorphic}. This follows from the
fact that $U_{\mathcal F}$ (and $E_{\mathcal F}$ also, of course) is
a locally trivial fibration. In terms of the previous construction,
we take a local transversal $T\subset X^0$ and a nowhere vanishing
holomorphic section $v_p$, $p\in T$, of $E_{\mathcal F}$ over $T$.
The previous construction gives a vertical vector field
$\widetilde{v}$ on $U_T$, which is, on every fiber, complete
holomorphic and nowhere vanishing, and moreover it is holomorphic
along $p_T(T)\subset U_T$. After a trivialization $U_T\simeq
T\times{\mathbb C}$, sending $p_T(T)$ to $\{ w=0\}$, this vertical
vector field $\widetilde{v}$ becomes something like
$F(z,w)\frac{\partial}{\partial w}$, with $F$ nowhere vanishing,
$F(z,\cdot )$ holomorphic for every fixed $z$, and $F(\cdot ,0)$
also holomorphic. The completeness on fibers gives that $F$ is in
fact {\it constant} on every fiber, i.e. $F=F(z)$, and so $F$ is in
fact fully holomorphic. Thus $\widetilde{v}$ is fully holomorphic on
the tube. This means precisely that the above map $U_{\mathcal F}\to
E_{\mathcal F}\vert_{X^0}$ is holomorphic.
\end{proof}

\begin{example}\label{complete} {\rm
Consider a foliation ${\mathcal F}$ generated by a global
holomorphic vector field $v\in \Theta (X)$, vanishing precisely on
$Sing({\mathcal F})$. This means that $T_{\mathcal F}$ is the
trivial line bundle, and $E_{\mathcal F} = X\times{\mathbb C}$. The
compactness of $X$ permits to define the {\it flow} of $v$
$$\Phi : X\times{\mathbb C} \rightarrow X$$
which sends $\{ p\}\times{\mathbb C}$ to the orbit of $v$ through
$p$, that is to $L_p^0$ if $p\in X^0$ or to $\{ p\}$ if $p\in
Sing({\mathcal F})$. Recalling that $L_p=L_p^0$ for a generic leaf,
and observing that every $L_p^0$ is obviously parabolic, we see that
${\mathcal F}$ is a parabolic foliation. It is also not difficult to
see that, in fact, $L_p=L_p^0$ for every leaf, i.e. there are no
vanishing ends, and so the map
$$\Pi_{\mathcal F} : U_{\mathcal F} \rightarrow X^0$$ is everywhere
holomorphic, with values in $X^0$. We have $U_{\mathcal F} =
X^0\times{\mathbb C}$ (by Theorem \ref{trivialtube}, which is
however quite trivial in this special case), and the map
$\Pi_{\mathcal F} : X^0\times{\mathbb C}\to X^0$ can be identified
with the restricted flow $\Phi : X^0\times{\mathbb C}\to X^0$.}
\end{example}

\begin{remark} \label{canonicalnef} {\rm
It is a general fact \cite{Br3} that vanishing ends of a foliation
${\mathcal F}$ produce rational curves in $X$ over which the
canonical bundle $K_{\mathcal F}$ has negative degree. In
particular, if $K_{\mathcal F}$ is algebraically nef (i.e.
$K_{\mathcal F}\cdot C \ge 0$ for every compact curve $C\subset X$)
then ${\mathcal F}$ has no vanishing end.}
\end{remark}

\subsection{Foliations by rational curves}

We shall say that a foliation by curves ${\mathcal F}$ is a {\bf
foliation by rational curves} if for every $p\in X^0$ there exists a
rational curve $R_p\subset X$ passing through $p$ and tangent to
${\mathcal F}$. This class of foliations should not be confused with
the smaller class of rational quasi-fibrations: certainly a rational
quasi-fibration is a foliation by rational curves, but the converse
is in general false, because the above rational curves $R_p$ can
pass through $Sing({\mathcal F})$ and so $L_p$ (which is equal to
$R_p$ minus those points of $R_p\cap Sing({\mathcal F})$ not
corresponding to vanishing ends) can be parabolic or even
hyperbolic. Thus the class of foliations by rational curves is
transversal to our fundamental trichotomy rational quasi-fibrations
/ parabolic foliations / hyperbolic foliations.

A typical example is the radial foliation in the projective space
${\mathbb C}P^n$, i.e. the foliation generated (in an affine chart)
by the radial vector field $\sum z_j\frac{\partial}{\partial z_j}$:
it is a foliation by rational curves, but it is parabolic. On the
other hand, it is a standard exercise in bimeromorphic geometry to
see that any foliation by rational curves can be transformed, by a
bimeromorphic map, into a rational quasi-fibration. For instance,
the radial foliation above can be transformed into a rational
quasi-fibration, and even into a ${\mathbb P}$-bundle, by blowing-up
the origin.

We have seen in Section 6 that the canonical bundle $K_{\mathcal F}$
of a hyperbolic foliation is always pseudoeffective. At the opposite
side, for a rational quasi-fibration $K_{\mathcal F}$ is never
pseudoeffective: its degree on a generic leaf (a smooth rational
curve disjoint from $Sing({\mathcal F})$) is equal to $-2$, and this
prevents pseudoeffectivity. For parabolic foliations, the situation
is mixed: the radial foliation in ${\mathbb C}P^n$ has canonical
bundle equal to ${\mathcal O}(-1)$, which is not pseudoeffective; a
foliation like in Example \ref{complete} has trivial canonical
bundle, which is pseudoeffective. One can also easily find examples
of parabolic foliations with ample canonical bundle, for instance
most foliations arising from complete polynomial vector fields in
${\mathbb C}^n$ \cite{Br4}.

The following result, which combines Theorem \ref{trivialtube} and
Theorem \ref{extension}, shows that most parabolic foliations have
pseudoeffective canonical bundle.

\begin{theorem} \label{parabolic}
Let ${\mathcal F}$ be a parabolic foliation on a compact connected
K\"ahler manifold $X$. Suppose that its canonical bundle
$K_{\mathcal F}$ is not pseudoeffective. Then ${\mathcal F}$ is a
foliation by rational curves.
\end{theorem}

\begin{proof}
Consider the meromorphic map
$$\Pi_{\mathcal F}: E_{\mathcal F}\vert_{X^0}\simeq U_{\mathcal F}
\dashrightarrow X$$ given by Theorem \ref{trivialtube}. Because
$Sing({\mathcal F}) = X\setminus X^0$ has codimension at least two,
such a map meromorphically extends \cite{Siu} to the full space
$E_{\mathcal F}$:
$$\Pi_{\mathcal F} : E_{\mathcal F}\dashrightarrow X .$$

The section at infinity of $E_{\mathcal F}$ is the same as the null
section of $E_{\mathcal F}^*$, the total space of $K_{\mathcal F}$.
If $K_{\mathcal F}$ is not pseudoeffective, then by Theorem
\ref{extension} $\Pi_{\mathcal F}$ extends to the full
$\overline{E_{\mathcal F}}= E_{\mathcal F}\cup\{$ section at
$\infty\}$, as a meromorphic map
$$\overline{\Pi_{\mathcal F}} : \overline{E_{\mathcal F}}
\dashrightarrow X .$$

By construction, $\overline{\Pi_{\mathcal F}}$ sends the rational
fibers of $\overline{E_{\mathcal F}}$ to rational curves in $X$
tangent to ${\mathcal F}$, which is therefore a foliation by
rational curves.
\end{proof}

Note that the converse to this theorem is not always true: for
instance, a parabolic foliation like in Example \ref{complete} has
trivial (pseudoeffective) canonical bundle, yet it can be a
foliation by rational curves, for some special $v$.

We may resume the various inclusions of the various classes in the
diagram below.

\begin{figure}[h]
\includegraphics[width=10cm,height=6cm]{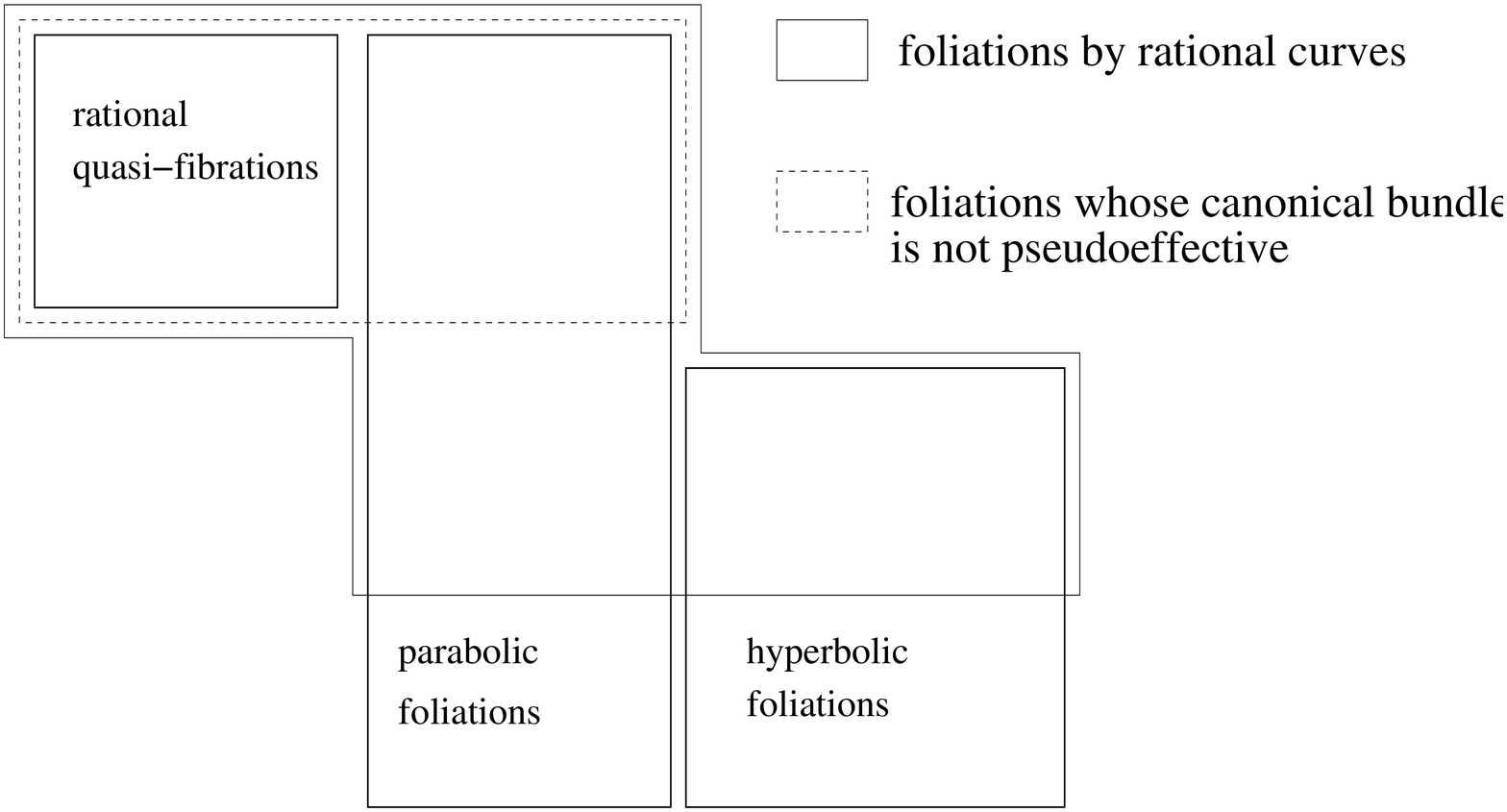}
\end{figure}

Let us discuss the classical case of fibrations.

\begin{example} \label{fibrations} {\rm
Suppose that ${\mathcal F}$ is a fibration over some base $B$, i.e.
there exists a holomorphic map $f:X \to B$ whose generic fiber is a
leaf of ${\mathcal F}$ (but there may be singular fibers, and even
some higher dimensional fibers). Let $g$ be the genus of a generic
fiber, and suppose that $g\ge 1$. The relative canonical bundle of
$f$ is defined as
$$K_f = K_X\otimes f^*(K_B^{-1}) .$$
It is related to the canonical bundle $K_{\mathcal F}$ of ${\mathcal
F}$ by the relation
$$K_f = K_{\mathcal F}\otimes {\mathcal O}_X(D)$$
where $D$ is an {\it effective} divisor which takes into account the
possible ramifications of $f$ along nongeneric fibers. Indeed, by
adjunction along the leaves, we have $K_X = K_{\mathcal F}\otimes
N_{\mathcal F}^*$, where $N_{\mathcal F}^*$ denotes the determinant
conormal bundle of ${\mathcal F}$. If $\omega$ is a local generator
of $K_B$, then $f^*(\omega )$ is a local section of $N_{\mathcal
F}^*$ which vanishes along the ramification divisor $D$ of $f$,
hence $f^*(K_B)=N_{\mathcal F}^*\otimes {\mathcal O}_X(-D)$, whence
the relation above.

Because ${\mathcal F}$ is not a foliation by rational curves, we
have, by the Theorems above, that $K_{\mathcal F}$ is
pseudoeffective, and therefore also $K_f$ is pseudoeffective. In
particular, $f_*(K_{\mathcal F})$ and $f_*(K_f)$ are
``pseudoeffective sheaves'' on $B$, in the sense that their degrees
with respect to K\"ahler metrics on $B$ are nonnegative. This must
be compared with Arakelov's positivity theorem \cite{Ara} \cite[Ch.
III]{BPV}. But, as in Arakelov's results, something more can be
said. Suppose that $B$ is a curve (or restrict the fibration $f$
over some curve in $B$) and let us distinguish between the
hyperbolic and the parabolic case.

$\bullet$ $g\ge 2$. Then the pseudoeffectivity of $K_{\mathcal F}$
is realized by the leafwise Poincar\'e metric (Theorem
\ref{hyperbolic}). A subtle computation \cite{Br2} \cite{Br1} shows
that this leafwise (or fiberwise) Poincar\'e metric has a {\it
strictly} plurisubharmonic variation, unless the fibration is
isotrivial. This means that if $f$ is {\it not} isotrivial then the
degree of $f_*(K_{\mathcal F})$ (and, a fortiori, the degree of
$f_*(K_f)$) is {\it strictly positive}.

$\bullet$ $g=1$. We put on every smooth elliptic leaf of ${\mathcal
F}$ the (unique) flat metric with total area 1. It is shown in
\cite{Br4} (using Theorem \ref{trivialtube} above) that this
leafwise metric extends to a metric on $K_{\mathcal F}$ with
positive curvature. In other words, the pseudoeffectivity of
$K_{\mathcal F}$ is realized by a leafwise flat metric. Moreover,
still in \cite{Br4} it is observed that if the fibration is not
isotrivial then the curvature of such a metric on $K_{\mathcal F}$
is strictly positive on directions transverse to the fibration. We
thus get the same conclusion as in the hyperbolic case: if $f$ is
{\it not} isotrivial then the degree of $f_*(K_{\mathcal F})$ (and,
a fortiori, the degree of $f_*(K_f)$) is {\it strictly positive}.}
\end{example}

Let us conclude with several remarks around the pseudoeffectivity of
$K_{\mathcal F}$.

\begin{remark} {\rm
In the case of hyperbolic foliations, Theorem \ref{hyperbolic} is
very efficient: not only $K_{\mathcal F}$ is pseudoeffective, but
even this pseudoeffectivity is realized by an explicit metric,
induced by the leafwise Poincar\'e metric. This gives further useful
properties. For instance, we have seen that the polar set of the
metric is filled by singularities and parabolic leaves. Hence, for
example, if all the leaves are hyperbolic and the singularities are
isolated, then $K_{\mathcal F}$ is not only pseudoeffective but even
nef (numerically eventually free \cite{Dem}). This efficiency is
unfortunately lost in the case of parabolic foliations, because in
Theorem \ref{parabolic} the pseudoeffectivity of $K_{\mathcal F}$ is
obtained via a more abstract argument. In particular, we do not know
how to control the polar set of the metric. See, however, \cite{Br4}
for some special cases in which a distinguished metric on
$K_{\mathcal F}$ can be constructed even in the parabolic case,
besides the case of elliptic fibrations discussed in Example
\ref{fibrations} above.}
\end{remark}

\begin{remark} {\rm
According to general principles \cite{BDP}, once we know that
$K_{\mathcal F}$ is pseudoeffective we should try to understand its
discrepancy from being nef. There is on $X$ a unique maximal
countable collection of compact analytic subsets $\{ Y_j\}$ such
that $K_{\mathcal F}\vert_{Y_j}$ is {\it not} pseudoeffective. It
seems reasonable to try to develop the above theory in a
``relative'' context, by replacing $X$ with $Y_j$, and then to prove
something like this: every $Y_j$ is ${\mathcal F}$-invariant, and
the restriction of ${\mathcal F}$ to $Y_j$ is a foliation by
rational curves. Note, however, that the restriction of a foliation
to an invariant analytic subspace $Y$ is a dangerous operation.
Usually, we like to work with ``saturated'' foliations, i.e. with a
singular set of codimension at least two (see, e.g., the beginning
of the proof of Theorem \ref{parabolic} for the usefulness of this
condition). If $Z=Sing({\mathcal F})\cap Y$ has codimension one in
$Y$, this means that our ``restriction of ${\mathcal F}$ to $Y$'' is
not really ${\mathcal F}\vert_Y$, but rather its saturation.
Consequently, the canonical bundle of that restriction is not really
$K_{\mathcal F}\vert_Y$, but rather $K_{\mathcal
F}\vert_Y\otimes{\mathcal O}_Y(-{\mathcal Z})$, where ${\mathcal Z}$
is an effective divisor supported in $Z$. If $Z=Sing({\mathcal
F})\cap Y$ has codimension zero in $Y$ (i.e., $Y\subset
Sing({\mathcal F})$), the situation is even worst, because then
there is not a really well defined notion of restriction to $Y$.}
\end{remark}

\begin{remark} {\rm
The previous remark is evidently related to the problem of
constructing minimal models of foliations by curves, i.e. birational
models (on possibly singular varieties) for which the canonical
bundle is nef. In the projective context, results in this direction
have been obtained by McQuillan and Bogomolov \cite{BMQ} \cite{MQ2}.
From this birational point of view, however, we rapidly meet another
open and difficult problem: the resolution of the singularities of
${\mathcal F}$. A related problem is the construction of birational
models for which there are no vanishing ends in the leaves, compare
with Remark \ref{canonicalnef} above.}
\end{remark}

\begin{remark} {\rm
Finally, the pseudoeffectivity of $K_{\mathcal F}$ may be measured
by finer invariants, like Kodaira dimension or numerical Kodaira
dimension. When $\dim X=2$ then the picture is rather clear and
complete \cite{MQ1} \cite{Br1}. When $\dim X > 2$ then almost
everything seems open (see, however, the case of fibrations
discussed above). Note that already in dimension two the so-called
``abundance'' does not hold: there are foliations (Hilbert Modular
Foliations \cite{MQ1} \cite{Br1}) whose canonical bundle is
pseudoeffective, yet its Kodaira dimension is $-\infty$. The
classification of these exceptional foliations was the first
motivation for the plurisubharmonicity result of \cite{Br2}.}
\end{remark}


\begin{thebibliography}{100}
\bibitem[AWe]{AWe} H. Alexander, J. Wermer, {\sl Several complex
variables and Banach algebras}, Graduate Texts in Mathematics 35,
Springer Verlag (1998)
\bibitem[Ara]{Ara} S. Ju. Arakelov, {\sl Families of algebraic
curves with fixed degeneracies}, Izv. Akad. Nauk SSSR 35 (1971),
1269--1293
\bibitem[BPV]{BPV} W. Barth, C. Peters, A. Van de Ven, {\sl Compact
complex surfaces},  Ergebnisse der Mathematik (3) 4, Springer Verlag
(1984)
\bibitem[BeG]{BeG} E. Bedford, B. Gaveau, {\sl Envelopes of
holomorphy of certain 2-spheres in ${\mathbb C}^2$},  Amer. J. Math.
105 (1983), 975--1009
\bibitem[BeT]{BeT} E. Bedford, B. A. Taylor, {\sl The Dirichlet
problem for a complex Monge-Amp\`ere equation},  Invent. Math. 37
(1976), 1--44
\bibitem[Bis]{Bis} E. Bishop, {\sl Conditions for the analyticity of
certain sets}, Michigan Math. J. 11 (1964), 289--304
\bibitem[BMQ]{BMQ} F. A. Bogomolov, M. McQuillan, {\sl Rational
curves on foliated varieties}, Preprint IHES (2001)
\bibitem[BDP]{BDP} S. Boucksom, J.-P. Demailly, M. Paun, Th.
Peternell, {\sl The pseudo-effective cone of a compact K\"ahler
manifold and varieties of negative Kodaira dimension}, Preprint
(2004)
\bibitem[Br1]{Br1} M. Brunella, {\sl Foliations on complex projective
surfaces}, in Dynamical systems, Pubbl. Cent. Ric. Mat. Ennio De
Giorgi, SNS Pisa (2003), 49--77
\bibitem[Br2]{Br2} M. Brunella, {\sl Subharmonic variation of the
leafwise Poincar\'e metric}, Invent. Math. 152 (2003), 119--148
\bibitem[Br3]{Br3} M. Brunella, {\sl Plurisubharmonic variation
of the leafwise Poincar\'e metric}, Internat. J. Math. 14 (2003),
139--151
\bibitem[Br4]{Br4} M. Brunella, {\sl Some remarks on parabolic
foliations}, in Geometry and dynamics, Contemp. Math. 389 (2005),
91--102
\bibitem[Br5]{Br5} M. Brunella, {\sl A positivity property for
foliations on compact K\"ahler manifolds}, Internat. J. Math. 17
(2006), 35--43
\bibitem[Br6]{Br6} M. Brunella, {\sl On the plurisubharmonicity
of the leafwise Poincar\'e metric on projective manifolds}, J. Math.
Kyoto Univ. 45 (2005), 381--390
\bibitem[Br7]{Br7} M. Brunella, {\sl On entire curves tangent to a
foliation}, J. Math. Kyoto Univ. 47 (2007), 717--734
\bibitem[CLN]{CLN} C. Camacho, A. Lins Neto, {\sl Geometric theory
of foliations}, Birkh\"auser (1985)
\bibitem[CaP]{CaP} F. Campana, Th. Peternell, {\sl Cycle spaces},
in Several complex variables VII, Encyclopaedia Math. Sci. 74,
Springer Verlag (1994), 319--349
\bibitem[Che]{Che} K. T. Chen, {\sl Iterated path integrals}, Bull.
Amer. Math. Soc. 83 (1977), 831--879
\bibitem[Chi]{Chi} E. M. Chirka, {\sl Complex analytic sets},
Mathematics and its Applications 46, Kluwer (1989)
\bibitem[ChI]{ChI} E. M. Chirka, S. Ivashkovich, {\sl On
nonimbeddability of Hartogs figures into complex manifolds},  Bull.
Soc. Math. France  134 (2006), 261--267
\bibitem[Dem]{Dem} J.-P. Demailly, {\sl $L^2$ vanishing theorems for
positive line bundles and adjunction theory}, in Transcendental
methods in algebraic geometry (Cetraro, 1994), Lecture Notes in
Math. 1646 (1996), 1--97
\bibitem[Din]{Din} P. Dingoyan, {\sl Monge-Amp\`ere currents over
pseudoconcave spaces}, Math. Ann. 320 (2001), 211--238
\bibitem[For]{For} F. Forstneric, {\sl Polynomial hulls of sets
fibered over the circle}, Indiana Univ. Math. J. 37 (1988), 869--889
\bibitem[Ghy]{Ghy} E. Ghys, {\sl Laminations par surfaces de Riemann},
in Dynamique et g\'eom\'etrie complexes, Panor. Synth\`eses 8
(1999), 49--95
\bibitem[GuR]{GuR} R. C. Gunning, H. Rossi, {\sl Analytic functions of
several complex variables}, Prentice-Hall (1965)
\bibitem[Il1]{Il1} Ju. S. Il'yashenko, {\sl Foliations by analytic
curves}, Mat. Sb. (N.S.) 88 (130) (1972), 558--577
\bibitem[Il2]{Il2} Ju. S. Il'yashenko, {\sl Covering manifolds for
analytic families of leaves of foliations by analytic curves},
Topol. Methods Nonlinear Anal. 11 (1998), 361--373 (addendum: 23
(2004), 377--381)
\bibitem[Iv1]{Iv1} S. Ivashkovich, {\sl The Hartogs-type extension
theorem for meromorphic maps into compact K\"ahler manifolds},
Invent. Math. 109 (1992), 47--54
\bibitem[Iv2]{Iv2} S. Ivashkovich, {\sl Extension properties of
meromorphic mappings with values in non-K\"ahler complex manifolds},
Ann. of Math. 160 (2004), 795--837
\bibitem[IvS]{IvS} S. Ivashkovich, V. Shevchishin, {\sl Structure
of the moduli space in a neighborhood of a cusp-curve and
meromorphic hulls}, Invent. Math.  136 (1999), 571--602
\bibitem[Kiz]{Kiz} T. Kizuka, {\sl On the movement of the Poincar\'e
metric with the pseudoconvex deformation of open Riemann surfaces},
Ann. Acad. Sci. Fenn.  20 (1995), 327--331
\bibitem[Kli]{Kli} M. Klimek, {\sl Pluripotential theory},
London Mathematical Society Monographs 6, Oxford University Press
(1991)
\bibitem[Lan]{Lan} S. Lang, {\sl Introduction to complex hyperbolic
spaces}, Springer Verlag (1987)
\bibitem[MQ1]{MQ1} M. McQuillan, {\sl Non-commutative Mori theory},
Preprint IHES (2001)
\bibitem[MQ2]{MQ2} M. McQuillan, {\sl Semi-stable reduction of
foliations}, Preprint IHES (2005)
\bibitem[Miy]{Miy} Y. Miyaoka, {\sl Deformations of a morphism
along a foliation and applications}, in Algebraic geometry (Bowdoin
1985), Proc. Sympos. Pure Math. 46 (1987), 245--268
\bibitem[Moi]{Moi} B. G. Moishezon, {\sl On $n$-dimensional compact
varieties with $n$ algebraically independent meromorphic functions},
Amer. Math. Soc. Transl. 63 (1967), 51--177
\bibitem[Nap]{Nap} T. Napier, {\sl Convexity properties of coverings
of smooth projective varieties}, Math. Ann. 286 (1990), 433--479
\bibitem[Nis]{Nis} T. Nishino, {\sl Nouvelles recherches sur les
fonctions enti\`eres de plusieurs variables complexes I-V}, J. Math.
Kyoto Univ. 8, 9, 10, 13, 15 (1968-1975)
\bibitem[Ohs]{Ohs} T. Ohsawa, {\sl A note on the variation of
Riemann surfaces}, Nagoya Math. J. 142 (1996), 1--4
\bibitem[Ran]{Ran} R. M. Range, {\sl Holomorphic functions and
integral representations in several complex variables}, Graduate
Texts in Mathematics 108, Springer Verlag (1986)
\bibitem[ShB]{ShB} N. Shepherd-Barron, {\sl Miyaoka's theorems on
the generic seminegativity of $T_X$ and on the Kodaira dimension of
minimal regular threefolds}, in Flips and abundance for algebraic
threefolds, Ast\'erisque 211 (1992), 103--114
\bibitem[Siu]{Siu} Y. T. Siu, {\sl Techniques of extension of
analytic objects}, Lecture Notes in Pure and Applied Mathematics 8,
Marcel Dekker (1974)
\bibitem[Suz]{Suz} M. Suzuki, {\sl Sur les int\'egrales premi\`eres
de certains feuilletages analytiques complexes}, in Fonctions de
plusieurs variables complexes III (S\'em. Norguet 1975-77), Lecture
Notes in Math. 670 (1978) 53--79
\bibitem[Ya1]{Ya1} H. Yamaguchi, {\sl Sur le mouvement des constantes
de Robin}, J. Math. Kyoto Univ. 15 (1975), 53--71
\bibitem[Ya2]{Ya2} H. Yamaguchi, {\sl Parabolicit\'e d'une fonction
enti\`ere}, J. Math. Kyoto Univ. 16 (1976), 71--92
\bibitem[Ya3]{Ya3} H. Yamaguchi, {\sl Calcul des variations
analytiques}, Japan. J. Math. (N.S.) 7 (1981), 319--377
\end{thebibliography}
\end{document}